\documentclass[11pt]{amsart}

\usepackage{palatino} 

\usepackage{amsmath,amssymb,amsthm,amsfonts,amscd,flafter,epsf, epsfig,graphicx,verbatim,pinlabel,mathrsfs}
\usepackage[all]{xy}
\usepackage{epsf}
\usepackage[abs]{overpic}
\usepackage{epstopdf}
\usepackage{color}
\usepackage{mathtools}

\definecolor{red}{rgb}{1,0,0}

\def\tt{\overline{t}}
\def\SS{\mathcal{S}}
\def\ZZ{\mathcal{Z}}
\def\XX{\mathcal{X}}
\def\Sat{\mathit{Sat}}

\newcommand\Trans{\mathit{Trans}}
\newcommand\Leg{\mathit{Leg}}
\newcommand\Lc{\mathcal{L}}
\newcommand\Q{\mathcal{Q}}
\newcommand\CC{\mathcal{C}}
\def\RR{\mathcal{R}}
\def\R{{\mathbb{R}}}
\newcommand\mer{\mathbf{m}}
\newcommand\lon{\mathbf{l}}
\newcommand\K{\mathcal{K}}
\newcommand\Pat{\mathcal{P}}
\newcommand\U{\mathcal{U}}
\newcommand\W{\mathcal{W}}
\newcommand\B{\mathcal{B}}
\newcommand\St{{\mathit {St}}}

\newcommand\tb{{\mathit {tb}}}
\newcommand\self{{\mathit {sl}}}
\newcommand\reltb{{\mathit {reltb}}}
\newcommand\relsl{{\mathit {relsl}}}
\newcommand\relrot{{\mathit {relrot}}}

\newcommand\Rold{{D\times I}}

\newcommand\rot{{\mathit{rot}}}

\def\co{\colon\thinspace}

\newcommand\Xkl{X(k,l)}
\newcommand\Skl{S(k,l)}
\newcommand\Zkl{Z(k,l)}

\def\Z{{\mathbb{Z}}}

\def\dfn#1{{\em #1}}
\def\co{\colon\thinspace}

\newtheorem{theorem}{Theorem}[section]
\newtheorem{lemma}[theorem]{Lemma}

\newtheorem{corollary}[theorem]{Corollary}

\theoremstyle{definition}
\newtheorem{definition}[theorem]{Definition}

\newtheorem{remark}[theorem]{Remark}

\newtheorem{example}[theorem]{Example}

\title{Legendrian satellites}

\author[John Etnyre]{John Etnyre}
\address{Department of Mathematics \\ Georgia Institute of Technology}
\email{etnyre@math.gatech.edu}

\author[Vera V\'ertesi]{Vera V\'ertesi}
\address{IRMA \\ Universit\'e de Strasbourg}
\email{vertesi@unistra.fr}

\begin{document}
\begin{abstract}
In this paper we study Legendrian knots in the knot types of satellite knots. In particular, we classify Legendrian Whitehead patterns and learn a great deal about Legendrian braided patterns. We also show how the classification of Legendrian patterns can lead to a classification of the associated satellite knots if the companion knot is Legendrian simple and uniformly thick. This leads to new Legendrian and transverse classification results for knots in the 3-sphere with its standard contact structure as well as a more general perspective on some previous classification results. 
\end{abstract}
\maketitle


\section{Introduction}
\label{sec:introduction}

In \cite{EtnyreHonda03} Honda and the first authors proved the first ``structure theorem" about Legendrian knots by showing how the classification of Legendrian representatives of a knot type behave under the topological operation of connected sum and used this structure theorem to show several new qualitative features of Legendrian knots. We call this a structure theorem since it shows the structure of Legendrian knots under a general topological construction, even if the actual classification of Legendrian knots is not known. This paper concerns another such structure theorem. Specifically we will consider the behavior of Legendrian knots under the satellite operation and several associated results. 

We begin by establishing some notation. Throughout this paper, when not stated otherwise, we will be considering Legendrian knots in the standard contact structure $\xi_{std}$ on $S^3$ (or equivalently the standard structure on $\R^3$ so that we may draw front diagrams to represent our knots). It will be important at times to distinguish a specific knot from its knot type, that is the isotopy equivalence class determined by a knot. We will use calligraphic letters, such as $\K$ and $\Lc$, to denote a knot type (smooth or Legendrian depending on context) and roman letters, such as $K$ and $L$, to denote specific knots. So the notation $K\in \K$ indicates that $K$ is a representative of the knot type $\K$. Given a smooth knot type $\K$ we will denote the set of Legendrian knots realizing this knot type by $\Leg(\K)$ and the set of transverse knots realizing this knot type by $\Trans(\K)$.  Similarly if integers $t$ and $r$ are given then $\Leg(\K;t,r)$ denotes the subset of $\Leg(\K)$ containing Legendrian knots with Thurston-Bennequin invariant $t$ and rotation number $r$. (If only one integer is given, $\Leg(\K;t)$, it will specify the Thurston-Bennequin invariant.)

Now consider a smooth knot type $\Pat$ in $V=D^2\times S^1$, which we will call a {\em pattern}, and  a smooth knot type $\K$ in $S^3$, that we will call the {\em companion knot}. From this data we can fix an identification of $V$ with a neighborhood of a representative $K$ of $\K$ (this depends on a framing of $\K$, see Section~\ref{subsec:satknot} for a more precise definition but here we assume that the Seifert framing on $\K$ is used to make this identification) and consider the image of a representative $P$ of $\Pat$ under this identification. This gives a new knot $P(K)$ in $S^3$  called the {\em satellite of $K$ with pattern $P$}. We denote the resulting knot type $\Pat(\K)$ and as the notation suggests, one can think of a pattern as giving a function on the set of knot types. 

Turning back to contact geometry, recall that (the interior of) $V$ can be thought of as the 1-jet space of $S^1$ and as such has a standard contact structure $\xi_V=\ker(dz-y\, d\theta)$, where $(y,z)$ are coordinates on $\R^2$ and $\theta$ is the angular coordinate on $S^1$. Projecting out the $y$ coordinate is called the front projection, and one may easily see that front projections in $V$ have many of the same properties of front projections in $\R^3$. See Figure~\ref{fig:patex}. 
\begin{figure}[h]
\small
\begin{overpic}
{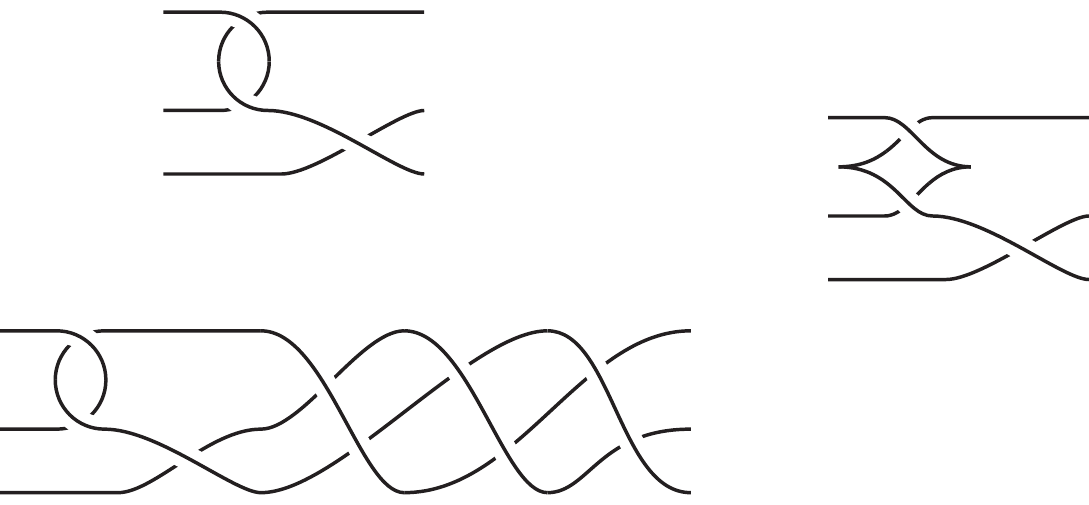}
\put(25,120){$\Pat$}
\put(70,60){$\Delta\Pat$}
\put(222,85){$\Q$}
\end{overpic}
\caption{In the upper left is a pattern $\Pat$. In all the diagrams identify the right and left hand sides to obtain a knot in $D^2\times S^1$. In the lower left is $\Delta\Pat$ and on the right is a Legendrian knot type $\Q$ realizing the smooth knot type $\Pat$.}
\label{fig:patex}
\end{figure}
It is known that any Legendrian knot $L$ has a neighborhood $\nu(L)$ contactomorphic to $(V,\xi_V)$, see Section~\ref{sec:braidsatellite} for more details. Now given a Legendrian knot $Q$ in $V$ representing a pattern $\Pat$ and a Legendrian knot $L$ in $S^3$ then we denote by $Q(L)$ the image of $Q$ under the above contactomorphism. This operation is well-defined on Legendrian isotopy classes and is called the {\em Legendrian satellite operation}. It is important to notice that the contactomorphism used in this definition takes the product framing on $V$ to the contact framing on the neighborhood of $L$ and not the Seifert framing. Thus if the underlying smooth knots types of $L$ and $Q$ are $K$ and $P$, respectively, then $Q(L)$ is not in the smooth knot type of the smooth satellite $P(K)$ defined above, but in the knot type of $(\Delta^{tb(L)}P)(K)$, where $tb(L)$ is the Thurston-Bennequin invariant of $L$ (that is the contact framing of $L$ relative to the Seifert framing) and $\Delta$ is the result of applying a full right handed twist to $P$. See Figure~\ref{fig:patex}. A front diagram of a Legendrian satellite is shown in Figure~\ref{fig:satexample}. The front diagram is created by taking enough copies of the front diagram for the companion Legendrian knot, translated in the vertical direction, so that the front diagram for the pattern can be inserted in some portion of the diagram as shown in the figure. See Section~\ref{subsec:legsatellites} for more details on the construction. 
\begin{figure}[h]
\small
\begin{overpic}
{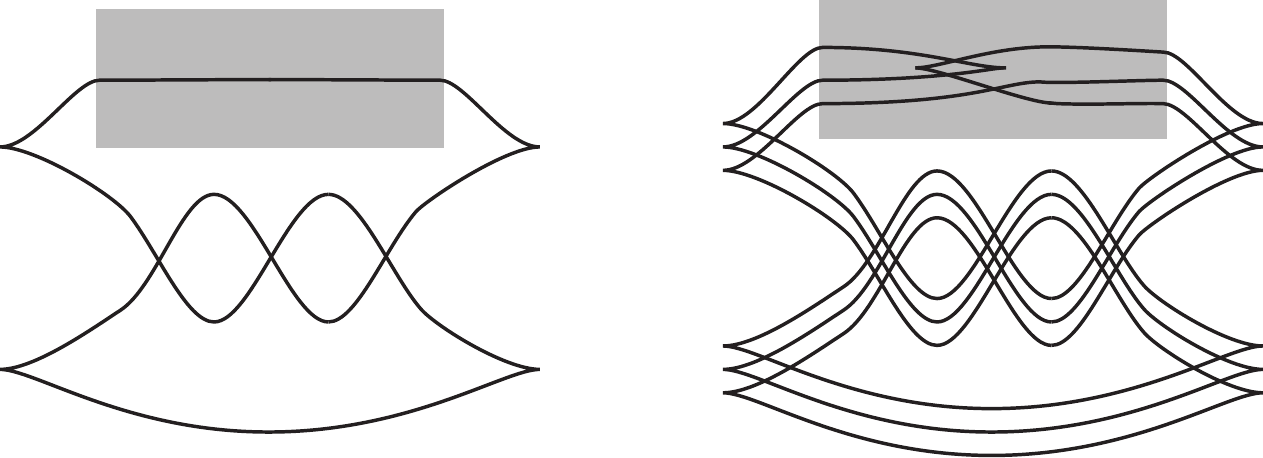}
\end{overpic}
\caption{A Legendrian right handed trefoil $\Lc$ is shown on the left. On the right is the Legendrian satellite $\Q(\Lc)$ for the Legendrian pattern $\Q$ shown in Figure~\ref{fig:patex}. Notice that since $tb(\Lc)=1$ the Legendrian knot $\Q(\Lc)$ represents the smooth knot type $(\Delta\Pat)(\K)$, where $\K$ is the smooth knot type of the right handed trefoil.}
\label{fig:satexample}
\end{figure}

The Legendrian satellite construction was first explicitly defined and studied in \cite{NgThesis} where it was shown to be well-defined. However various types of satellites were used prior to this work. For example $n$-copies of Legendrian knots were discussed in \cite{Michatchev01} and the famous Chekanov-Eliashberg examples that gave the first Legendrian non-simple knot types can be thought of as Legendrian Whitehead doubles of Legendrian unknots (see Section~\ref{LWD} for more on Legendrian Whitehead doubles). 

The basic structure theorem for Legendrian satellite operations would involve understanding the map 
\[
\widetilde\Sat\co \bigcup_{t\in \Z} \left(\Leg_V(\Delta^{-t}\Pat)\times \Leg(\K; t)\right)\to \Leg(\Pat(\K))\co (\Q,\Lc)\mapsto \Q(\Lc)
\]
or the slightly more tractable map
\[
\widetilde\Sat'\co \left(\Leg_V(\Delta^{-\tt}\Pat)\times \Leg(\K; \tt)\right)\to \Leg(\Pat(\K))\co (\Q,\Lc)\mapsto \Q(\Lc),
\]
where $\tt$ is the maximal Thurston-Bennequin invariant of the knot type $\K$. The obvious questions one would like to address are the following.
\begin{enumerate}
\item\label{q1} Is $\widetilde\Sat$ or $\widetilde\Sat'$ onto?
\item\label{q2} Is there an equivalence relation that may be placed on the domain to make $\widetilde\Sat$ or $\widetilde\Sat'$ injective?
\item\label{q3} Can one obtain classification results for new knot types using an understanding of $\widetilde\Sat$ or $\widetilde\Sat'$?
\end{enumerate}
Clearly to have a chance at answering the last question one needs to know the answer to the following question.
\begin{enumerate}
\setcounter{enumi}{3}
\item\label{q4} For what patterns $\Pat$ do we understand $\Leg_V(\Pat)$?
\end{enumerate}

To answer some of these questions we first recall a knot type is called {\em uniformly thick} if every solid torus whose core realizes that knot type can be contained in another such torus that is a standard neighborhood of a maximal Thurston-Bennequin invariant Legendrian representative of that knot type. In \cite{EtnyreHonda05} it was shown that negative torus knots are uniformly thick, connected sums of uniformly thick knot types are uniformly thick, and sufficiently negative cables of uniformly thick knot types are uniformly thick. On the other hand the unknot and positive torus knots are not uniformly thick.  
One of the main results of the paper is Theorem~\ref{lem:incommonstab} which we paraphrase as follows (refer to Section~\ref{sec:braidsatellite} for a precise statement). 

\begin{theorem}\label{mainsat}
Suppose that $\K$ is a uniformly thick and Legendrian simple knot type, and $\Pat$ is a pattern in $V$. Assume that $\Pat(\K)$ satisfies certain symmetry hypotheses (see Sections~\ref{subsec:satosymmetry} and~\ref{sec:braidsatellite}). Then the kernel of $\widetilde\Sat'$ is given by an explicit equivalence relation, see Definition~\ref{def:ker}, such that $\widetilde\Sat'$ induces a bijection
\[
{{\Sat'}}\co \left(\frac{\Leg_V(\Delta^{-\tt}\Pat)\times \Leg(\K;\tt)}{\sim}\right) \to \Leg(\Pat(\K)).
\]
\end{theorem}
This theorem gives an affirmative answer to Questions~(\ref{q1}) and~(\ref{q2}) in certain cases.  However, since cabling is a satellite operation, one can see that the results in \cite{EtnyreHonda05} and \cite{EtnyreLafountainTosun12} imply that neither $\widetilde\Sat$ nor $\widetilde\Sat'$ is onto when considering many positive cables of a positive torus knot. However, results in \cite{TosunPre12} imply that Theorem~\ref{mainsat} can sometimes be true even for knots that are not uniformly thick. Finding the exact generality in which the theorem can be proved would be very interesting. 

The first attempts to address Question~\ref{q4} ocured in \cite{Traynor01} and \cite{NgTraynor04}, where generating family and, respectively, contact homology invariants where used to show several subtle phenomena about Legendrian knots in $(V,\xi_V)$. 

We give several results for Legendrian knots in $(V,\xi_V)$. We call a pattern {\em braided} if it is the closure in $V$ of an element $w$ of the $n$-strand braid group $B_n$ for some $n$. In Theorem~\ref{thm:legbraids1} we prove the following result. 
\begin{theorem}\label{braidstatement}
 A Legendrian braid $B$ in $(V,\xi_V)$ is Legendrian isotopic to the closure of a concatenation of the building blocks in Figure~\ref{fig:openbraidsimpl}.
\end{theorem}
\begin{figure}[h]
\centering
\begin{overpic}
{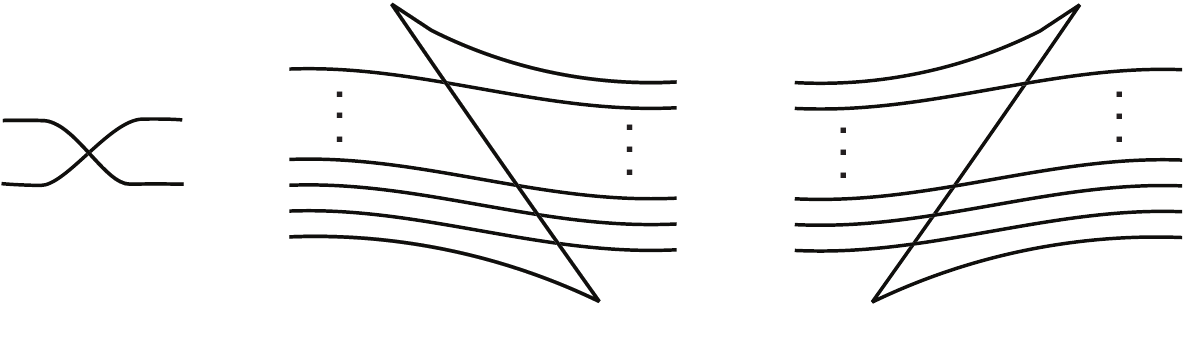}
\put(-2,23){$X=X(1,1)$}
\put(100,-2){$S=S_0(1,n-1)$}
\put(250,-2){$Z=Z_0(1,n-1)$}
\end{overpic}
\caption{Building blocks of Legendrian braids. There may be other strands both above and bellow of the pictured braids, but they are all assumed to be horizontal strands that are disjoint from the strands in the picture.}
\label{fig:openbraidsimpl}
\end{figure}
From this result in Theorem~\ref{thm:closedposlegendrianbraid} and Lemma~\ref{2starndbraid} we show:
\begin{enumerate}
\item  Let $\Pat$ be the closure of a positive braid $w$ in the $n$-strand braid group $\B_n$, then the relative Thurston-Bennequin invariant is 
$
 \overline{\reltb}_V(\Pat)= \textit{length}(w)
$
and  
\[
\vert\Leg_{V}(\Pat;\overline{\reltb}_{V}(\Pat))\vert=1.
\]
\item Let $\Pat_{m}$ be a 2--braid pattern with $m$ (odd) half twists. Then $\Pat_m$ is Legendrian simple in particular:
\begin{enumerate}
\item If $m> 0$, 
then $\Pat_m$ has a unique Legendrian representative with maximal relative Thurston--Bennequin number $m$ and rotation number $0$. \item If $m<0$, then $\Pat_m$ has $|m|+1$ representatives with maximal Thurston-Bennequin number $\reltb_{V}=-2|m|$ and with different rotation numbers 
\[
\relrot_{V}\in \{-m,-(m-2),\dots,m-2,m\}.
\]
\end{enumerate}
\end{enumerate}
In  \cite{EtnyreNgVertesi13} the proof of Theorem~\ref{braidstatement}, or more precisely Corollary~\ref{thm:legbraidsimpl}, was used to classify Legendrian twist knots. In Theorem~\ref{whiteheadpatternclass} we generalize part of that work to classify Legendrian Whitehead patterns. In addition, in Theorem~\ref{cablepats} we classify Legendrian cable patterns. 

With these results we can prove several classification results for Legendrian knots in $(S^3,\xi_{std})$. 
\begin{enumerate}
\item In Theorem~\ref{csum} we can reprove the structure theorem for connected sums from \cite{EtnyreHonda03} under the extra hypothesis that one of the summands is uniformly thick and Legendrian simple (and their are no symmetries). 
\item In Theorem~\ref{pqcableclass} we can reprove the result from \cite{EtnyreHonda05} that cables of Legendrian simple, uniformly thick knot types are Legendrian simple. 
\item In Theorem~\ref{thm:twistsat} we classify Legendrian knots in the knot types of Whitehead doubles of Legendrian simple, uniformly thick knots types. This leads to many new non-Legendrian simple and non-transversely simple knot types. See Example~\ref{whdtorus}.
\end{enumerate}
We also make several similar observations about transverse knots in satellite knot types. 

In Section~\ref{sec:preliminaries} we discuss the topological satellite construction and make several observations about the topology of the complement of satellite knots. We also discuss several features about contact structures on solid tori that will be needed in the paper. Section~\ref{sec:braidsind2xi} concerns ``open" Legendrian braids in $D^2\times I$ with and $I$-invariant contact structure. Gluing the ends of this thickened disk together gives $V$ and so these results are used in Section~\ref{sec:legendrianbraidsind2xs1} to prove our results about braided patterns. In that section we also consider Whitehead patterns. Finally in Section~\ref{sec:braidsatellite} we discuss the Legendrian satellite construction and prove our new structure theorems. 

\noindent
{\em{Acknowledgments}} The authors are very grateful for helpful conversations with Vincent Colin, Paolo Ghiggini, Lenny Ng, Dave Futer, and Andr\'as Stipsicz. Part of this work was carried out while both the authors were at the Mittag-Leffler Institut, while the first author was at the Institute for Advanced Study, and while the second author was at Universit\'e de Nantes and UC Santa Barbara. We gratefully acknowledge their support of this work. The first author was partially supported by a grant from the Simons Foundation (\#342144), The Bell Companies Fellowship Fund, and NSF grants DMS-1309073 and DMS-1608684. The second author was supported by ERC Geodycon, OTKA grant numbers 49449, 67867 and NK81203 and NSF grant number 1104690.

\section{Preliminaries}
\label{sec:preliminaries}

We assume the reader is familiar with basic knot theory and in particular braid theory as can be found in \cite{Birman74, Rolfsen76}. We also refer to \cite{EtnyreHonda01b} for the basic notions from contact geometry, Legendrian and transverse knot theory, and the use of convex surfaces to study Legendrian knots. 

In Subsection~\ref{subsec:satknot} we recall the satellite operation from knot theory and in the following subsection we discuss some relevant results about Legendrian knots and contact structures on the solid torus. 

\subsection{Satellite knots, patterns and companions}
\label{subsec:satknot}

Let $V=D^2\times S^1$ where $D^2$ is the unit disk in $\R^2$. 
A \dfn{smooth pattern type} $\Pat$ is an isotopy class of embeddings of a closed 1--manifold into $V$ that cannot be included in a ball inside $V$.  Let $\mer=\partial D^2$ be the meridian of $V$ and fix $\lon=\{p\}\times S^1$, where $p\in \partial D^2$.  For any representative $P\in \Pat$ of the isotopy class the algebraic intersection number $P\cdot(D^2\times \{\theta\})$ has the same value for any $\theta\in S^1$. This value is independent of the chosen representative $P\in\Pat$, and it is called the \dfn{winding number}, $n=n(\Pat)$, of $\Pat$. Further the \emph{wrapping number}, $w=\mathit{wrap}(\Pat)$, is the minimal geometric intersection of $\Pat$ with a meridian of $V$.

Usually patterns are pictured by their projection as tangles in ${{\Rold}}$ with matching endpoints in $D \times \{0\}$ and $D\times \{1\}$. These will be called open patterns (and we will distinguish them from closed patterns by a subscript $\Pat_{\textrm{open}}$ whenever it is not clear from the context.)

A full twist of a pattern $\Pat$ can be defined as the image $\Delta \Pat$ of $\Pat$ under the map $\Delta\colon V\to V$ that sends $\lon$ to $\lon+\mer$. This operation on the tangle representation is reflected as concatenation of the tangle with a full twist $\Delta$. 
See Figure~\ref{fig:patex}.
A negative full twist $\Delta^{-1}\Pat$ can be defined similarly.

Let $\K$ be a smooth knot type and $\Pat$ a smooth pattern type in $V=D^2\times S^1$. Take representatives $K\in\K$ and $P\in\Pat$, and 
take a tubular neighborhood $N(K)$ of $K$, and fix a longitude $\lambda$ on $\partial N(K)$. The \emph{satellite $P_\lambda(K)$ with companion $K$ and pattern $P$} is the image of $P$ under a diffeomorphism $\psi\colon V\to N(K)$ which sends $\lon$ into $\lambda$. This notion is well defined up to isotopy.
\begin{lemma} \label{satwelldefined}
Let $N(K_0)$ and $N(K_1)$ be neighborhoods of $K_0, K_1\in\K$. Suppose that $\lambda_0$ and $\lambda_1$ represent the same framing. Then for any $P_0, P_1\in\Pat$ and differomorphisms $\psi_0\colon V\to N(K_0)$ and $\psi_1\colon V\to N(K_1)$ that sends $\lon$ to $\lambda_0$, and $\lambda_1$, respectively, 
$\psi_0(P_0)$ and $\psi_1(P_1)$ are isotopic. \qed
\end{lemma}
The above lemma shows that the satellite construction gives a well-defined knot type, which we denote by $\Pat_{\lambda}(\K)$.
The same knot type with respect to a different longitude, has a different pattern:
\[
\Pat_\lambda(\K)\cong(\Delta^{-k}\Pat)_{\lambda+k\mu}(\K).
\]
where $\mu$ is the meridian of $\K$.
If $\K$ has a Seiffert surface then $\Pat(\K)$ denotes $\Pat_\lambda(\K)$, where $\lambda$ is the Seifert framing for $\K$. 

\subsubsection{Symmetries of satellite knots}
\label{subsec:satsymmetry}

When considering satellite knots we will assume that $\K$ is not the unknot and $\Pat$ is not the core of $V$. In this case notice that $T_S=\psi(\partial V)$ will be a non-boundary parallel incompressible torus in the complement of $\Pat_\lambda(\K)$ and it will be called the {\em satellite incompressible torus}. In general, recall that if $C$ denotes the complement of $\Pat_\lambda(\K)$ then there is a JSJ decomposition of $C$, \cite{JacoShalen78, Johannson79}. That is there is a union of tori $T$ in $C$ such that $C\setminus T$ is a union of Seifert fibered spaces and atoroidal manifolds. If the collection $T$ is taken to be a minimal such collection, then it is unique up to isotopy. While JSJ decompositions are defined for general prime 3--manifolds, the case of knot complements has been extensively studied and there is a careful an thorough exposition of this case in \cite{Budney06}.

It is easy to see that $T_S$ is part of this JSJ decomposition of $C$. When there are more tori in the JSJ decomposition of $C$ we will be concerned with certain symmetries that permute the tori. Specifically, consider the situation in Lemma~\ref{satwelldefined}. After isotoping   $\psi_0(P_0)$ to $\psi_1(P_1)$ we have two incompressible tori $T_0=\psi_0(\partial V)$ and $T_1=\psi_1(\partial V)$ in the complement $C$ of $\psi_0(P_0)=\psi_1(P_1)$. In many situations $T_0$ and $T_1$ will have to be isotopic. For example if $\K$ is a hyperbolic knot or a torus knot and $\Pat$  has $\mathit{wrap}(\Pat)\ge 2$. If $T_0$ and $T_1$ are not isotopic then there is a diffeomorphism of $C$ that takes $T_0$ to $T_1$. We will call this an {\em (un-oriented) topological symmetry} and these can be seen through Budney's companionship graphs \cite{Budney06}. 

As we will only be considering cases where un-oriented topological symmetries do not occur we will not discuss the material in \cite{Budney06} in detail, but we will discuss one situation that we need to explicitly exclude below and another to help the reader understand that such symmetries can be subtle. As a first example consider two prime knots $\K_1$ and $\K_2$. The complement of $\K_1\#\K_2$ has two disjoint ``swallow-follow" incompressible tori. To see them consider a neighborhood $N$ of $\K_1\#\K_2$ and the sphere $S^2$ that intersects $N$ in two disks. Note that $S^2$ separates $\partial N$ into two annuli, say $A_1$ and $A_2$. Now let $T_1$ be $S^2\setminus (S^2\cap N)$ union $A_1$ and similarly for $T_2$. See Figure~\ref{fig:sftorus}. These can easily be seen to be incompressible tori in the complement of $\K_1\#\K_2$. Moreover, if $\K_1=\K_2$ then it is easy to see that there is an isotopy of $S^3$ that takes $\K_1\#\K_2$ to itself and exchanges $T_1$ and $T_2$. But if we consider the complement of $\K_1\#\K_2$, then the two tori are not isotopic. This is the simplest example of an un-oriented topological symmetry. 
\begin{figure}[h]
\centering
\begin{overpic}
{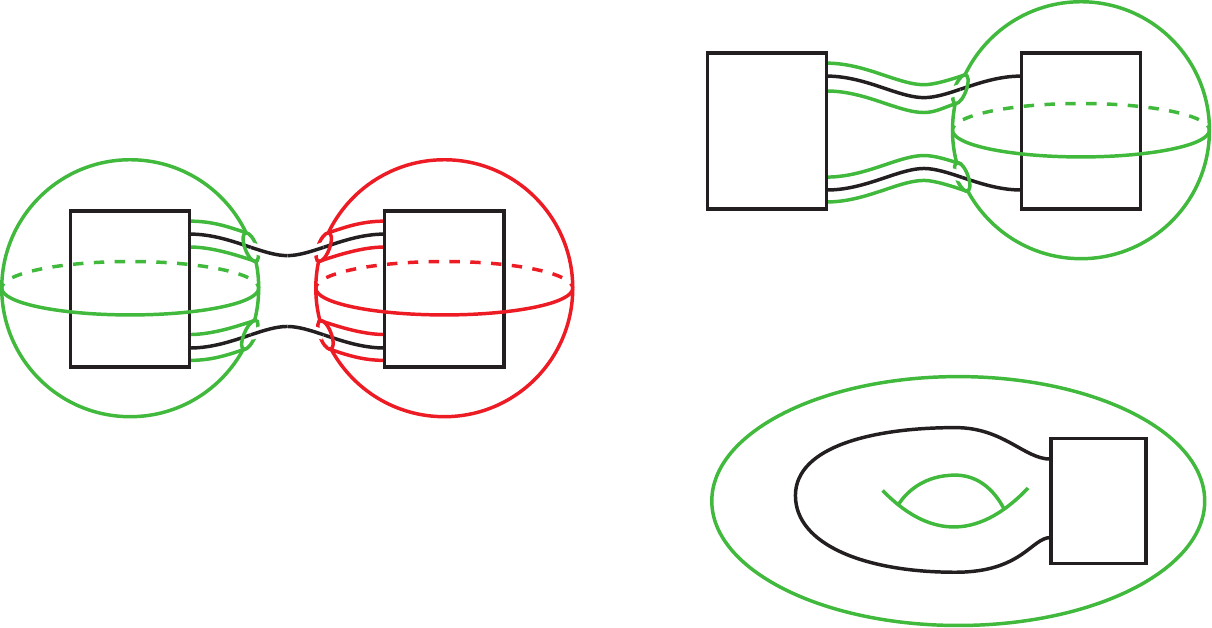}
\put(33,95){$\K_1$}
\put(124,95){$\K_2$}
\put(218,140){$\K_1$}
\put(308,140){$\K_2$}
\put(313,33){$\K_2$}
\end{overpic}
\caption{On the left the two swallow-follow tori in the complement of the connected sum of $\K_1\#\K_2$. On the upper right, the green torus is isotoped to more clearly ``follow" $\K_1$ and ``swallow" $\K_2$. On the lower right is the solid torus the green $T^2$ bounds that shows $\K_2$ as a pattern in $V$.} 
\label{fig:sftorus}
\end{figure}
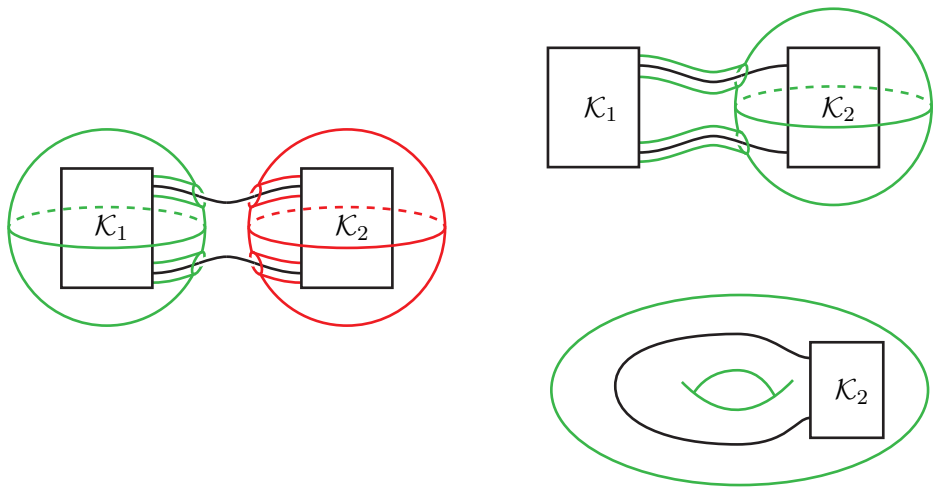

To see this situation arrises from a satellite operation notice that each $T_i$ bounds a solid torus $S_i$ and $S_1$ contains a copy of $\K_2$ and hence defines a pattern $\Pat_{\K_2}$. Similarly $S_2$ contains a copy of $\K_1$ and defines a pattern $\Pat_{\K_1}$. Clearly $\Pat_{\K_2}(\K_1)=\K_1\#\K_2=\Pat_{\K_1}(\K_2)$, see Figure~\ref{fig:sftorus}, and we see the well-known fact that connected sums are a special case of the satellite operation. 

We now consider another situation where topological symmetries arise. We first define splicing of two knots. Given $K_1$ and $K_2$ in two separate copies of $S^3$ let $X_1$ and $X_2$ be the complements of open neighborhoods of $K_1$ and $K_2$, respectively. The {\em splice} of $K_1$ and $K_2$ is the result of gluing $X_1$ and $X_2$ together along their boundaries by a diffeomorphism that interchanges their longitudes and meridians. If one of the $K_i$ is an unknot then it is clear that the resulting manifold is $S^3$. Now given a link $L$ with components $L_0\cup\ldots \cup L_k$ such that $L-L_0$ is an unlink and knots $K_1,\ldots, K_k$ then consider the result of splicing each $K_i$ to $L_i$. The result will again be $S^3$ with a knot $L_0$ in it. Notice that if the $K_i$ are non-trivial knots then the complement of $L_0$ has lots of incompressible tori, namely the boundaries of the neighborhoods of the $L_i$. It is an easy exercise to see that if $L=L_0\cup L_1\cup L_2$ where $L_0$ is the unknot and the $L_i$ are meridians to $L_0$, then splicing $K_1$ and $K_2$ to $L_1$ and $L_2$ will result in $L_0$ being the connected sum of $K_1$ and $K_2$ and the two incompressible tori coming form the $L_i$ are the ones described above and the topological symmetry when $K_1=K_2$ comes from the symmetry of $L$ that fixes $L_0$ and interchanges the other $L_i$. 

Now consider the Borromean rings $L=L_0\cup L_1\cup L_2$. Notice we again have a symmetry fixing $L_0$ and interchanging the other $L_i$. Splicing in $K_1$ and $K_2$ will turn $L_0$ into a knot $K$. We claim that $K$ is a satellite knot. To see this let $L'=L'_0\cup L_2'$ be the image of $L_0\cup L_2$ after $L_1$ is spliced to $K_1$. It is easy to see that $L_2'$ is still an unknot. So $L_0'$ in $S^3-L_2'$ is a pattern $\Pat$. And splicing $L_2'$ with $K_2$ is the same thing as forming $\Pat(K_2)$. But as discussed above if $K_1=K_2$ then there will be a topological symmetry of the complement of $\Pat(K_2)$ that interchanges the two incompressible tori that can both be seen as satellite tori. This is an example of a topological symmetry that does not come from the connected sum of two knots. 

\subsubsection{Orientation symmetries of satellite knots}
\label{subsec:satosymmetry}

We will also need to consider {oriented topological symmetries}. To that end notice that in the definition of the satellite construction we are implicitly considering oriented knots to identify $V$ with the neighborhood $N(K)$ of the knot $K$ we need not only a framing on $K$ but also an orientation on $K$. We will also be assuming that our patterns are oriented. Later we will want to consider solid tori that are standard neighborhoods of oriented Legendrian knots (representing $K$) and in particular we will be focusing on the boundary of these standard neighborhoods.  The boundaries of these neighborhoods uniquely determine the oriented Legendrian knot {\em if} an orientation on a longitude is chosen. However, given a pattern $\Pat$ with non-zero winding number, we will always orient $\Pat$ so that the winding number is positive. Now given a torus $\psi(\partial V)$ the orientation on the longitude is determined by the image of $\Pat$. In particular if there are no un-oriented topological symmetries, as discussed above, then when $\psi_0(P_0)$ is isotoped to be the same as $\psi_1(P_1)$ the tori $T_0$ and $T_1$ (we are using the notation from the paragraph above on topological symmetries) will be isotopic through an isotopy taking the preferred orientation on a longitude of $T_0$ to the preferred orientation on a longitude of $T_1$. 

If the winding number of $\Pat$ is zero then this is not the case. Consider the diffeomorphism $f\co V\to V$ defined by $f((x,y),\theta)=((-x,y), -\theta)$ where $V=D^2\times S^1$ with angular coordinate $\theta$ on $S^1$ and Cartesian coordinates on $D^2$. Then there are patterns $\Pat$ such that $\Pat$ and $f(\Pat)$ are the same, for example the Whitehead patterns discussed in Section~\ref{lwp} have this property. Notice that it is a necessary condition for this to happen that the winding number of $\Pat$ is zero. Now suppose $\psi\co V\to S^3$ parameterizes a neighborhood of an oriented knot $\K$ then $\psi\circ f$ parameterizes a neighborhood of $-\K$ (that is $\K$ with the opposite orientation). Moreover, $\Pat(-\K)=(f(\Pat))(\K)=\psi\circ f(\Pat)=\psi(\Pat)=\Pat(\K)$, thus there is no way to assign an orientation to a longitude of $\psi(\partial V)$ (or equivalently fix the orientation on $\K$) from the satellite knot $\Pat(\K)$. We will call this ambiguity in the orientation of $\K$ an {\em oriented topological symmetry}.

\subsection{Legendrian and transverse knots} 
We assume the reader is familiar with Legendrian and transverse knots. The majority of the material used in this paper can be found in \cite{Etnyre05, EtnyreHonda01b} but we recall a few lesser-known results we will need below. We will denote the set of contact structures on a 3--manifold $M$ by $\Xi(M)$. 

\begin{theorem}\label{isoIScontacto}
Let $M$ be a closed 3--manifold on which the space of contact structures isotopic to a fixed contact structure $\xi$ is simply connected. Then classifying Legendrian knots up to contactomorphism (smoothly isotopic to the identity) is equivalent to classifying them up to Legendrian isotopy.  This is also true for a manifold with boundary if the contact structures and diffeomorphisms are all fixed near the boundary. 
\end{theorem}
\begin{proof}
Fix a contact structure $\xi$ on $M$. Clearly if two Legendrian knots are isotopic then there is a contactomorphism taking one to the other (since on a compact manifold Legendrian isotopies can be extended to global contact isotopies).

Now suppose there is a contactomorphism $\phi\co M\to M$ that take the Legendrian knot $L$ to $L'$ and is smoothly isotopic to the identity. Let $\phi_t, t\in[0,1],$ be that isotopy where $\phi_0(x)=x$ and $\phi_1(x)=\phi(x)$. Notice that $\xi_t=(\phi_t)_*(\xi)$ is a loop of contact structures on $M$ based at $\xi$. By hypothesis this loop is contractible. Thus there is a map $H\co [0,1]\times [0,1]\to \Xi(M)$ such that $H(t, 0)=\xi_t, H(t,1)=\xi, H(0,s)=\xi$ and $H(1,s)=\xi$. Applying Moser's method to $H(t, s)$ for $s\in[0,1]$ and $t$ fixed and then noticing that as $t$ varies the diffeomorphism constructed by the method vary smoothly we see that there is a map $F\co [0,1]\times[0,1]\to \text{Diff}(M)$ such that $F(t,0)(x)=x$ and $F(t,s)_*(H(t,s))=H(t,1)=\xi$. In particular $F(0,s), F(t,1)$ and $F(1,s)$ are all contactomorphism of $\xi$. Moreover concatenating these paths gives a path of contactomorphisms that is isotopic, rel.\ endpoints, to $\phi_t$ as a path of diffeomorphisms.  Thus this gives the desired ambient contact isotopy taking $L$ to $L'$. 
\end{proof}

We now recall a fundamental result of Eliashberg. 
\begin{theorem}[Eliashberg 1992, \cite{Eliashberg92a} {\em cf.\ }\cite{Grioux93}]\label{contactOnB3}
Given the 3--ball $B^3$ with a singular foliation $\mathcal{F}$ on its boundary that could be the characteristic foliation of a tight contact structure on $B^3$. There is a unique (up to contact isotopy) tight contact structure on $B^3$ inducing $\mathcal{F}$ and the space of tight contact structures on $B^3$ inducing $\mathcal{F}$ is simply connected. 
\end{theorem}

Let $V=D^2\times S^1$ and $\Gamma$ denote a two component slope zero longitudinal dividing curve on $\partial V$. Let $\Xi(V,\Gamma)$ denote the space of tight contact structures on $V$ with a fixed characteristic foliation on $\partial V$ divided by $\Gamma$. Whenever we talk about contact structures on manifolds with boundary we need to fix a characteristic foliation $\mathcal F$ on the boundary divided by the dividing curve of the boundary. Then uniqueness means, that up to isotopy fixing a neighborhood of the boundary there is a unique contact structure with the prescribed characteristic foliation. 
The following lemma is a well-known folk theorem. A proof recently appeared in \cite[Theorem~2.36]{VogelPre13} but we sketch an argument here for completeness. 
\begin{lemma}\label{lem:contactstr} 
With the notations above
$\pi_1(\Xi(V,\Gamma))=1$.
\end{lemma}

\begin{proof}
Let $\xi_t$ be an $S^1$-family of contact structures with the given boundary conditions. Choose a meridional disc $D$ of $V$, and isotope $\xi_t$ so that it is convex for all $t\in S^1$ (to guarantee this one needs to observe that since our contact structures are tight any bypass attachment to such a disk must be trivial and so unnecessary). The dividing curve on $D$ is one connected arc, and by another isotopy of $\xi_t$ we can arrange that the characteristic foliation on $D$ is isotopic for all $t$. A further isotopy makes the $\xi_t$ agree in a neighborhood $N$ of $\partial V\cup D$. There is an $S^2$ in this neighborhood that bounds a 3-ball $B$ in $V$ so that $V=N\cup B$. Thus we have  an $S^1$-family of contact structures on $B$ with fixed boundary conditions. By Theorem~\ref{contactOnB3} the fundamental group of the space of tight contact structures on $B^3$ with a given characteristic foliation is trivial. Thus this loop of contact structures is contractible. This completes a contraction of the loop $\xi_t$ as well.
\end{proof}
By Theorem~\ref{isoIScontacto} and this lemma we can conclude the following. 
\begin{corollary}\label{cor:iso}
Let $\xi_V$ be any tight contact structure on $V$ with convex boundary and dividing curves $\Gamma$. Then the classification of Legendrian knots in $(V,\xi_V)$ up to contactomorphism (smoothly isotopic to the identity) and up to isotopy are the same.\hfill\qed
\end{corollary}
We note that since a classification of Legendrian knots in a knot type determines the classification of transverse knots in that knot type, \cite{EtnyreHonda01b}, this corollary also holds for transverse knots.

\section{Open Legendrian and transverse braids}
\label{sec:braidsind2xi}
In this section we will classify Legendrian and transverse representatives of open braids in $\Rold$. 
\subsection{Legendrian braids in $\Rold$}\label{basicdefs}
Throughout this section we will be considering the contact structure $\xi_\Rold$ on $\Rold$, where $I=[0,1]$, that is $I$-invariant, tangent to the boundary of each $D\times \{t\}$, and induces a single dividing curve on $D$. We note that the interior of $\Rold$ can naturally be identified with the 1-jet space of $I$ and hence we can depict Legendrian knots in $\Rold$ via their front projection.

We say a Legendrian arc $\gamma$ in $(\Rold,\xi_\Rold)$ is \dfn{straight} if it is of the form $\{p\}\times I$ for some point $p$ in the dividing set $\Gamma_D$ of $D$. An arc $\gamma$ that intersects $D\times \{0\}$ or $D\times\{1\}$ is \dfn{straight near the boundary} if near its end point it agrees with a straight Legendrian arc.  A \dfn{(open) Legendrian braid} $Q$ of index $n$ in $(\Rold,\xi_\Rold)$ is a collection of $n$ Legendrian arcs forming a topological braid that are straight near the boundary. We note that it is easy to see that any collection of Legendrian arcs that topologically form a braid can be isotoped through Legendrian arcs to a Legendrian braid. 
Figure~\ref{fig:openbraid} depicts the front projection of some Legendrian braids.
\begin{figure}[h]
\centering
\begin{overpic}
{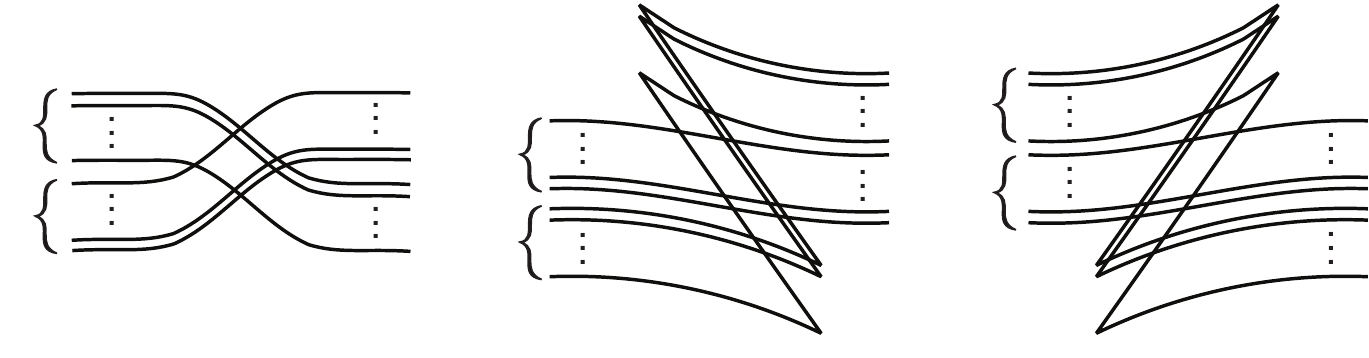}
\put(1,65){$k$}
\put(2,38){$l$}
\put(142,56){$l$}
\put(143,31){$k$}
\put(280,70){$k$}
\put(281,45){$l$}
\put(50, 15){$X(k,l)$}
\put(190,0){$S(k,l)$}
\put(330,0){$Z(k,l)$}
\end{overpic}
\caption{Front projections of \emph{basic Legendrian braids}. There may be other strands both above and bellow of the pictured braids, but they are all assumed to be horizontal strands that are disjoint from the strands in the picture.} 
\label{fig:openbraid}
\end{figure}
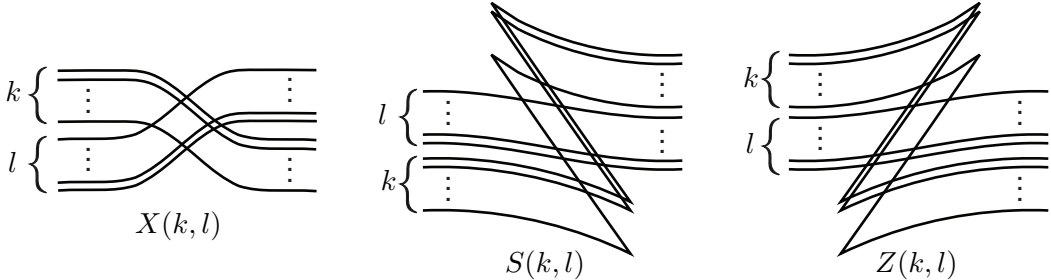
When considering Legendrian braids we allow the end points to move along the dividing set $\Gamma_D$, but they will always remain straight near the boundary. 

We notice that if two copies of $(\Rold, \xi_\Rold)$ are glued together so that $D\times \{1\}$ in the first copy is glued to $D\times \{0\}$ in the second copy, then the result is a contact manifold that is naturally contactomorphic to a subset of $(\Rold, \xi_\Rold)$. Thus two Legendrian $n$-braids can be concatenated to obtain a new Legendrian braid. 

We define basic building blocs for Legendrian braids. Fixing the braid index $n$ for each triple of natural numbers $i, k, l$ such that $i+k+l\leq n$ we define $X_i(k,l)$ to be the Legendrian braid depicted on the left of Figure~\ref{fig:openbraid} with $i$ straight Legendrian arcs below the pictured braid and $n-(i+k+l)$ straight Legendrian arcs above the pictured braid. We similarly have $S_i(k,l)$ and $Z_i(k,l)$ indicated in the middle and right of the figure, respectively. We will usually drop the subscript $i$ from the notation when the meaning is clear from context. The braids $X_i(k,l), S_i(k,l)$ and $Z_i(x,l)$ are called \dfn{basic Legendrian braids}. 

\begin{theorem}\label{thm:legbraid}
 A Legendrian $n$-braid $Q$ in $(\Rold,\xi_\Rold)$ is Legendrian isotopic to a concatenation of the basic Legendrian braids. 
\end{theorem}

The above theorem will be proved in Section~\ref{proofofbraid} by classifying tight contact structures on the complement of $Q$ but before giving the proof we discuss some corollaries of this theorem. First notice that the building blocks in Theorem~\ref{thm:legbraid} can be simplified.
\begin{corollary}\label{thm:legbraidsimpl}
A Legendrian braid $Q$ in $(\Rold,\xi_\Rold)$ is Legendrian isotopic to the concatenation of the basic Legendrian braids $X_i =X_i(1,1)$, $S=S_0(1,n-1)$ and $Z=Z_0(1,n-1)$ shown in Figure~\ref{fig:openbraidsimpl}. 
\end{corollary}

Notice that the corollary implies that a Legendrian 2-braid is a concatenation of  the building blocks $X_0, S$, and $Z$. This is a key result necessary for the classification of Legendrian twist knots given in \cite{EtnyreNgVertesi13}.

\begin{proof}
The Legendrian braid $X(k,l)$ is a concatenation of $(kl)$ copies of $X_i$ for the appropriate choices of $i$. The braid $S(1,l)$ is obtained from $S$ and $(n-1-l)$ copies of $X_i$ placed before or after as necessary. Then $S(k,l)$ is the concatenation of $k$ copies of $S(1,l)$. Similarly  $Z(k,l)$ can be built up from $k$ copies of ${Z}$ and $k(n-1-l)$ copies of $X_i$.
\end{proof}

\subsection{Invariants of open Legendrian braids}
Let $\Pat$ be denote the smooth isotopy class of an (open) braid (with isotopies relative to the boundary). The set of Legendrian isotopy classes representing $\Pat$ is denoted by $\Leg_{\Rold}(\Pat)$. We also note that we always orient all strands of a braid from left to right. We define the \emph{relative Thurston--Bennequin number} and the \emph{relative rotation number} of a Legendrian braid type, $\Q$, in $(\Rold,\xi_\Rold)$ using the front projections of its chosen representative $Q$ as follows
\begin{eqnarray*}
\reltb_{\Rold}(Q)&=&\textrm{writhe}(Q)-\frac12c(Q);\\
\relrot_{\Rold}(Q)&=&\frac12(d(Q)-u(Q)).
\end{eqnarray*}
where $c(Q)$ denotes the number of cusps, and $d(Q)$ and $u(Q)$ denotes the number of downward and upward oriented cusps, respectively. This number is independent on the chosen representation, thus giving rise to the invariants $\reltb_{\Rold}(\Q)$ and $\relrot_{\Rold}(\Q)$ of the Legendrian isotopy type $\Q$. 

Denote the the set of Legendrian isotopy classes with relative Thurston--Bennequin number $t$  by $\Leg_{\xi_\Rold}(\Pat;t)$. Let $\mathcal{X}_i,\mathcal{S},\mathcal{Z}$ denote the Legendrian isotopy classes of the braids $X_i, S$ and $\mathcal{Z}$. The relative classical invariants for the basic Legendrian braids are
\[
\begin{array}{lll}
\reltb_{\Rold}(\mathcal{X}_i)=1, & \reltb_{\Rold}(\mathcal{S})=-n, & \reltb_{\Rold}(\mathcal{Z})=-n,\\
\relrot_{\Rold}(\mathcal{X}_i)=0, & \relrot_{\Rold}(\mathcal{S})=-1, & \relrot_{\Rold}(\mathcal{Z})=1.\\
\end{array}
\]

\subsection{Positive Legendrian braids in $(\Rold,\xi_\Rold)$}
For positive braids the maximal Thurston--Bennequin number is known to be the length of $\Pat$.
\begin{theorem}\label{openmaxtb}
Let $\Pat$ be a (open) braid represented by a positive word $w$ in the braid group $B_n$, then the maximal relative Thurston-Bennequin invariant for Legendrian knots realizing $\Pat$ is
 \[\overline{\reltb}_{\Rold}(\Pat)=\textit{length}(w).\]
\end{theorem}
\begin{proof}
If the topological braid generators $\sigma_i$ in the positive word $w$ are replaced by the basic Legendrian braids $X_i$, then we easily see that $\overline{\reltb}_{\Rold}(\Pat)\geq \textit{length}(w)$. To see the other inequality notice that if there is a Legendrian braid $Q$ in  $\Leg_{\Rold}(\Pat)$ with larger relative Thurston-Bennequin invariant then we can embed it in the standard contact structure on $\R^3$ and complete it as shown in Figure~\ref{fig:cb} resulting in a Legendrian link. 
\begin{figure}[h]
\begin{overpic}
{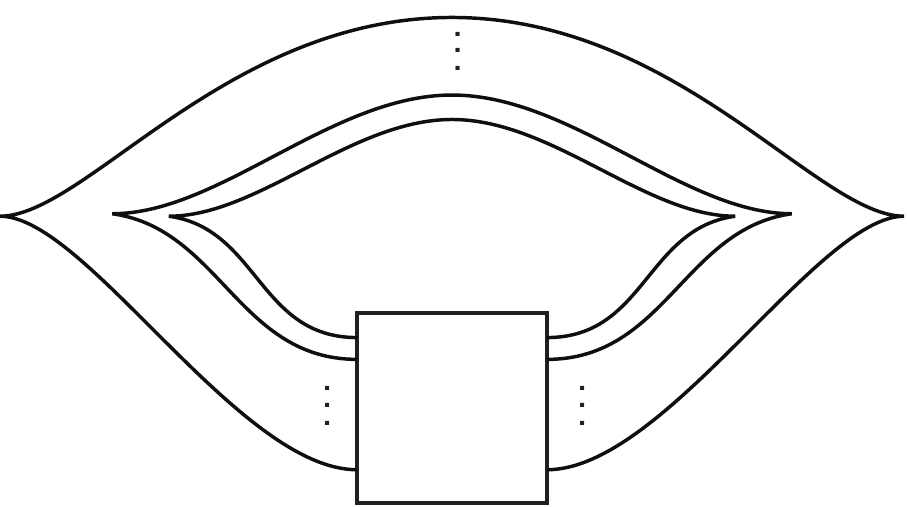}
\end{overpic}
\caption{The closure of an open $n$-braid in $\R^3$.}
\label{fig:cb}
\end{figure}
If this is a knot then its Thurston-Bennequin invariant is $\reltb_{\Rold}(Q)-n$. Moreover, the maximal Euler characteristic of a Seifert surface for the knot is $n-\textit{length}(\Pat)$, thus we have a Legendrian knot violating the Bennequin bound. If the link in Figure~\ref{fig:cb} is not a knot then one may easily concatenate $\Q$ with some of the basic Legendrian braids $X_i$ so that its ``closure" is a knot and again violates the Bennequin bound. 
\end{proof}
Moreover we can classify (open) positive Legendrian braids with maximal relative Thurston-Bennequin invariant.
\begin{theorem}\label{thm:poslegendrianbraid}
Let $\Pat$ be a (open) braid represented by a positive word $w$ in the braid group $B_n$, then $\Pat$ has a unique Legendrian representative with maximal Thurston--Bennequin number.
 \end{theorem}
\begin{proof}
Given a Legendrian representative of $w$ Theorem~\ref{thm:legbraidsimpl} allows us to express it in terms of the basic Legendrian braids $X_i$, $S$ and $Z$.  This will give a presentation $w'$ of $w$ in terms of the standard generators $\sigma_i$ of the braid group. For each $Z$ or $S$ this word has a term of the form $(\sigma_0\cdots \sigma_{n-1})^{-1}$ (or the reverse of this). Since the algebraic length of $w$ and $w'$ are the same there will have to be $n-1$ compensatory $X_i$s. So the over all contribution to the Thurston-Bennequin invariant of the terms is $-1$. Thus we see that the Thurston-Bennequin invariant of this Legendrian braid is equal to the algebraic length of $w$ minus the number of $S$s and $Z$s, and so there can be no $S$s and $Z$s.
 
Now we need to prove that any two representatives with maximal Thurston-Bennequin number are Legendrian isotopic. The braid moves that contain only positive powers ($\sigma_i\sigma_{i+1}\sigma_i=\sigma_{i+1}\sigma_i\sigma_{i+1}$ and $\sigma_i\sigma_j=\sigma_i\sigma_j$ (for $|i-j|\ge 2$) ) can be represented by Legendrian isotiopies. These are the relations in the monoid $B_n$ of positive (open) braids, thus Theorem \ref{thm:posbraid} finishes the proof.
 \end{proof}
\begin{theorem}[{\em cf} \protect{\cite[Section~6.5.4]{KasselTuraev08}}]\label{thm:posbraid}
 If two positive braid words are equivalent in the group $B_n$ then they are equivalent in the positive braid monoid too. \qed
\end{theorem}

\subsection{Bypasses and Legendrian braids}\label{proofofbraid}
In this section we will prove Theorem~\ref{thm:legbraid}. To that end we begin by observing that by Corollary~\ref{cor:iso} classifying Legendrian braids in $(\Rold, \xi_\Rold)$ up to isotopy and contactomorphism are equivalent. Moreover, the contactomorphism type of a Legendrian braid is determined by the contact structure on the complement of a standard neighborhood of the braid up to contactomorphism fixing the back face of the braid complement. To clarify this last statement we begin by discussing the standard neighborhood of a Legendrian braid. 

\subsubsection{Standard neighborhoods of Legendrian braids}
We consider (open) Legendrian $n$-braids and isotop them so that they become straight near the boundary. To this end we fix $n$ points $p_1,\ldots, p_n$ on $\Gamma_D$ ordered from bottom to top as shown in Figure~\ref{fig:punctured disc}. 
\begin{figure}[ht]
\begin{overpic}
{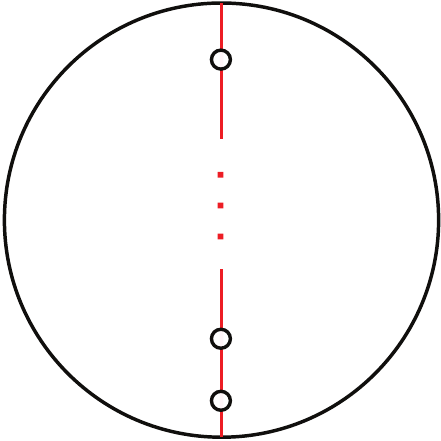}
\end{overpic}
\caption{The dividing curve on $S=D-\nu(Q)$. The small circles are the $D^j_i$.}
\label{fig:punctured disc}
\end{figure}
By Giroux realisation we can arrange that on $D\times \{j\}, j=0,1,$ there are disjoint disks $D_i^j, i=1,\ldots n,$ containing $p_i$ such that $D_i^j$ has Legendrian boundary with Thurston-Bennequin invariant $-1$ and standard characteristic foliation shown in Figure~\ref{fig:sdtfoliation}. 
\begin{figure}[h]
\begin{overpic}
{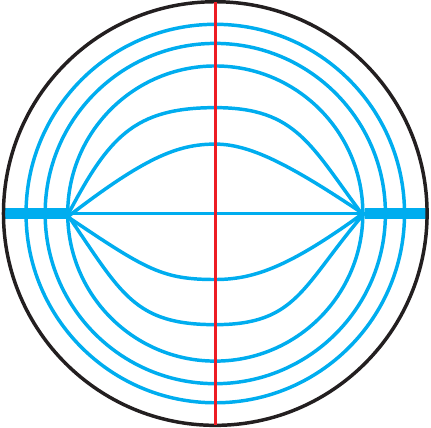}
\end{overpic}
\caption{The characteristic foliation on a disk neighborhood $D^j_i$ of the $p_i$.} 
\label{fig:sdtfoliation}
\end{figure}
Now given a Legendrian $n$-braid $Q$ a standard neighborhood of a strand of $Q$ will be a neighborhood $D\times I$ of the strand such that $D\times \partial I$ consists of two disks that are sub-disks of the $D_i^j$ with Legendrian boundary and $(\partial D)\times I$ is a convex annulus with two dividing curves running between the boundary components of the annulus.  We can also assume, by Giroux realization, that the characteristic foliation on the annulus consists of two lines of singularities parallel to the dividing curves and the rest of the foliation given by the boundary of meridional disks. Now a standard neighborhood of $Q$ is a neighborhood $\nu(Q)$ of $Q$ that is a standard neighborhood of each of its strands. We will call the contact manifold $\Rold\setminus \nu(Q)$ the \dfn{exterior} of $Q$. We will call $D\times \{0\}$ intersected with the exterior the \dfn{back face}, the intersection with $D\times \{1\}$ the \dfn{front face} and the remainder of the boundary the \dfn{vertical boundary}

Since there is a unique tight contact structure on a ball, any contactomorphism of the complement of the standard neighborhoods of Legendrian braids that fixes the back face can be extended over the neighborhoods to a contactomorphism of $D\times I$ preserving the Legendrian braids. (Notice that extending over the neighborhood of the braid is done by gluing 2-handles to the complement that correspond to neighborhoods of the meridional disks. It is important that we fix the back face of the complement of the braid so that the contactomorphism preserves the attaching regions of the 2-handles and can thus be extended over them.) Thus the classification of Legendrian braids up to contactomorphism is equivalent to the classification of the exteriors of Legendrian braids up to contactomorphism fixing the back face. 

\subsubsection{Straightening standard neighborhoods of Legendrian braids}
Notice that we can put the exterior of a Legendrian braid in a standard form. The basic idea is to make the boundary of the exterior of all braids look the same for all braids except for the front face $D\times \{1\}$. See Figure~\ref{fig:straightD2xI}. 
\begin{figure}[h]
\begin{overpic}
{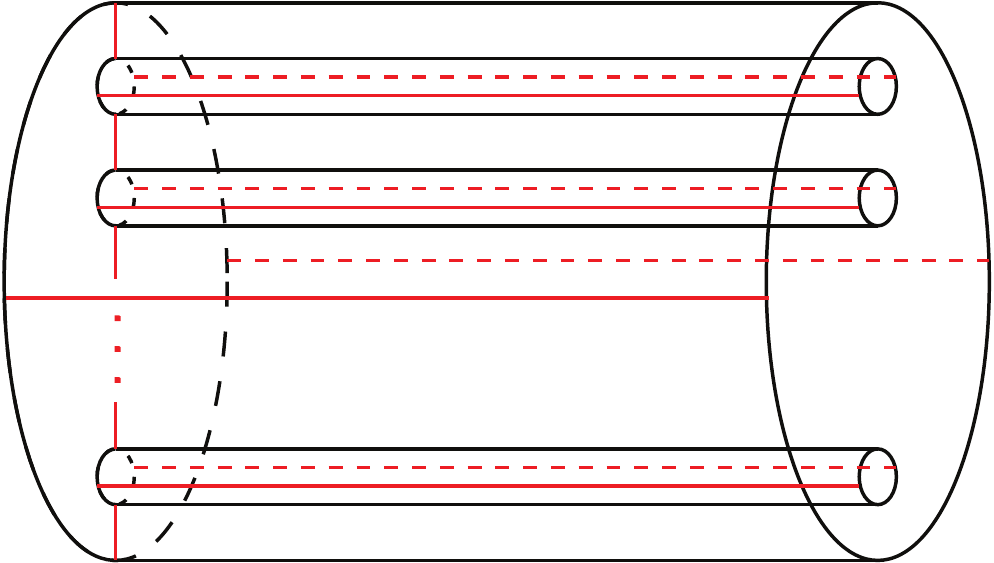}
\end{overpic}
\caption{Straightening the strands of $\nu(Q)$}
\label{fig:straightD2xI}
\end{figure}
More specifically, Let $D_n$ be the convex disk shown in Figure~\ref{fig:punctured disc} and let $\xi_n$ be the $I$ invariant contact structure on $D_n\times I$. Notice that $(\partial D_n)\times I$ can be made convex so that the dividing curves are all parallel to the $I$-factor. We can also assume that $\partial D_n\times \{t\}$ is Legendrian for each $t\in I$.  Given a Legendrian $n$-braid $Q$ and a standard neighborhood $\nu(Q)$ of $Q$, there is a smooth diffeomorphism of $\Rold\setminus \nu(Q)$ to $D_n\times I$ that is the identity on the back face, and takes the front face to the front face and the vertical boundary to the vertical boundary. 
Pushing forward the contact structure on $\Rold\setminus \nu(Q)$ by this diffeomorphism gives the \dfn{straightened neighborhood of $Q$}. Notice that everything on the boundary is standard except on the front face where the dividing curves can be quite complicated.  In Figure~\ref{fig:straight} and ~\ref{fig:straight2} we show the front face of the straightened basic braids $X(k,l)$ and $Z(k,l)$, respectively. The front face for the straightened braid $S(k,l)$ is obtained from Figure~\ref{fig:straight2} by rotating the picture by $\pi$. 

\begin{figure}[htb]
\begin{overpic}
{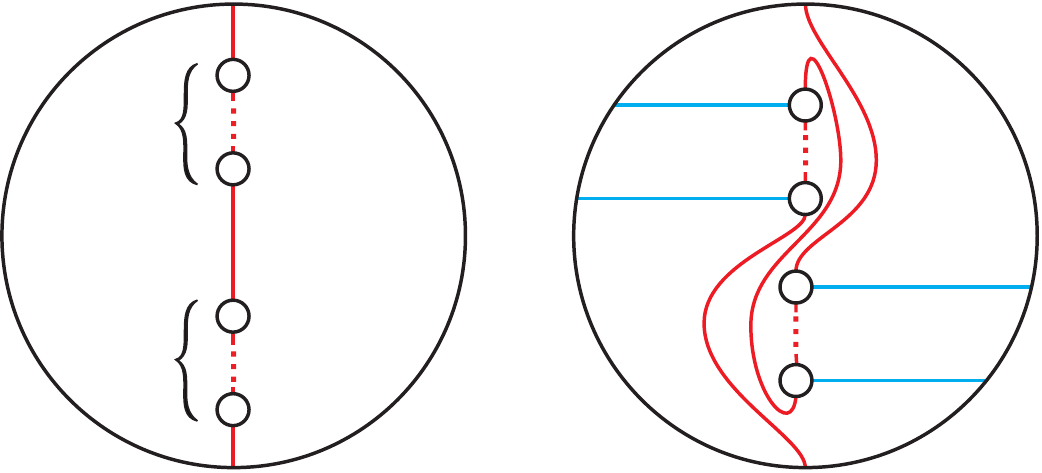}
\put(45,98){$l$}
\put(43,29){$k$}
\end{overpic}
\caption{The front face of the exterior of the basic Legendrian braid $X(k,l)$ before straightening on the left and after straightening on the right. The horizontal arcs on the right will be the boundary of product disks used below.}
\label{fig:straight}
\end{figure}

\begin{figure}[htb]
\begin{overpic}
{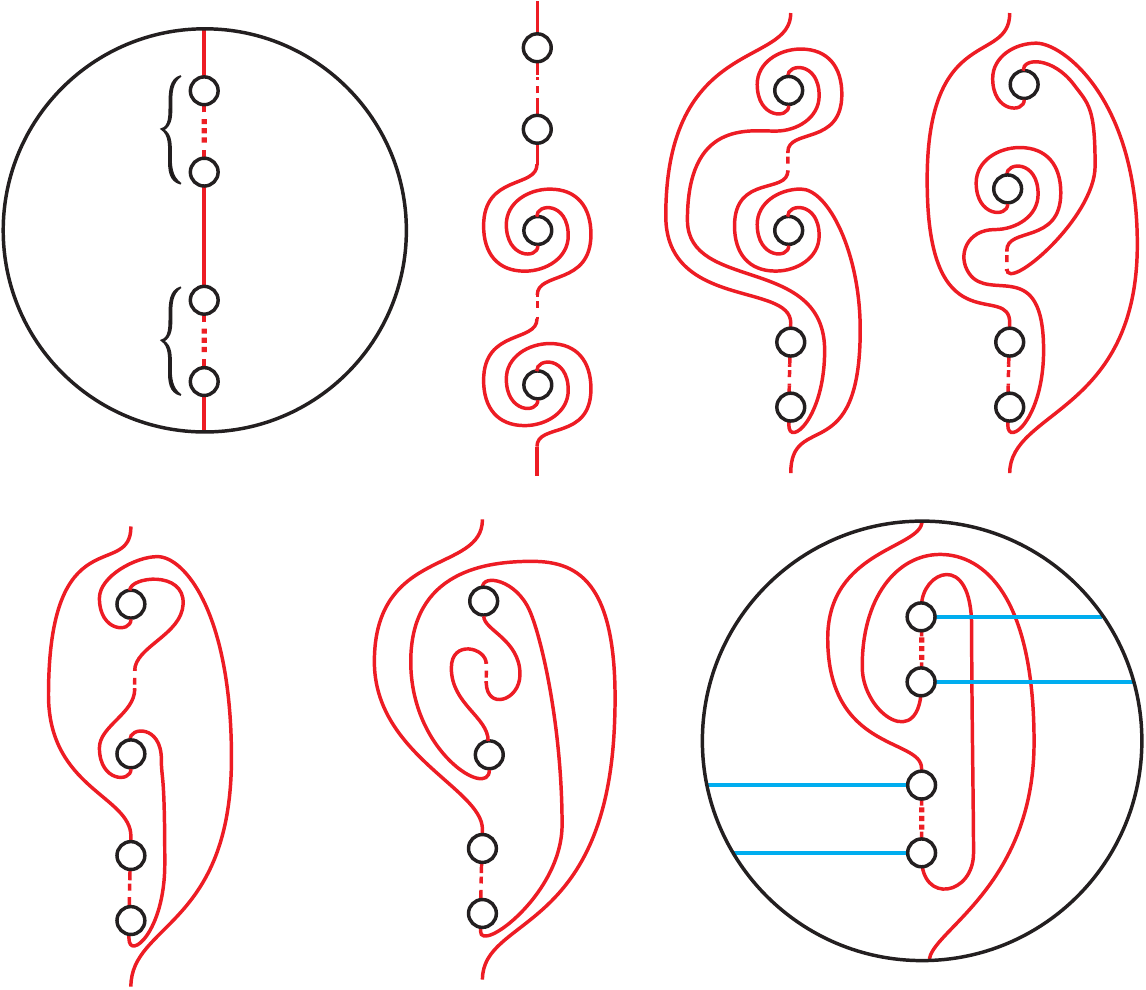}
\put(40,245){$l$}
\put(38,185){$k$}
\end{overpic}
\caption{The front face of the exterior of the basic Legendrian braid $Z(k,l)$ before straightening on the upper left and after straightening on the lower right. In the second figure on the top the twists in the vertical boundary have been pushed to the front face. The next figure untwists the half twist between the $k$ and $l$ strands. The next three figures untwists the full twist between the $k$ strands. The horizontal arcs on the right will be the boundary of product disks used below.}
\label{fig:straight2}
\end{figure}

Notice that in the straightened neighborhood each $D_n\times \{t\}$ has Legendrian boundary. According to \cite[Proof of Lemma 3.10]{Honda02} we can arrange that there are finite number of $0=t_0<t_1<\ldots <t_k=1$ such that each $D_n\times [t_{i-1},t_i]$ is a bypass layer (that is obtained from an $I$ invariant contact structure on a neighborhood of $D_n\times \{t_i\}$ by attaching a bypass). 

\subsubsection{Bypasses and basic Legendrian braids}
We proceed by understanding a single bypass attachment. The $6$ ways of attaching a bypass to $D\times\{0\}$ are depicted in Figure \ref{fig:bypassxsz}. Since $\xi_\Rold$ is tight, only the first 3 type of attachment is allowed (otherwise after the bypass attachment we would create a dividing set with a contractible component, when considered on the disk $D$, and thus an overtwisted disk).
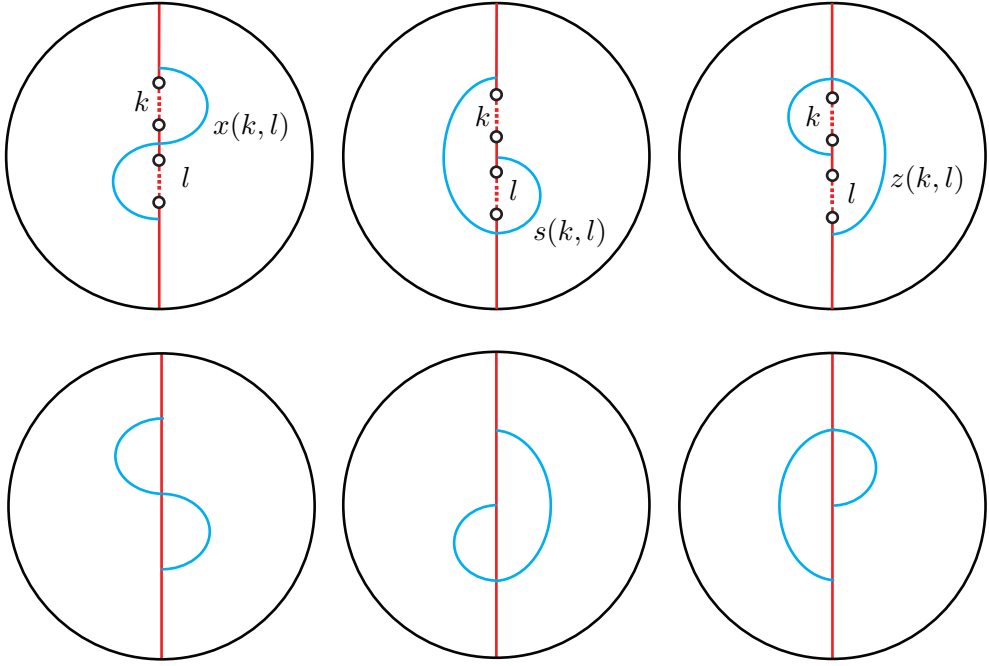
\begin{figure}[h]
\centering
\begin{overpic}
{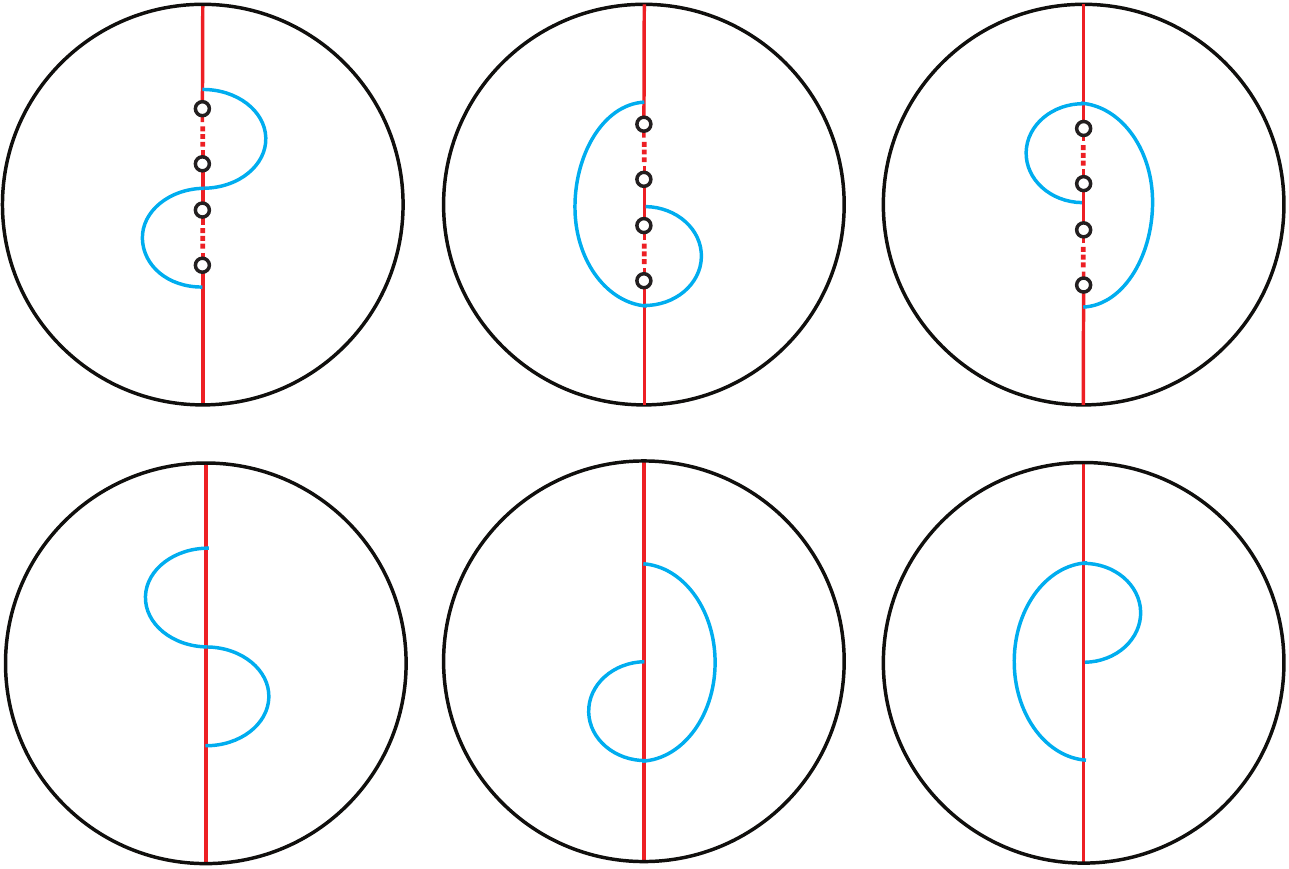}
\put(49,209){$k$}
\put(67,180){$l$}
\put(178,203){$k$}
\put(191,175){$l$}
\put(303,203){$k$}
\put(319,174){$l$}
\put(79,200){$x(k,l)$}
\put(200,160){$s(k,l)$}
\put(335,180){$z(k,l)$}
\end{overpic}
        \caption{Possible bypass attachments to $S$. Only some of the punctures $D^i$ that occur along the dividing curve are depicted here.}
        \label{fig:bypassxsz}
\end{figure} 
In the following we will show that these three types of attachments correspond to the front projections of Figure~\ref{fig:openbraid}.
First note, that after the bypass attachments along the curves $x(k,l),s(k,l),z(k,l)$ depicted in Figure~\ref{fig:bypassxsz} we obtain the dividing curves $\Gamma_{\Xkl}, \Gamma_{\Skl}$ and $\Gamma_{\Zkl}$ on $S\times\{1\}$ shown in Figure~\ref{fig:afterbypassxsz}.
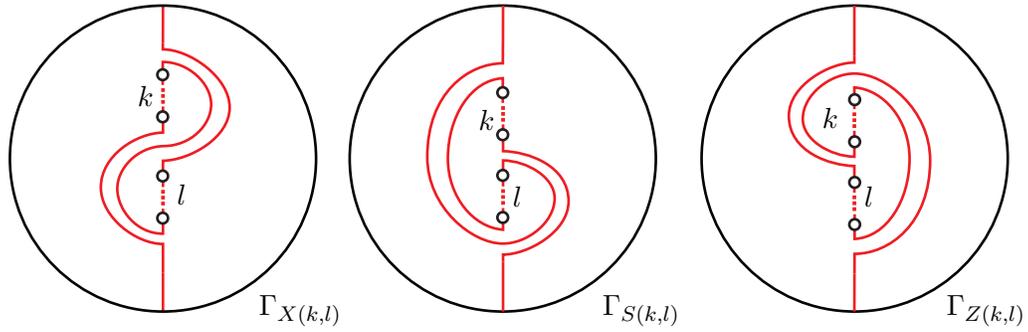
\begin{figure}[ht]
\begin{overpic}
{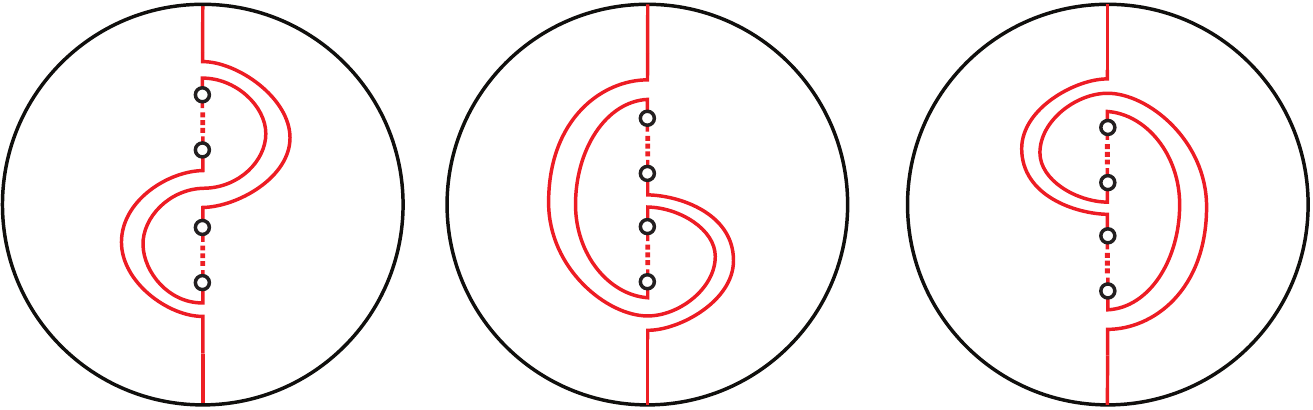}
\put(49,79){$k$}
\put(64,42){$l$}
\put(178,70){$k$}
\put(191,41){$l$}
\put(308,71){$k$}
\put(324,41){$l$}
\put(95,0){$\Gamma_{X(k,l)}$}
\put(225,0){$\Gamma_{S(k,l)}$}
\put(355,0){$\Gamma_{Z(k,l)}$}
\end{overpic}
        \caption{Result of a bypass attachment. The punctures $D^i$ are denoted by dots on the dividing curve.}
        \label{fig:afterbypassxsz}
\end{figure}

\begin{lemma}
Using the notation established above,  suppose that $(D_n\times I, \xi)$ is obtained by a bypass attachment along the curve $x(k,l), s(k,l)$ or $z(k,l)$. Then $(D_n\times I, \xi)$ is contactomorphic to the complement of the standard neighborhood of the basic Legendrian braid $X(k.l)$, $S(k,l)$ or $Z(k,l)$, respectively, shown in Figure~\ref{fig:openbraid}, by a contactomorphism preserving the back face, front face and vertical boundaries. 
\end{lemma}

\begin{proof}
Let $\xi_X$ be the contact structure on $D_n\times I$ coming from straightening the exterior of $X(k,l)$ and let $\xi_x'$ be the contact structure on $D_n\times I$ obtained from an $I$ invariant neighborhood of $D_n$ by attaching a bypass layer along the curve $x(k,l)$ shown in Figure~\ref{fig:bypassxsz}. By construction the characteristic foliation and dividing sets on the back face and vertical boundaries of these braid complements agree. Notice that there is an isotopy (fixed outside a neighborhood of the front face) of the identity map on $D_n\times I$ so that the dividing curves on the front face are also preserved. By Giroux Realization \cite{Giroux91} we can assume the map also preserves the characteristic foliations (technically we are removing a small neighborhood of one of the front faces that lies in an $I$-invariant neighborhood, but it should be clear that this does not affect our argument). So the identity map can be isotoped to a map $\phi$ that is a contactomorphism from $\xi_X$ to $\xi_x'$ in a neighborhood $U$ of $\partial (D_n\times I)$. Let $S$ be a convex surface 
embedded in the interior of the region $U$ that is obtained by rounding the corners of $\partial (D_n\times I)$. There are $n$ disks properly embedded in 
$D_n\times I$ that come as the product of the horizontal lines in Figure~\ref{fig:straight}. We can think of the boundaries of these disks as lying on $S$ and one easily checks they intersect the dividing set exactly twice. Thus we may Legendrian realize the boundaries of these disks and make them convex. Each will contain exactly one dividing curve, thus $\phi$ may be further isotoped to fix the characteristic foliation on the disk. Hence we can isotope $\phi$ to be a contactomorphism on $U'$ which is $U$ union a neighborhood of these disks. Since the complement of $U'$ can be assumed to be a ball and there is a unique tight contact structure on the ball, we may finally isotope $\phi$ to a contactomorphism from $\xi_X$ to $\xi_x'$ on all of $D_n\times I$.

Considering Figure~\ref{fig:straight2} and the rightmost diagram in Figure~\ref{fig:afterbypassxsz} we see that the argument for $Z(k,l)$ is almost identical. The only difference is that the uppermost $k$ horizontal arcs in Figure~\ref{fig:straight2} will result in convex disks with 2 dividing curves each. There are two possibilities for such dividing curves, but one of them will give a bypass that straddles the ``vertical" dividing curve in the bottom right diagram of Figure~\ref{fig:straight2}. Pushing over this bypass will result in a disconnected dividing curve on $\partial (D\times I)$ and hence an overtwisted disk. Thus there is a unique possible configuration for the dividing curves on the disks corresponding the the horizontal lines in Figure~\ref{fig:straight2}. With this observation the argument is identical to the one given above for $X(k,l)$. 

The proof of $S(k,l)$ is identical to the one given above for $Z(k,l)$ after one draws the straightened exterior of $S(k,l)$ and compares it to the middle diagram in Figure~\ref{fig:afterbypassxsz}. 
\end{proof}

\begin{proof}[Proof of Theorem \ref{thm:legbraid}]
The manifold $D_n\times [0,1]$ is built up from bypass layers, and above we understood what front projections each of them correspond to. Thus any Legendrian knot can be built up from concatenations of the Legendrian braids on Figure~\ref{fig:openbraid}. 
\end{proof}

\section{Patterns in $D^2\times S^1$}
\label{sec:legendrianbraidsind2xs1}
This section discusses the definitions, constructions, and basic computations concerning Legendrian and transverse pattens. The definitions and notations used below for smooth patterns in $V=D^2\times S^1$ are given in Subsection~\ref{subsec:satknot}. 

\subsection{Legendrian patterns}
Let $\xi_V$ be the unique (up to isotopy) $S^1$--invariant tight contact structure on $V=D^2\times S^1$ with convex boundary $\partial V$ and dividing curve $\Gamma_{\partial V}=\lon\cup-\lon$.  To be specific we can take $V$ to be a subset of the 1-jet space $T^*S^1\times \R$ with its standard contact structure $\ker(dz-y\, d\theta)$, where $z$ is the coordinate on $\R$, $\theta$ the coordinate on $S^1$ and $y$ the coordinate on the fiber of $T^*S^1=S^1\times \R$. The core $C=\{(0,0)\}\times S^1$ of $V$ can be assumed to be a Legendrian curve. 

\subsubsection{Invariants of Legendrian patterns}\label{invariantsinV}
To define invariants of Legendrian patterns in $V$ we think of $V$ as the 1--jet space of $S^1$ and use the front projection. More specifically, as first observed by Ng \cite{NgThesis} (or \cite{NgTraynor04} for more sophisticated invariants) the relative Thurston-Bennequin number and rotation number of a Legendrian braid $Q$ in $V$ can be computed in terms of the front projection:
\begin{eqnarray*}
\reltb_{V}(Q)&=&\textrm{writhe}(Q_{\rm{open}})-\frac12c(Q_{\rm{open}});\\
\relrot_{V}(Q)&=&\frac12(d(Q_{\rm{open}})-u(Q_{\rm{open}})),
\end{eqnarray*}
where $Q_{\rm{open}}$ is any open Legendrian tangle whose closure is $Q$, $c(Q_{\rm{open}})$ denotes the number of cusps and $d(Q_{\rm{open}})$ and $u(Q_{\rm{open}})$ denotes the number of downward and upward cusps, respectively. Note that the above value is independent of the open pattern $Q_{\rm{open}}$ whose closure is $Q$.
Let $\Leg_{V}(\Pat)$ denote the set of Legendrian isotopy classes of Legendrian representations of the smooth pattern $\Pat$. 
We will denote Legendrian isotopy classes of patterns in the smooth pattern type $\Pat$ with relative Thurston-Bennequin invariant $t$ and relative rotation number $r$ by $\Leg_{V}(\Pat;t,r)$. The set of representatives with relative Thurston--Bennequin number $t$ are denoted by $\Leg_{V}(\Pat;t)$.
\subsubsection{Reimbeddings of Legendrian patterns}\label{reembed}
When studying Legendrian satellite knots it will be useful to ``reimbed" certain patterns into the solid torus. We discuss this here. 

The solid torus $(V,\xi_V)$ can be embedded into itself as follows. The core $C$ is a Legendrian curve, and $(V,\xi_V)$ can be interpreted as a standard neighborhood of $C$. The standard neighborhood $\nu(\St_+^z\St_-^s(C))$ of a stabilization of the core is on the one hand naturally a subset of $(V,\xi_V)=\nu(C)$ and on the other hand it is contactomorphic to $(V,\xi_V)$, thus defining an embedding
\[\zeta^z\sigma^s\colon (V,\xi_V)\hookrightarrow (V,\xi_V)\]
whose image is $\nu(\St_+^z\St_-^s(C))$. Notice that the image of $n$ horizontal strands parallel to $C$ under $\zeta^z\sigma^s$ is $\mathcal{Z}^z\mathcal{S}^s$. 
Since the concatenation of open patterns is well-defined, the above discussion leads to the following simple observation whose proof is left to the reader. 
\begin{lemma}\label{lem:stabpattern}
For a Legendrian pattern type $\Q\in\Leg_{V}(\Pat)$, the followings are equivalent:
\begin{enumerate}
\item\label{lem:stabpatternt item:1} There exists a Legendrian pattern $\widetilde{\Q}\in \Leg_{V}(\Delta^{(z+s)}\Pat)$ such that $\Q=\zeta^z\sigma^s(\widetilde{Q})$;
\item\label{lem:stabpatternt item:2} We have the inclusions $\Q\subset \nu\subset V$ for some standard neighborhood $\nu$ of $\St_+^z\St_-^s(C)$ with $\partial \nu$ isotopic to $\partial V$ in the complement of $\Q$; and 
\item\label{lem:stabpatternt item:3} There exists an open Legendrian pattern $\widetilde{\Q}_\textrm{open}\in\Leg_{\xi_{{\Rold}}}(\Delta^{(z+s)}\Pat_\textrm{open})$ such that $\Q$ is the closure of the open braid $\mathcal{Z}^{z}\mathcal{S}^{s}\widetilde{\Q}_\textrm{open}$.
\end{enumerate}
In particular the condition in Item (\ref{lem:stabpatternt item:3}) is independent on the chosen tangle-representation of $\Q$.\qed
\end{lemma}
\begin{remark}
It is interesting to note, that it is not clear how one would define a full positive twist of a Legendrian pattern. One reason for this is that $(V,\xi_V)$ has no solid sub-torus with two dividing curves of slope $1$.  So one cannot use the construction above for negative twists. One might try cutting open a pattern to get a tangle and then concatenating with $\Delta$ to add a positive twist, but Figure~\ref{fig:noposttwist} shows that this is not well-defined and can result in a patterns related by stabilization. 
\begin{figure}[ht]
\centering
\begin{overpic}
{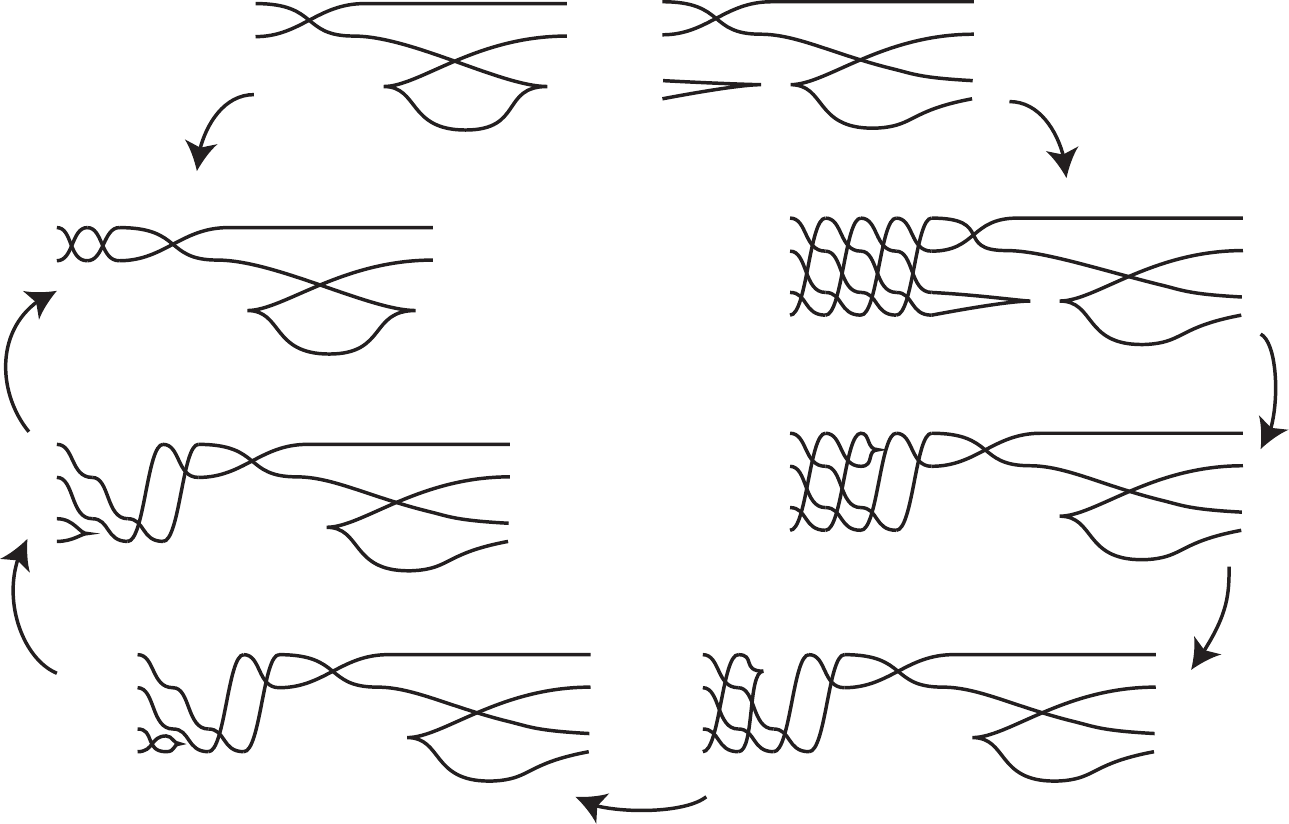}
\put(174,220){$=$}
\put(310,205){add twist}
\put(325,125){isotope}
\put(295,60){destabilize}
\put(170,-8){isotope}
\put(15,60){destabilize}
\put(10,130){isotope}
\put(5,205){add twist}
\end{overpic}
\caption{The top row shows two ways to cut open a pattern in $V$. The next row shows the effect of adding a full positive twist. The other diagrams show that these patterns are related by stabilizations.} 
\label{fig:noposttwist}
\end{figure}
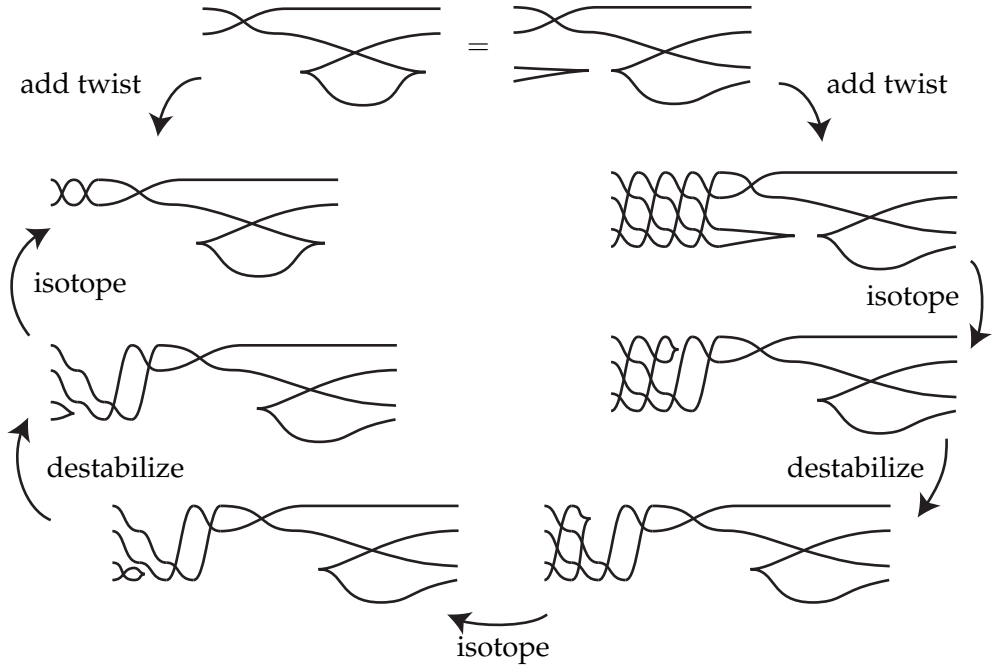
\end{remark}

\subsubsection{Legendrian braid patterns}

Suppose that $\Pat$ represents the closure of a braid word $w\in B_n$ in $V$, and let $Q$ be a Legendrian representative of a Legendrian isotopy class $\Q\in\Leg_{V}(\Pat)$. We may cut $V$ along a meridional disk to get an open Legendrian pattern $Q_{open}$  in $({{\Rold}},\xi_{{\Rold}})$. Moreover this open pattern smoothly represents a conjugate $uwu^{-1}$ of  $w$, where $u\in B_n$.  By Theorem \ref{thm:legbraid} the open Legendrian braid $Q_{\rm{open}}$ is Legendrian isotopic to a concatenation of the building blocks of Figure \ref{fig:openbraid}. This sequence of building blocks defines a braid word equivalent to $uwu^{-1}$ in the group $B_n$.
Thus after gluing the ends of $\Rold$ to obtain $V$ we have established the following result. 
\begin{theorem}\label{thm:legbraids1}
 A Legendrian braid $B$ in $(V,\xi_V)$ is Legendrian isotopic to the closure of (cyclic) concatenation of the building blocks in Figure~\ref{fig:openbraidsimpl}.\hfill \qed
\end{theorem}

We can also prove a closed version of Theorems~\ref{openmaxtb} and~\ref{thm:poslegendrianbraid}.
\begin{theorem}\label{thm:closedposlegendrianbraid}
 Let $\Pat$ be the closure of a positive braid $w\in B_n$, then 
 $
 \overline{\reltb}_V(\Pat)= \textit{length}(w)
 $
 and  $\vert\Leg_{V}(\Pat;\overline{\reltb}_{V}(\Pat))\vert=1$.
 \end{theorem}
\begin{proof}
The computation of the maximal relative Thurston-Bennequin invariant follows from Theorem~\ref{openmaxtb} and the observations in Section~\ref{invariantsinV} about the relation between the invariants in $V=D^2\times S^1$ and $\Rold$. 

For uniqueness take a Legendrian representation $Q$ with maximal relative Thurston--Bennequin number.  As noted above we can cut $V$ open to obtain an open Legendrian braid $Q_{\rm{open}}$ in $\Rold$ representing the braid word $uwu^{-1}$ for some $n$-braid $u$. We know the algebraic length of $uwu^{-1}$ will equal the length of $w$ and thus $\overline{\reltb}_V(Q_{\rm{open}})=\textit{length}(w)- c$, where $c$ is the number of left cusps in the front projection of $Q_{\rm{open}}$. Since we are assuming that $\overline{\reltb}(Q_{\rm{open}})=\textit{length}(w)$ we see that there are no cusps and hence when $Q_{\rm{open}}$ is represented in terms of the basic Legendrian braids from Corollary~\ref{thm:legbraidsimpl} there will be no $\mathcal{Z}$s or $\mathcal{S}$s. From this we see that $uwu^{-1}$ must just be some cyclic permutation of $w$. By choosing the cutting disc differently we can make sure that  the open braid we get from $Q$ represents $w$, and thus as in the proof of Theorem~\ref{thm:poslegendrianbraid} we conclude that $Q$ is the unique Legendrian representation with maximal Thurston--Bennequin number.
\end{proof}
We can give a complete Legendrian classification for 2-braid patterns. This is due to the fact that it is clear when 2--braid patterns  destabilize.
\begin{theorem}\label{2starndbraid}
Let $\Pat_{m}$ be a 2--braid pattern with $m$ (odd) half twists. Then $\Pat_m$ is Legendrian simple. In particular:
\begin{enumerate}
\item If $m> 0$, 
then $\Pat_m$ has a unique Legendrian representative which has maximal relative Thurston--Bennequin number $m$ and rotation number $0$. \item If $m<0$, then $\Pat_m$ has $|m|+1$ representatives with maximal Thurston-Bennequin number $2m$ and with different rotation numbers $\relrot_{V}\in \{-|m|,-(|m|-2),\dots,|m|-2,|m|\}$.\end{enumerate}
\end{theorem}
\begin{proof}
The key observation is that if a 2-braid is represented as a product of basic braids and there is a basic $\mathcal{X}$-braid next to a basic $\mathcal{S}$ or $\mathcal{Z}$-braid then the braid will destabilize. We have already observed that there is a unique maximal Thurston-Bennequin invariant representative for $m>0$ and we now see that all other destabilize. 

For the $m<0$ case we see that all Legendrian representatives destabilize to one represented by a concatenation of basic $\mathcal{S}$ and $\mathcal{Z}$-braids. Moreover there are an odd number of basic braids in this representation and thus there must be two adjacent $\mathcal{S}$ or $\mathcal{Z}$-braids. Noting the isotopy in Figure~\ref{fig:ssz} 
\begin{figure}[ht]
\centering
\begin{overpic}
{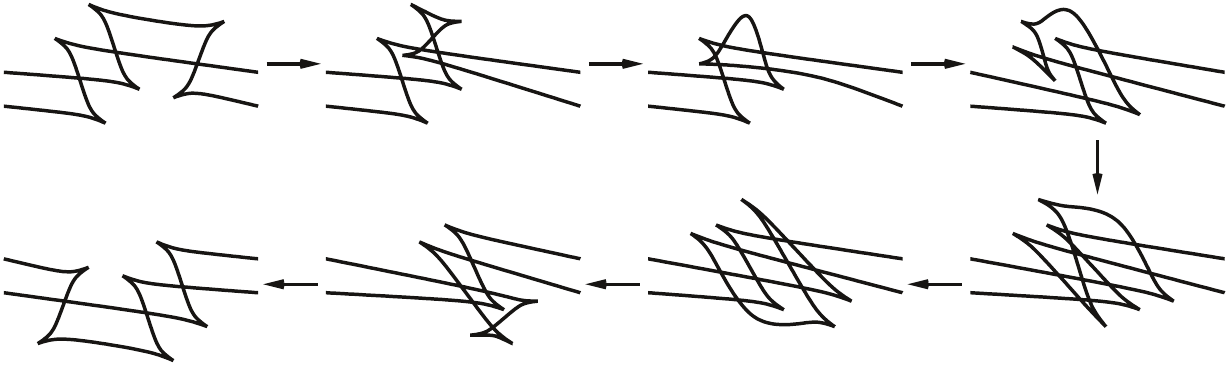}
\end{overpic}
\caption{Legendrian isotopy of a tangle.}
\label{fig:ssz}
\end{figure}
and recalling that we can cyclically permute the basic $\mathcal{S}$ and $\mathcal{Z}$-braids one may easily see that two Legendrian representatives that are written with the same number of $\mathcal{S}$ and $\mathcal{Z}$-braids are isotopic. Moreover if the number of $\mathcal{S}$ and $\mathcal{Z}$-braids used to represent two Legendrian knots is different then their rotation numbers will be different. 
The result now easily follows.
\end{proof}

\subsubsection{Legendrian cable patterns}
Cabling a knot is a satellite operation. To see this let $T\subset V$ be a torus parallel to the boundary of $V$ then let $\mu$ be the primitive element in $H_1(T^2)$ that becomes trivial when included into $V$ and let $\lambda$ be the homology class of $x\times S^1$ for some $x\in D^2$. Then for any relatively prime integers $p$ and $q$ the homology class $p\lambda + q\mu$ can be realized by an embedded curve $C_{p,q}$ on $T$. The knot $C_{p,q}$ represents the knot type $\CC_{p,q}$ that we call the {\em $(p,q)$-cable pattern}. It is clear that for a given knot $\K$ the satellite $\CC_{p,q}(\K)$ is simply the $(p,q)$-cable of $\K$. 
\begin{theorem}\label{cablepats}
Let $p$ and $q$ be relatively prime integers. Then $\CC_{p,q}$ is Legendrian simple. In particular
\begin{enumerate}
\item If $p/q>0$, then there is a unique maximal relative Thurston-Bennequin invariant representative which has  $\reltb_V=pq-p$ and rotation number $0$. 
\item If $p/q<0$, then the maximal relative Thurston-Bennequin invariant is $pq$ and $\Leg_V(\CC_{p,q};pq)$ has $2n+2$ elements with rotation numbers
\[
\{\pm(p+q(n+k))| k=-n, -n+2, \ldots, n-2, n\},
\]
where $n$ is the unique integer such that $-n-1<p/q<-n$.
\end{enumerate}
\end{theorem}
\begin{remark}
Note that this theorem subsumes Lemma~\ref{2starndbraid}, but the proofs are significantly different and the proof of the former demonstrates the utility of the the constructions of Legendrian braids in Theorem~\ref{thm:legbraid} (as the classification of Legendrian Whitehead patters below will too). 
\end{remark}
\begin{proof}
The proof of this theorem follows the proof of Theorem~3.2 and~3.6 in \cite{EtnyreHonda05} almost exactly, so we only sketch the details here. 

We first observe that by the classification of contact structures on solid tori, \cite{Giroux00, Honda00a}, there is a convex $T^2$ in $(V,\xi_V)$ that is parallel to the boundary with dividing slope $r/s$ for any $r/s\leq 0$ and none with dividing slope greater than zero. Now given a Legendrian representative $L$ of $\CC_{p,q}$ with $p/q>0$ we claim that the twisting of $\xi_V$ along $L$ relative to any torus $T^2$ parallel to the boundary of $V$ must be less than zero. If it were not then there would be a Legendrian $L'$ in that knot type with twisting zero. We could then place it on a convex torus that would necessarily have to have dividing slope $p/q$ which is impossible. Knowing that the twisting of $\xi_V$ along $L$ relative to $T$ is negative we can put $L$ on a convex torus. Suppose this torus has dividing slope $r/s\leq0$. Then the twisting of $L$ relative to $T$ is $|rq-ps|$. One may easily see that this is maximized by $-p$ exactly when $r/s=0/1$. Thus any maximal Thurston-Bennequin representative will be a ruling curve on the unique (up to isotopy) convex torus with dividing slope $0$ and thus is itself unique up to isotopy. The relative Thurston-Bennequin invariant differs from the twisting of $\xi_V$ relative to $T$ by $pq$. So the maximal relative Thurston-Bennequin invariant is $pq-p$. One may now easily draw a front diagram for $L$ (using only basic $\mathcal{X}$ braids) and see that the rotation number is $0$. 

Given any $L$ with relative Thurston-Bennequin invariant less than $pq-p$ we can put it on a convex torus with dividing slope less than $0$ and use a convex annulus with one boundary component $L$ and the other a ruling curve on the dividing slope $0$ convex torus to find a bypass for $L$ and destabilize it. Thus all Legendrian knots realizing $\CC_{p,q}$ will destabilize to the one with maximal relative Thurston-Bennequin invariant. 

Turning now to $p/q<0$ one can construct a Legendrian representative of $\CC_{p,q}$ as a Legendrian divide on a a convex torus $T$ parallel to the boundary of $V$. This Legendrian will have relative Thurston-Bennequin invariant $pq$. The proof of Theorem 1.2 in \cite{EtnyreHonda05} shows that the relative Thurston-Bennequin invariant cannot be larger than $pq$. From this it is easy to see that any maximal relative Thurston-Bennequin invariant  representative of $\CC_{p,q}$ is a Legendrian divide on a convex torus. Moreover one can argue that all non-maximal representatives will destabilize to one of these as was done for positive $p/q$. 

To compute the rotation numbers notice that $V$ is a standard neighborhood of a Legendrian core curve $C$ and that any convex torus with dividing slope $p/q$ will be contained between the boundary of a standard neighborhood of $\St^{n_1}_+\St^{n_2}_-(C)$ and $\St_\pm(\St^{n_1}_+\St^{n_2}_-(C))$ where $n=n_1+n_2$ and $n_1,n_2\geq 0$. The computation of the rotation numbers can now be done as in the proof of Lemma~3.8 from \cite{EtnyreHonda05}. Here we indicate the presentation of such curves in terms of $\mathcal{Z}$ and $\mathcal{S}$ braids. Notice that $\|p\|=n\|q\|+e$ where $e>0$. Now any maximal Thurston-Bennequin invariant representative of $\CC_{p,q}$ will be of the form $(Z^q)^{n_1}(S^q)^{n_2}Z^e$ or $(Z^q)^{n_1}(S^q)^{n_2}S^e$. Assuming $p>0$ (and hence $q<0$) we see that $p+nq=e$ and the computation of the rotation numbers is clear. 
\end{proof}

\subsubsection{Legendrian Whitehead patterns} \label{lwp}

Whitehead doubling of a knot is based on the sequence of patterns   $\W_m$ ($m\in\Z$) in $V$ shown in Figure \ref{fig:whiteheadtop}.
\begin{figure}[h]
\centering
\begin{overpic}
{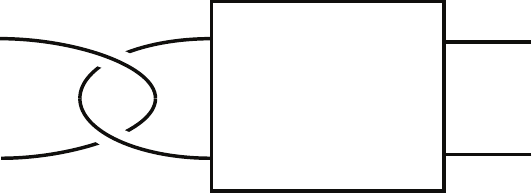}
\put(90,25){$m$}
\end{overpic}
\caption{The Whitehead pattern $\W_m$ (where the box contains $m\in\Z$ half twists).} 
\label{fig:whiteheadtop}
\end{figure}
Note that for these patterns winding number is $0$. In the following we give a complete description of $\Leg_{V}(\W_m)$.
\begin{theorem}\label{whiteheadpatternclass}
Let $\W_m$ be a smooth pattern with $m$ half twists. Then
\begin{enumerate}
\item For $m\ge 0$ even, there are two Legendrian representatives of $\W_m$ with 
\[
(\reltb_{V},\relrot_{V})=(1-m,0).
\]
These two Legendrian patterns become isotopic after a single stabilization (of the same sign). 
All other Legendrian patterns of type $\W_m$ destabilize to one of these two.
\item For $m>0$ odd, there are exactly two Legendrian representatives of $\W_m$ with maximal relative Thurston-Bennequin number, $\reltb_{V} = -m-3$. These representatives are distinguished by their relative rotation numbers $\relrot_{V}=\pm1$ and a negative stabilization of the $\relrot_{V}=1$ pattern is isotopic to a positive stabilization of the $\relrot_{V}=-1$ pattern. All other Legendrian knots destabilize to at least one of these two. In particular, the pattern type is Legendrian simple. 
\item For $m<0$ odd, $\W_m$ has $|m|+1$ Legendrian representatives with 
\[
(\reltb_{V},\relrot_{V})=(-3,0).
\]
All other Legendrian knots destabilize to one of these. After any stabilizations, these $|m|+1$
representatives all become isotopic.
\item\label{item4} For $m< 0$ even, $\W_m$ has $\left(\frac{|m|}2+1\right)^2$ different Legendrian representations with  
\[
(\reltb_{V},\relrot_{V})=(1,0).
\]
All other Legendrian knots destabilize to one of these. These Legendrian knots fall into $\frac{|m|}2+1$ different Legendrian isotopy classes after any given positive number of positive stabilizations, and $\frac{|m|}2+1$ different Legendrian isotopy classes after any given positive number of negative stabilizations. After at least one positive and one negative stabilization (with a fixed number of each), the knots all become Legendrian isotopic.
\end{enumerate}
\end{theorem}
See Figure~\ref{fig:mainthm1} and~\ref{fig:mainthm3} for a schematic picture for the Legendrian mountain range for $\W_m$.
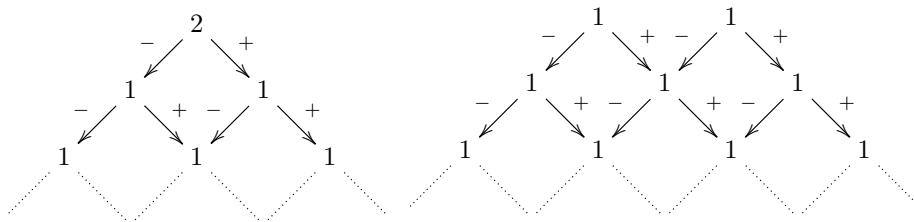
\begin{figure}
\centering
{\small
$
\xymatrixrowsep{2em}
\xymatrixcolsep{2em}
\vcenter{\xymatrix  @dr {
2 \ar[r]^{+} \ar[d]_{-} & 1 \ar[r]^{+} \ar[d]_{-} & 1 \ar@{.}[r] \ar@{.}[d]& \\
1 \ar[r]^{+} \ar[d]_{-} & 1 \ar@{.}[r] \ar@{.}[d]  & &\\
1 \ar@{.}[r] \ar@{.}[d] & & & \\
&&&}}
\vcenter{\xymatrix  @dr {
& 1 \ar[r]^{+} \ar[d]_{-} & 1 \ar[r]^{+} \ar[d]_{-} & 1 \ar@{.}[r] \ar@{.}[d]& \\
1 \ar[r]^{+} \ar[d]_{-} & 1 \ar[r]^{+} \ar[d]_{-}& 1 \ar@{.}[r] \ar@{.}[d]  & &\\
1 \ar[r]^{+} \ar[d]_{-} &1 \ar@{.}[r] \ar@{.}[d] & & & \\
1 \ar@{.}[r] \ar@{.}[d]&&&&\\
&&&&}}
$
}
\vspace{-3cm}
\caption{Schematic Legendrian mountain range for $\W_m$ for $m\geq 0$ even,
  left, and $m>0$ odd, right.  Rotation number is plotted
  in the horizontal direction, Thurston--Bennequin number in the
  vertical direction. The numbers represent the number of Legendrian
  representatives for a particular $(\rot, \tb)$, and the signed arrows
  represent positive and negative stabilization.}
\label{fig:mainthm1}
\end{figure}

\begin{figure}
\centering
{\small
$
\xymatrixrowsep{2em}
\xymatrixcolsep{2em}
\vcenter{\xymatrix  @dr {
|m|+1 \ar[r]^{+} \ar[d]_{-} & 1 \ar[r]^{+} \ar[d]_{-} & 1 \ar@{.}[r] \ar@{.}[d]& \\
1 \ar[r]^{+} \ar[d]_{-} & 1 \ar@{.}[r] \ar@{.}[d]  & &\\
1 \ar@{.}[r] \ar@{.}[d] & &  & \\
&&&&
}}
\vcenter{\xymatrix  @dr {
 \left(\frac{|m|}2+1\right)^2 \ar[r]^{+} \ar[d]_{-} &  \left(\frac{|m|}2+1\right)  \ar[r]^{+} \ar[d]_{-}  &  \left(\frac{|m|}2+1\right)  \ar@{.}[r] \ar@{.}[d]& \\
 \left(\frac{|m|}2+1\right) \ar[r]^{+} \ar[d]_{-} &  1 \ar@{.}[r] \ar@{.}[d]  & &\\
 \left(\frac{|m|}2+1\right) \ar@{.}[r] \ar@{.}[d] &  & & \\
&&&&}}
$}
\vspace{-3cm}
\caption{Legendrian mountain range for $\W_m$, $m$ odd and negative, left, and even and negative, right.}
\label{fig:mainthm3}
\end{figure}
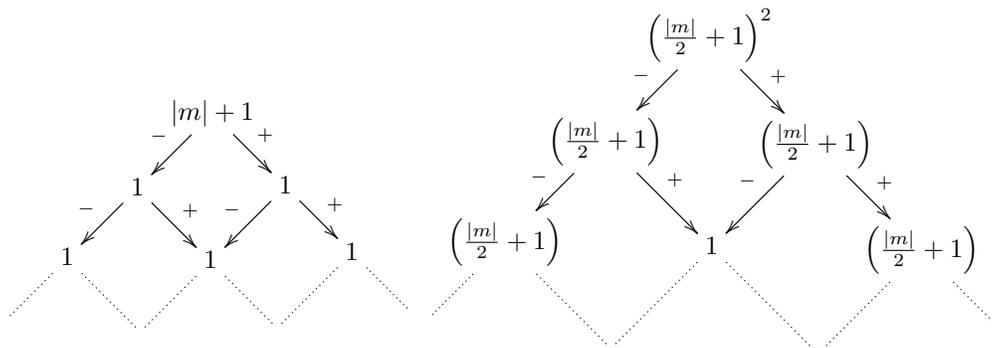

The proof follows the authors paper with Ng \cite{EtnyreNgVertesi13}.
The method of \cite{EtnyreNgVertesi13} for putting a Legendrian representation of twists knots in $(S^3,\xi_\textrm{st})$ in a standard form only uses that the contact structure is tight and the existence of certain discs in some decompositions of the knot complement. The contact structure we  consider is still tight, and all the discs still exist. More specifically, after showing that any Legendrian twist knot in $S^3$ has a Legendrian unknot with $tb=-1$ linking it in a particular way, the rest of the classification in \cite{EtnyreNgVertesi13} is done by analyzing Legendrian patterns in the smooth type of $\W_m$. For a Legendrian pattern that does not destabilize, this is done by finding two meridional disks in $V$ that cut $V$ into two copies of $\Rold$, one containing a 2-braid representing the $m$ twists in $\W_m$ and one containing the clasp shown in Figure~\ref{fig:whiteheadtop}. The Legendrian knots representing the 2-braid are classified in Theorem~\ref{thm:legbraid} above and the Legendrian representative for the clasp is understood in Section~4.2 of \cite{EtnyreNgVertesi13}. The proofs there cary over to our case (more or less verbatim) yielding the following result. 
\begin{theorem}\label{citethm}
Any Legendrian representation of $\W_m$ either destabilizes or isotopic to the Legendrian patterns depicted in Figure~\ref{fig:Legtwistpattern}.\qed
\end{theorem}
\begin{figure}[h]
\centering
\begin{overpic}
{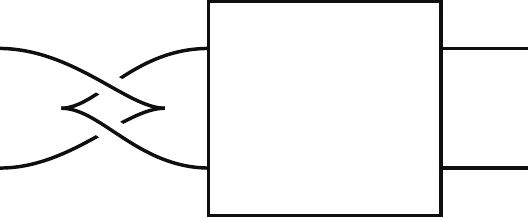}
\put(89,29){$m$}
\end{overpic}
\caption{Legendrian representatives of the Whitehead pattern $\W_m$ with maximal Thurston-Bennequin number. The box contains $m$ copies of the basic ${X}$-braid if $m\ge0$, and  $\vert m\vert$ copies of the basic  ${S}$ and ${Z}$-braids if $m<0$.}
\label{fig:Legtwistpattern}
\end{figure}
\begin{proof}[Proof of Theorem~\ref{whiteheadpatternclass}]
We start with the case when $m<0$, so the box in Figure~\ref{fig:Legtwistpattern} contains $\vert m\vert$ basic ${\SS}$ and ${\ZZ}$-braids in arbitrary order. Depending on the orientation of the strands there are two types of ${\SS}$s and ${\ZZ}$s as depicted on Figure~\ref{fig:zsplusminus}. 
\begin{figure}[h]
\centering
\includegraphics
{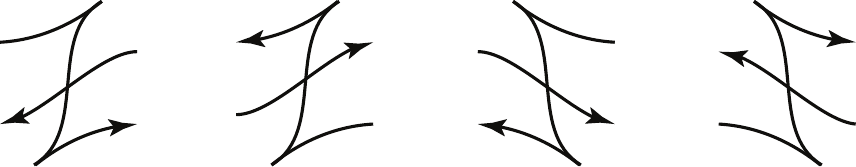}
\caption{The Legendrian braids ${\ZZ}^+$, ${\ZZ}^-$, ${\SS}^+$ and ${\SS}^-$, respectivelly}
\label{fig:zsplusminus}
\end{figure}
Thus the content of the box can be described by a word $w$ of length $\vert m\vert$ in
${\ZZ}^+$, ${\ZZ}^-$, ${\SS}^+$ and ${\SS}^-$, with alternating signs.
Let $\Q_w$ denote the Legendrian isotopy class corresponding to the word $w$.
In Figure~\ref{fig:ssz} it was shown that if $w'$ is obtained from $w$ by exchanging  an appearance of ${\ZZ}^\pm{\SS}^\mp{\SS}^\pm$ for ${\SS}^\pm{\SS}^\mp{\ZZ}^\pm$ or similarly exchanging  ${\ZZ}^\pm{\ZZ}^\mp{\SS}^\pm$ for ${\SS}^\pm{\ZZ}^\mp{\ZZ}^\pm$ then $\Q_w=\Q_{w'}$. 

Now consider the case when $m$ is even. The isotopies in Figure~\ref{fig:ssz} show that any two consecutive $+$ letters can be exchanged and similarly that any two consecutive $-$ letters can be exchanged. Thus the Legendrian pattern determined by a word $w$ is determined by the number of  ${Z}^+$s, ${Z}^-$s, ${S}^+$s and ${S}^-$ in the word --- which we denote by $z^+(w),z^-(w),s^+(w)$, and $s^-(w)$, respectively --- and whether or not the word begins with a $+$ or $-$ letter.  Figure~\ref{fig:movez} shows that $\Q_{w'{S}^\pm}=\Q_{{S}^\pm w'}$ and similarly reflecting the diagram about a horizontal axis shows that $\Q_{w'\mathcal{Z}^\pm}=\Q_{\mathcal{Z}^\pm w'}$. Thus we can always assume that the word begins with, say, a $+$ letter. 
\begin{figure}[h]
\centering
\begin{overpic}
{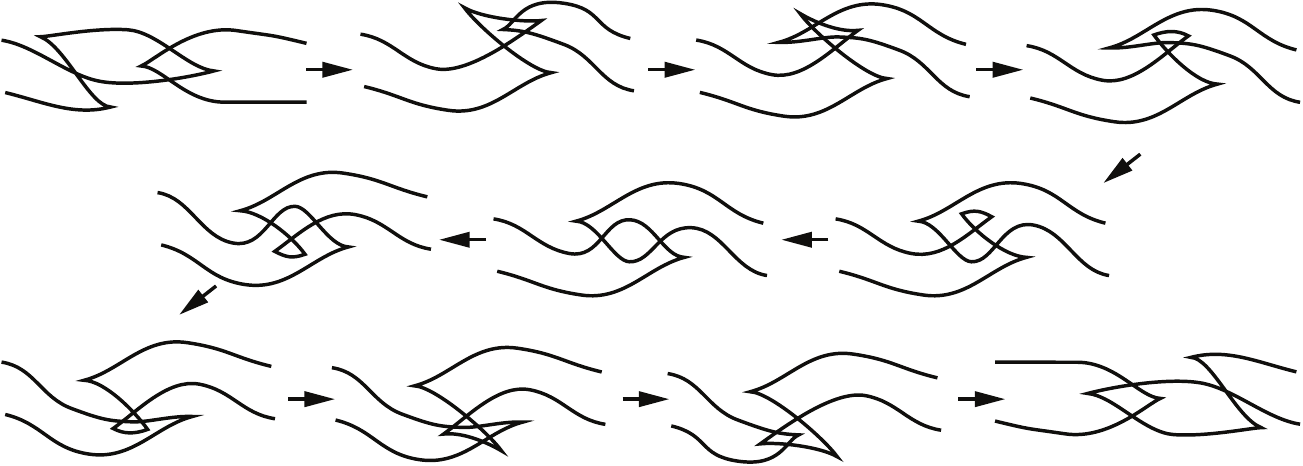}
\end{overpic}
\caption{The sequence of front diagrams showing how to move an $\mathcal{S}$ from one side of the cusp to the other. In other words that $\Q_{w'{S}^\pm}=\Q_{{S}^\pm w'}$.}
\label{fig:movez}
\end{figure}
We conclude that the Legendrian isotopy class of $\Q_w$ can be described by the quadruple of numbers $(z^+,z^-,s^+,s^-)=(z^+(w),z^-(w),s^+(w),s^-(w))$, and we have the relations 
\[
z^++s^+=\frac{\vert m\vert}2 \quad \text{and}\quad  z^-+s^-=\frac{\vert m\vert}2.
\] 
Thus the quadruple can be replaced by the pair $(z^+,z^-)$ and all Legendrian representations of the pattern $\W_m$ fall into the isotopy classes $\Q_{(z^+,z^-)}$ determined by two numbers $z^+$ and $z^-$ that lie between 0 and $\frac{\vert m\vert}2$.

Now notice that the standard Legendrian unknot $\mathcal{U}$ has Thurston-Bennequin number $-1$, thus for a Legendrian representative $\Q$ of $\W_m$ the satellite $\Q(\mathcal{U})$, as described in Section~\ref{sec:braidsatellite}, is a Legendrian representative of a twist knot with $(m-2)$ half twists. 
By Theorem 1.1 of \cite{EtnyreNgVertesi13} $\Q_{(z^+,z^-)}(\mathcal{U})$ is different from $\Q_{(\tilde z^+,\tilde z^-)}(\mathcal{U})$ unless $(\tilde z^+,\tilde z^-)=(z^+,z^-)$ or $(m-1-z^+,m-1-z^-)$. Similarly the knots $\Q_{(z^+,z^-)}(\St_+\mathcal{U})$ and $\Q_{(\tilde z^+,\tilde z^-)}(\St_+\mathcal{U})$ are distinct unless $(\tilde z^+,\tilde z^-)=(z^+,z^-)$ or $(m-z^+,m-z^-)$. Thus we see that all the $\Q_{(z^+,z^-)}$ must be distinct as patterns.

This establishes Item~(\ref{item4}) of the theorem for the maximal Thurston-Bennequin examples. For the non-maximal examples we notice that Theorem~\ref{citethm} and the proof of Lemma~\ref{2starndbraid} imply that a non-maximal Thurston-Bennequin invariant pattern will destabilize. We now note that in Figure~\ref{fig:szstab} it is shown that the stabilized knot $\St^\pm\Q_w$ is Legendrian isotopic to $\St^\pm\Q_{w'}$, where $w'$ is obtained from $w$ by exchanging any $\mathcal{\ZZ}^\pm$ for $\mathcal{\SS}^\pm$ or vice versa.
\begin{figure}[h]
\centering
\includegraphics
{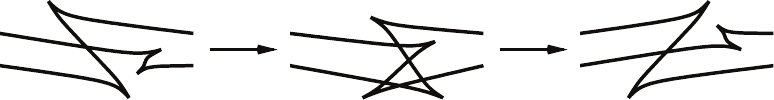}
\caption{Legendrian isotopy of a tangle.}
\label{fig:szstab}
\end{figure}
Thus to complete Item~(\ref{item4}) of the theorem we merely need to notice that by embedding the solid torus $V$ as a neighborhood of the maximal Thurston-Bennequin unknot in $S^3$ the classification of twist knots in \cite{EtnyreNgVertesi13} implies that $\St^\pm\Q_w$ is determined by $z^\pm$. 

For the $m<0$ odd case the proof is similar, but simpler. The relation of Figure~\ref{fig:szstab} translates to $\Q_{w\mathcal{Z}^\pm}=\Q_{\mathcal{Z}^\mp w}$ and $\Q_{w\mathcal{S}^\pm}=\Q_{\mathcal{S}^\mp w}$. Since $m$ is odd, this shows that $\Q_{w}=\Q_{w'}$ whenever they have the same number of $\mathcal{Z}$s and $\mathcal{S}$s. Let $\Q_z$ denote the Legendrian isotopy classes with $z$, $\ZZ$s (and $m-z$, $\mathcal{S}$s). Write $m=-2n+1$.
Then Theorem 1.1 of \cite{EtnyreNgVertesi13} says that $\Q_z(\U)$ and $\Q_{z'}(\U)$
are different for $z\neq z' \text{ or } n-z'$. To distinguish $\Q_z$ from $\Q_{n-z}$ we need to consider the satellies $\Q_z(\St_+ \U)=\Q_{z+2}(\U)$. Here we see that $\Q_z(\St_+ \U)\neq \Q_{z'}(\St_+ \U)$ whenever $z\neq z'\text{ or } n+1-z'$.  The $m<0$ odd case is now complete by noting that maximal Thurston-Bennequin invariant patterns have been classified and the non-maximal Thurston-Bennequin invariant patters are dealt with in the discussion of stabilization above.

For $m\ge 0$ the box of Figure~\ref{fig:Legtwistpattern} can be uniquely filled with $m$ copies of the basic ${\XX}$-braid, thus depending on the orientation of the strands there can be two Legendrian representatives of $\W_m$ with maximal Thurston-Bennequin number. For $m$ odd these representatives are distinguished by their relative rotation number. For $m$ even the representatives have the same relative rotation number, but can be distinguished using contact homology, see for example Proposition 5.11 of \cite{NgTraynor04}. After any stabilization the knot is isotopic to one that has a $\mathcal{Z}$ or $\mathcal{S}$ in it, thus the orientation of the cusp can be changed, and the knots become Legendrian isotopic.
\end{proof}

\subsection{Transverse patterns}
\label{sec:transversebraidsind2xi}
In this section we discuss transverse representations of patterns in $(V,\xi_V)$. Let $\Pat$ be a smooth pattern in $V$, and let $R$ be a transverse representative of $\Pat$ with transverse isotopy class $\RR$. The set of transverse isotopy classes of $\Pat$ is denoted by $\Trans(\Pat)$. In \cite{EtnyreHonda01b} it was shown that the relation between Legendrian representations and transverse representations is local.
\begin{theorem}[Etnyre and Honda 2001, \cite{EtnyreHonda01b}]\label{relation}
The transverse patterns $R$ and $R'$ are transverse isotopic if and only if they have Legendrian approximations $Q$ and $Q'$ that have Legendrian isotopic negative stabilizations.\qed
\end{theorem}
One can also adapt the proof in \cite{EpsteinFuchsMeyer01} to give an alternate proof of this result.  
The above theorem implies that the relative self-linking number can be defined using the relative invariants for Legendrian knots in $(V,\xi_V)$:
\[\textrm{relsl}_{V}(\mathcal{R})=\textrm{reltb}_{V}(\Q)-\textrm{relrot}_{\xi_V}(\Q),\]
where $\Q$ is any Legendrian approximation of $\RR$.
A pattern $\Pat$ is called \emph{transverse simple} if its transverse representatives are distinguished by their relative self-linking numbers. 

Recall from Subsection~\ref{reembed} that $(V,\xi_V)$ can be embedded into itself as a standard neighborhood of a stabilization of its Legendrian core, thus a negative full twist of a pattern can be defined for transverse patterns, just as they were for Legendrian patterns in that subsection, to which we refer for the notation used below. The image of the trivial transverse pattern with $n$ horizontal strands under $\sigma$ and $\zeta$ are the transverse push-offs of $\mathcal{S}^n$ and $\mathcal{Z}^n$, respectively.

\subsubsection{Transverse braid patterns}
Using the results for Legendrian braids patterns in $(V,\xi_V)$ we can classify transverse braids in a solid torus. 
\begin{theorem}\label{thm:transversebraid}
Let $\Pat$ be a braid pattern in $V$, and let $\RR\in \Trans({\Pat})$ be a transverse isotopy class in $(V,\xi_V)$. Then  $\RR$ can be represented as the closure of some concatenation of the basic transverse braids shown in Figure~\ref{fig:opentransversebraidsimpl}. \end{theorem}
\begin{figure}[h]
\centering
\begin{overpic}
{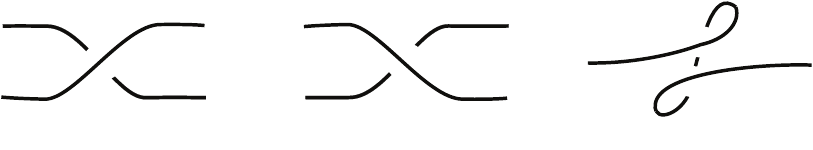}
\put(22,0){$\mathcal{X}_{\textrm{tr}} $}
\put(112,0){$\mathcal{S}_{\textrm{tr}} $}
\put(202,0){$\St_{\textrm{tr}}$}
\end{overpic}
\caption{Basic transverse braids (the strands not depicted are assumed to be horizontal).}
\label{fig:opentransversebraidsimpl}
\end{figure}
\begin{proof}
Take a Legendrian approximation $Q$ of the representative $R$ of $\RR$. From Corollary~\ref{thm:legbraidsimpl}, $Q$ is built up from basic Legendrian braids depicted in Figure~\ref{fig:openbraidsimpl}. 
Thus $\RR$ is the transverse push off of this concatenation, which means it is transverse isotopic to the transverse push offs of the sequence of the basic Legendrian braids. Figures~\ref{fig:destabtransvers} and~\ref{fig:s1ktr} show that the transverse push offs can be further simplified to the basic transverse braids of Figure~\ref{fig:opentransversebraidsimpl}.
\begin{figure}[h]
\centering
\begin{overpic}
{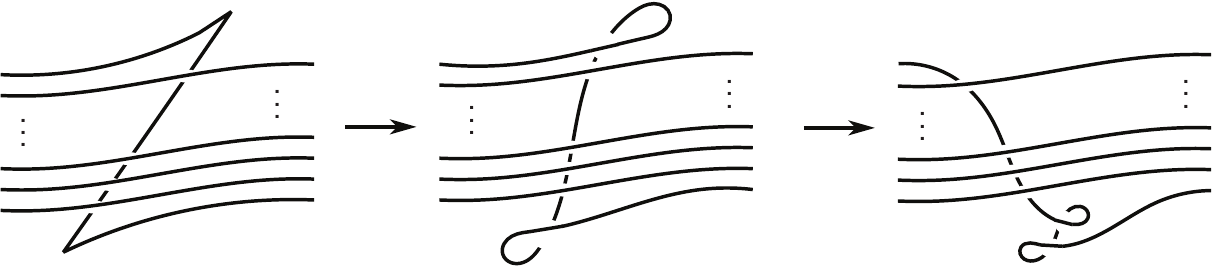}
\end{overpic}
\caption{The transverse push off of $Z(1,n-1)$ destabilizes.}
\label{fig:destabtransvers}
\end{figure}
\begin{figure}
\centering
\begin{overpic}
{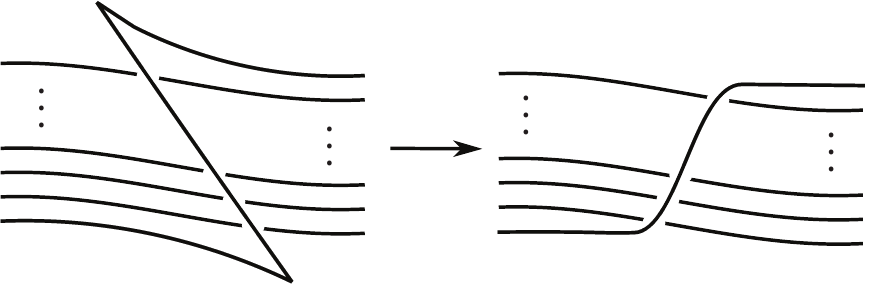}
\end{overpic}
\caption{The transverse push off of $S(1,n-1)$.}
\label{fig:s1ktr}
\end{figure}
\end{proof}
Unlike Legendrian braids, transverse braids are easy to understand.
\begin{theorem}\label{thm:transversed2s1}
Any braid pattern $\Pat$ in $V$ is transversally simple in $(V,\xi_V)$ and the maximal self-linking number is 
\[\overline{\textrm{relsl}}_{V}(\Pat)=\textit{writhe}(\Pat).\]
\end{theorem}
\begin{proof} 
Let $\RR$ be a transverse patters representing some braid $\Pat$. 
According to Theorem~\ref{thm:transversebraid} a pattern $\RR$ can be written as a word in the basic transverse braids in Figure~\ref{fig:opentransversebraidsimpl}.  If we denote by $\textit{writhe}(\RR)$ the number of $\XX_{\textrm{tr}}$ minus the number of $\SS_{\textrm{tr}}$ then one easily computes that the relative self-linking number of $\RR$ is $\textit{writhe}(\RR)-\# \St_{\textrm{tr}}$. Moreover, any word in the $\mathcal{S}_{\mathrm{tr}}$ and $\mathcal{X}_{\mathrm{tr}}$ corresponds directly to a braid word in the braid group and vice versa. So if we ignore the $\St_{\textrm{tr}}$ in the expression of $\RR$ as a word in the basic transverse braids then we get a braid word $w$ in the topological braid group representing $\Pat$. Examining the braid relations one sees that for any other braid word $w'$ representing $\Pat$ we have that $\textit{writhe}(w)=\textit{writhe}(w')$ and thus we can talk about $\textit{writhe}(\Pat)$. We have shown that any transverse pattern $\RR$ representing  $B$ satisfies
\[
\textrm{relsl}_{\xi_V}(\RR)\leq \textit{writhe}(\Pat).
\]

Representing $\RR$ as a word in the basic transverse braids in Figure~\ref{fig:opentransversebraidsimpl}, we see that it has maximal self-linking number unless there are some $\St_{\textrm{tr}}$ terms. If there are such terms we can clearly destabilize $\RR$. So any transverse braid with non maximal self-linking number destabilizes. 

To see that two transverse braids with maximal self-linking number are transversely isotopic recall that a word in the basic transverse braids $\mathcal{S}_{\mathrm{tr}}$ and $\mathcal{X}_{\mathrm{tr}}$ corresponds to a word in the standard generators of the braid group and vice versa. In addition, all braid relations in the braid group correspond to transverse isotopies of the associated transverse braids. Similarly conjugations of the braid word representing the braid are also transverse isotopies in $V$. Thus the topological types of a maximal self-linking number transverse braid determines the braid up to isotopy. 
\end{proof}

Note that a full twist $\Delta$ is naturally a transverse braid. For $R$ a transverse knot in $\RR\in\Trans_{V}(\Pat)$ take an arbitrary opening $R_{\textrm{open}}$ with a convex disc that intersects $R$ in $n$ points. And let $\Delta R_{\textrm{open}}$ be the concatenation of $\Delta$ and $R_{\textrm{open}}$.
Choosing a different cutting disc will not change the transverse isotopy class of the resulting closed braid $\Delta\RR$. Here we are using the fact that any two opening of $R$ are related to conjugation with elements in the braid group and these conjugations do not change the result. The operation $\Delta$ on transverse braids is the inverse of applying the map $\sigma^1$ defined in Subsection~\ref{reembed} to a braid. 
\begin{lemma}\label{lem:deltastab}
For any transverse braid $\RR\in\Trans(\Pat)$ we have
\[\Delta(\sigma^1(\RR))=\sigma^1(\Delta\RR)=\RR.\]\qed
\end{lemma}

\subsubsection{Transverse Whithead patterns}
As a consequence of the Legendrian classification of Whitehead patterns and the relation between Legendrian and transverse knots recalled in Theorem~\ref{relation} we can classify transverse Whitehead patterns. 
\begin{theorem}
Let $\W_m$ be a smooth pattern with $m$ half twists. Then
\begin{enumerate}
\item For $m\ge 0$ or $m<0$ odd, $\W_m$ is transversely simple. Moreover the maximal relative self-linking number equals 
\[
\relsl_{V}=
\begin{cases} 
1-m & m\ge 0 \text{ is even,}\\ 
-m-2 & m>0 \text{ is odd,}\\ 
-3 & m<0  \text{ is odd.} 
\end{cases}
\]
\item For $m<0$ even, $\W_m$ has $\frac{|m|}2+1$ transverse representatives with maximal self-linking number $\relsl_{V}=1$, and any nontrivial stabilization of these maximal representatives is isotopic (for a fixed number of stabilizations).\qed
\end{enumerate}
\end{theorem}

\section{Satellites}
\label{sec:braidsatellite}

Topological satellite knots were discussed in Subsection~\ref{subsec:satknot} and we refer the reader to that section to review notation and terminology. In this section we will begin by giving a general discussion of Legendrian satellite knots and what one can say about them in terms of the underlying pattern and companion knots. We then discuss the specific examples of Legendrian braids, cables and Whitehead doubles. We end this section with a discussion of transverse satellite knots. 

\subsection{Legendrian satellites}\label{subsec:legsatellites}
Let $\Lc\in \Leg(\K)$ be a Legendrian isotopy class of the smooth knot type $\K$. Choose a specific Legendrian knot $L$ representing $\Lc$ and a standard contact neighborhood $\nu(L)$ of $L$.
Denote by  $\tau$ the Thurston--Bennequin framing for $L$. Note that $\tau=\lambda+\tb(L)\mu$, where $\lambda$ and $\mu$ give the standard longitude-meridian basis for $\partial \nu(L)$ as discussed in Subsection~\ref{subsec:satknot}. Let $Q$ be a Legendrian pattern representing an isotopy class $\Q\in\Leg_{V}(\Delta^{-\tb(\Lc)}\Pat)$.  Denote by $\psi$ the unique (up to contact isotopy) contactomorphism from $(V,\xi_V)$ to $(\nu(L),\xi\vert_{\nu(L)})$. The contactomorphism $\psi$ necessarily maps $\Gamma_{\partial V}$ to $\Gamma_{\nu(L)}$, thus the product framing $\lon$ on $V$ is mapped to $\tau$. Denote the image of $Q$ under the contactomorphism $\psi$ by $Q(L)$. Then $Q(L)$ is a Legendrian knot in the smooth class 
\[
(\Delta^{-\tb(L)}\Pat)_{\tau}(\K)=(\Delta^{-\tb(L)}\Pat)_{\lambda+\tb(L)\mu}(\K)=\Pat_\lambda(\K)=\Pat(\K).
\]
A simple application of Lemma~\ref{isoIScontacto}, and the fact that standard neighborhoods of Legendrian knots are well-defined up to contact isotopy, shows that $Q(L)$ is well defined up to Legendrian isotopy independent of the choice of $L\in\Lc$ and $Q\in \Q$.
\begin{lemma} 
Let $L_0$ and  $L_1$ be two Legendrian knots in the Legendrian isotopy class $\Lc$, and choose contactomorphisms $\psi_0\colon (V,\xi_V)\to (\nu(L_0),\xi\vert_{\nu(L_0)})$ and $\psi_1\colon (V,\xi_V)\to (\nu(L_1),\xi\vert_{\nu(L_1)})$ for some  standard contact neighborhoods $\nu(L_0)$ and $\nu(L_1)$. Let $Q_0$ and  $Q_1$ be two Legendrian patters in the Legendrian isotopy class $\Q\in\Leg_{V}(\Delta^{-\tb(\Lc)}\Pat)$. Then $Q_0(L_0)$ and $Q_1(L_1)$ are Legendrian isotopic. \qed
\end{lemma}
This means that the Legendrian class $\Q(\Lc)$ is well-defined.
The classical invariants of Legendrian satellites  can be computed as follows.
\begin{lemma}\label{lem:legbraidsat} 
Let $n$ denote the winding number of the Legendrian pattern $\Q$. Then 
\[
\tb(\Q(\Lc))=n^2\cdot \tb(\Lc)+\reltb_{V}(\Q)
\]
and
\[
\rot(\Q(\Lc))=  n\cdot \rot(\Lc)+\relrot_{V}(\Q).
\]
\end{lemma}
\begin{proof} 
These formulas were first observed in  \cite{NgThesis} (see also \cite{NgTraynor04}) where it was noted that they follow easily from front diagrams. For the sake of completeness 
we give the computation here.  

A front projection for a Legendrian representative of the Legendrian satellite $\Q(\Lc)$ can be constructed as follows. Take a Legendrian representative  $L\in \Leg(\K)$, that has a straight horizontal segment $c$ pointing from left to right in its front projection $\pi(L)$ with a neighborhood $N(c)\cong[-\varepsilon,\varepsilon]\times c$ that intersects $\pi(L)$ only in $c$. Pick an open pattern $Q_{\textit{open}}$ with closure $Q \in \Leg_{V}(\Delta^{-\tb(\Lc)}\Pat)$, and which intersects the boundary in $w$ points. Insert a front projection of $Q_\textit{open}$ into $N(c)$ and add oriented parallel copies of $\pi(L)-c$ that match $Q_\textit{open}$ at $[-\varepsilon,\varepsilon]\times\partial c$. See Figure \ref{fig:satex}. 
\begin{figure}[h]
\small
\begin{overpic}
{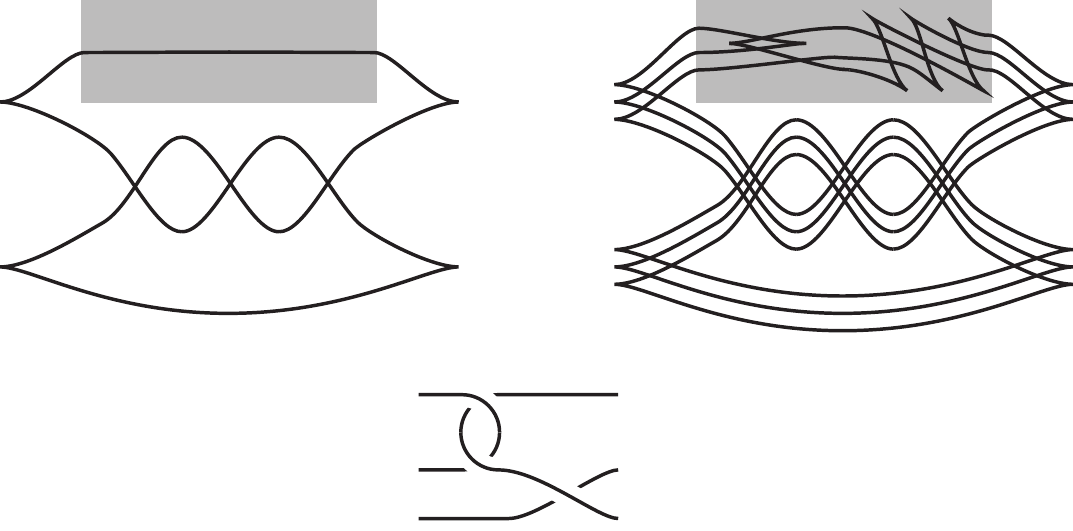}
\put(1,143){$N(c)$}
\put(147,44){$\Pat$}
\end{overpic}
\caption{A Legendrian knot $L$ depicted on the left together with the neighborhood $N(c)$ shaded. The (open) satellite pattern $\Pat_\textit{open}$ is shown in the middle. On the right is a satellite where the open pattern is inserted in the shaded region. Notice that since $\tb(L)=1$ the open Legendrian pattern represents $\Delta^{-1}\Pat$.}
\label{fig:satex}
\end{figure}
Let $w_+$ and $w_-$ denote the number of positively and negatively orieneted copies of $L$. Then the total wrapping number is $w=w_++w_-$ and the winding number is $n=w_+-w_-$. 

In this projection of $Q(L)$, the number of cusps can be explicitly computed. 
\begin{align*}
u(Q(L))&=w_+\cdot u(L)+w_-\cdot d(L)+u(Q),\\
d(Q(L))&=w_+\cdot d(L)+w_-\cdot u(L)+d(Q),\\
c(L)&=w\cdot c(L)+c(Q).
\end{align*}
Thus the rotation number $\rot(Q(L))=\frac{1}2\left(d(Q(L))-u(Q(L))\right)$ is indeed
\[\frac12 \left(w_+\cdot\left(d(L)-u(L)\right)-w_-\cdot\left(d(L)-u(L)\right)+d(Q)-u(Q)\right)=n\cdot\rot(L)+\relrot_V(Q)\]
Additionally to the usual notation let $x_+$, respectively $x_-$, denote the number of positive, respectively negative, crossings of a projection. Then
\begin{align*}
x_+(Q(L))=&
(w_+w_-)c(L)+(w_+^2 +w_-^2)x_+(L)+(2w_+w_-)x_-(L)+x_+(Q), \\
x_-(Q(L))=&
\left(\binom{w_+}{2} +\binom{w_-}{2}\right)c(L)\\
&\qquad +(2w_+w_-)x_+(L)+(w_+^2+w_-^2)x_-(L)+x_-(Q). 
\end{align*}
The writhe of $Q(L)$ is
\begin{align*}
\mathit{writhe}(Q(L))&=x_+(Q(L))-x_-(Q(L))\\&=\frac{w-n^2}2 c(L)+n^2(x_+(L)-x_-(L))+x_+(Q)-x_-(Q).
\end{align*}
And then 
\[
\tb(Q(L))=\mathit{writhe}(Q(L))-\frac12c(Q(L))=n^2\cdot \tb(L)+\reltb_{V}(Q).
\]
\vskip -.3in
\end{proof} 

To further understand  Legendrian satellites 
we recall from \cite{EtnyreHonda05} that a knot type $\K$  is \emph{uniformly thick}
if all solid tori representing $\K$ can be contained in another solid torus that is a
non-thickenable, standard neighborhood of a Legendrian representative $L$  of $\K$ with maximal Thurston--Bennequin number.
In the following we use the notation 
\[
\tt=\overline{\tb}(\K)
\]
for the maximal Thuston-Bennequin number of Legendrian knots in $\Leg(\K)$. 

For a while we only need to work with a weaker assumption. A knot is called \emph{thickenable} if all solid tori representing $\K$ can be contained in another solid torus that is the standard neighborhood of a Legendrian representative of $\K$.

\begin{lemma}\label{lem:utpsatelite}
Suppose $\K$ is thickenable and let $\Pat$ be a pattern in $V$. Then any element of  $\Leg(\Pat(\K))$
can be written as $\Q(\Lc)$, where $\Lc$ is a Legendrian representation of $\K$ and $\Q\in (\Delta^{-\tb(\Lc)}\Pat)$.

Moreover if $\Lc$ is uniformly thick then we can assume that $\Lc\in\Leg(\K;{\tt})$ and $\Q\in\Leg_{V}(\Delta^{-\tt}\Pat)$.
\end{lemma}
\begin{proof} 
Take a Legendrian representative $S$ of $\Q(\Lc)$. Smoothly $S$ is the satellite $\Pat(\K)$ of $\K$ and thus it is contained in a solid torus $T$ that represents $\K$. Since $\K$ is thickenable this solid torus can be thickened to a standard neighborhood $\nu(L)$, where $L$ is a Legendrian representation of $\K$. The solid torus $(\nu(K),\xi_{\rm{st}}\vert_{\nu(K)})$ is contactomorphic to $(V,\xi_V)$ via a contactomorphism $\psi$ that brings $\mathbf{l}$ to the contact framing $\lambda+\tb(\Lc)\mu$. Then $Q=\psi^{-1}(S)$ is a Legendrian pattern in $(V,\xi_V)$ that smoothly represents $\Delta^{-\tb(\Lc)}\Pat$.

The proof for the uniformly thick case is identical.
\end{proof}

For any satellite $\Pat(\K)$ we have $\overline{\tb}(\Pat(\K))\ge n^2\tt+\overline{\reltb_{V}}(\Delta^{-\tt}\Pat)$. A consequence of the previous lemma is that for uniformly thick knot types the above inequality is an equality.
\begin{corollary}\label{cor:utpsatmaxtb}
Suppose $\K$ is uniformly thick, and $\Pat(\K)$ is the satellite with pattern $\Pat$. Then
\[\overline{\tb}(\Pat(\K))=n^2\tt+\overline{\reltb_{V}}(\Delta^{-\tt}\Pat). \]
\vskip -.3in \qed
\end{corollary}

For simplicity in the following we will assume that $\K$ is Legendrian simple and thickenable. 

By Lemma \ref{lem:utpsatelite} for a thickenable knot type $\K$ the map
\[
\widetilde\Sat\co \bigcup_{t\in \Z} \left(\Leg_V(\Delta^{-t}\Pat)\times \Leg(\K; t)\right)\to \Leg(\Pat(\K))\co (\Q,\Lc)\mapsto \Q(\Lc)
\]
is surjective. 

If $\K$ is uniformly thick then we can work with a more trackable surjective map
\[
\widetilde\Sat'\co \left(\Leg_V(\Delta^{-\tt}\Pat)\times \Leg(\K;\tt)\right) \to \Leg(\Pat(\K))\co (\Q, \Lc)\mapsto \Q(\Lc).
\]
One can define stabilization of pairs by $\St_\pm(\Q,\Lc)=(\St_\pm\Q,\Lc)$ which makes $\widetilde\Sat$ and $\widetilde\Sat'$ $\St_\pm$ equivariant maps.

Stabilization on the $\K$ component is reflected by $\sigma$ or $\zeta$ on the $\Pat$-component: 
\begin{align*}
\widetilde\Sat(\Q,\St_+\Lc)&=\widetilde\Sat(\zeta\Q,\Lc);\\
\widetilde\Sat(\Q,\St^-\Lc)& =\widetilde\Sat(\sigma\Q,\Lc).
\end{align*}
Thus to make $\widetilde\Sat$ injective we need to factor out with an equivalence relation containing $(\Q,\St_+\Lc)\sim (\zeta\Q,\Lc)$ and $(\Q,\St_-\Lc)\sim (\sigma\Q,\Lc)$.  If there are no oriented topological symmetries then the equivalence relation of Theorem \ref{mainsat} is what follows from the above observation, if we have oriented topological symmetries then we need to include them in the relation too. 

\begin{definition}\label{def:kergen} Let $\K$ be a thickenable knot type, and $\Pat$ be a pattern in $V$. If the winding number of $\Pat$ is not zero then for $\Lc,\Lc'\in \Leg(\K)$ and patterns $\Q\in \Leg_V(\Delta^{-\tb(\Lc)}\Pat)$, $\Q'\in \Leg_V(\Delta^{-\tb(\Lc')}\Pat)$, the pairs 
$(\Lc,\Q)\sim (\Lc',\Q')$  if and only if there is a sequence 
\[
(\Q,\Lc)=(\Q_0,\Lc_0),(\Q_1,\Lc_1),\dots,(\Q_k,\Lc_k)=(\Q',\Lc'),
\] 
where either
\begin{itemize}
\item $\Lc_i$ is a stabilisation of $\Lc_{i-1}$ and depending on the sign of the stabilisation $\Q_i$ is either $\zeta \Q_{i-1}$ (if $\Lc_i=\St_+\Lc_{i-1}$) or $\sigma Q_{i-1}$ (if $\Lc_i=\St_-\Lc_{i-1}$); or
\item $\Lc_i$ is a destabilisation of $\Lc_{i-1}$ and depending on the sign of the destabilisation $\Q_{i-1}$ is either $\zeta Q_{i}$ (if $\Lc_{i-1}=\St_+\Lc_i$) or $\sigma Q_{i}$ (if $\Lc_{i-1}=\St_+\Lc_i$).
\end{itemize}
If $\Pat$ has winding number zero, $-\K=\K$, and $\Pat(\K)$ has oriented topological symmetries as defined at the end of Subsection~\ref{subsec:satknot} then we define the same equivalence relation except that at the first step in the sequence we also allow, but do not require, $\Lc_{1}=-\Lc_0$ and $\Q_{1}=f(\Q_0)$. (The map $f$ is defined at the end of Section~\ref{subsec:satknot} and is a contactomorphism of $V$.) 
\end{definition}
Then by the above observation the map $\widetilde\Sat$ descends to an $\St_\pm$ equivariant map on the quotient 
\[
{\Sat}\co   \left(\frac{\bigcup_{t\in \Z} \left(\Leg_V(\Delta^{-t}\Pat)\times \Leg(\K; t)\right)}{\sim}\right)\to \Leg(\Pat(\K))\co (\Q,\Lc)\mapsto \Q(\Lc)
\]
with respect to the inherited map 
\[
\St_\pm \colon \left( \frac{\bigcup_{t\in \Z} \left(\Leg_V(\Delta^{-t}\Pat)\times \Leg(\K; t)\right)}{\sim}\right) \to  \left( \frac{\bigcup_{t\in \Z} \left(\Leg_V(\Delta^{-t}\Pat)\times \Leg(\K; t)\right)}{\sim}\right).
\]
Moreover, as stated in Theorem~\ref{thm:main} the relation $\sim$ accounts for the non-injectivity of $\widetilde\Sat$, thus ${\Sat}$ is a bijection.  Before turning to the proof of Theorem~\ref{thm:main} we prove a technical lemma.

\begin{lemma}\label{lem:tori1}
Suppose that $N$ is a solid torus with convex boundary realizing the knot type $\K$. 
\begin{enumerate}

\item\label{2thick} If the dividing slope of $\partial N$ is $n$, for some integer $n$, then for any solid torus $N'$ containing $N$ with standard convex boundary of slope $n+1$ there is another solid torus $N''$ satisfying $N\subset N''\subset N'$ and $N''$ has standard convex boundary of slope $n$ (if $\partial N$ has two dividing curves then one can take $N''=N$). The torus $N''$ is a standard neighborhood of a Legendrian knot $L$ and $N'$ is a standard neighborhood of a Legendrian knot that stabilizes to $L$. If $\K$ is Legendrian simple then there are at most two possibilities for $N'$. It can be a standard neighborhood of the Legendrian knots $L_+$ or $L_-$ where $\St_+(L_+)=\St_-(L_-)=L$. 

\item\label{uthick} If the dividing slope of $\partial N$ is in the interval $(n,n+1)$, for some integer $n$, then any solid torus $N'$ containing $N$ with standard convex boundary of slope $n+1$ is a standard neighborhood of a Legendrian representative of $\K$. If $\K$ is Legendrian simple then there is a unique possibility for this Legendrian representative. 
\end{enumerate}
Moreover, in Item \eqref{uthick} any such $N'$ can be obtained by a sequence of bypass attachments from the outside of $N$, and if $N'$ is a solid torus with standard convex boundary of slope $n+1$ obtained from $N$ by a single bypass attachment from the outside then the dividing slope of the boundary of $N$ is in $(n,n+1]$. 
\end{lemma}

\begin{proof}
The only non-standard part of the statement is Item~(\ref{uthick}). Suppose that the dividing slope of $\partial N$ is contained in the interval $(n,n+1)$ for some integer $n$. Then one can use the classification of contact structures on solid tori to find a solid torus $N'' \subset N$ with dividing slope $n$. Now suppose we are given $N'$ with dividing slope $n+1$ and satisfying $N\subset N'$. The tori $N''$ and $N'$ are each neighborhoods of Legendrian curves $L''$ and $L'$, respectively. Moreover $L''$ is a stabilization of $L'$. The sign of the basic slices of the thickened tori $N\setminus N''$ is determined by the sign of the basic slice $N'\setminus N''$. Since $N\setminus N''$ is independent of the thickening $N'$ we see that $L'$ is a fixed destabilization of $L''$. 
\end{proof}

\begin{theorem}\label{thm:main}
Suppose that $\K$ is a thickenable and Legendrian simple knot type, and $\Pat$ is a pattern in $V$. Assume that $\Pat(\K)$ has no topological symmetries, and if $\Pat$ has winding $0$, assume that $-\K=\K$ and $\Pat(\K)$ has oriented topological symmetries.

Let $\Lc, \Lc'$ be Legendrian representatives of $\K$  and let $\Q$ and $\Q'$ be Legendrian patterns representing $\Delta^{-\tb(\Lc)}\Pat$ and $\Delta^{-\tb(\Lc')}\Pat$, respectively. Then $\Q(\Lc)$ is Legendrian isotopic to $\Q'(\Lc')$ if and only if $(\Q,\Lc) \sim (\Q',\Lc')$ as in Definition \ref{def:kergen}.

In particular, the map ${\Sat}$ induced by $\widetilde{\Sat}$ on the equivalence classes of $\sim$ is an $\St_\pm$ equivariant bijection.
\end{theorem}

\begin{proof}
From the discussion above it is clear that if $(\Q,\Lc) \sim (\Q',\Lc')$ then $\Q(\Lc)$ is Legendrian isotopic to $\Q'(\Lc')$.

To prove the other implication after Legendrian isotopy we can assume that ${S}=Q(L)=Q'(L')$ for some $L\in \Lc$, $L'\in \Lc'$, $Q\in \Q$, and $Q'\in\Q'$. Let $\phi\colon (V,\xi_V)\to (S^3,\xi_{\textrm{st}})$ and $\phi'\colon (V,\xi_V)\to (S^3,\xi_{\textrm{st}})$ be the defining maps for $Q(L)$ and $Q'(L')$. In particular if $C$ is the core of $V$ then $\phi(C)=L$,  
 $\phi'(C)=L'$, $\phi(V)=\nu(L)$, $\phi'(V)=\nu(L')$, $\phi(Q)=S$, and $\phi'(Q')=S$.
By our assumption on no topological symmetries  the neighborhoods $\nu(L)$ and $\nu(L')$ are smoothly isotopic by an isotopy fixing $S$. Moreover, if $\Pat$ has non-zero winding number then we may assume this isotopy preserves an orientation on a longitude on the boundaries of these neighborhoods. If $\Pat$ does have winding number zero then after possibly reversing the orientation on $L$ (and replacing $\Q$ with $f(\Q)$ where $f$ was defined at the end of Subsection~\ref{subsec:satknot}) we may similarly assume that there is a well-defined oriented longitude preserved by the isotopy. Thus by Colin's isotopy discretisation argument \cite[Lemma~3.10]{Honda02} there is a sequence of convex tori $\partial \nu(L)=T_0,T_1,\dots,T_l=\partial \nu(L')$ bounding corresponding solid tori $N_i$ all containing ${S}$, such that $T_{i+1}$ is obtained from $T_{i}$ by a bypass attachment for $0\leq i<l-1$. Let $s_i$ be the dividing slope of $T_i$.

Choose smallest integer thickenings for the $N_i$, i.e.\ choose solid tori $N'_i$ containing $N_i$ with standard convex boundary of dividing slope $n_i=\lceil s_i \rceil$. Making sure that if $s_i$ is an integer, and the dividing curve on $T_i$ has two components then $N_i=N_i'$. Each $N_i'$ is contactomorphic to $(V,\xi_V)$ via the embedding $\phi_i\colon (V,\xi_V)\to (S^3,\xi_{std})$. Now let $L_i=\phi_i(C)$ and $Q_i=\phi_i^{-1}(S)$. Here $\phi_0=\phi$ and $\phi_l=\phi'$.
 
We claim that for each $i$ the pairs $(L_{i-1},Q_{i-1})$ and $(L_i,Q_i)$ are Legendrian isotopic pairs; $L_i$ is a stabilization of $L_{i-1}$ and depending on the sign of the stabilization $Q_i$ is either $\zeta Q_{i-1}$ or $\sigma Q_{i-1}$; or vice versa.

By symmetry we may assume that $N_i$ is obtained from $N_{i-1}$ via a nontrivial bypass attachment on the outside, so $N_{i-1}\subset N_{i}$.
Suppose that the boundary slope of $N_{i-1}$ is in $(n,n+1)$ for some integer $n$, and thus the boundary slope of $N_i$ is in $(n,n+1]$. Then $N_{i-1}$ is included in two solid torus $N_{i-1}\subset N_{i-1}'$ and $N_{i-1}\subset N_i\subset N_i'$, with boundary slope $n+1$. Thus by Lemma~\ref{lem:tori1}, $L_i$ and $L_{i-1}$ are Legendrian isotopic. Now, from the proof of Lemma~\ref{lem:tori1} there is a contactomorphism  $\phi\co N_{i-1}'\to N_i'$ fixing $N_{i-1}$ and hence also fixing $S$. Thus we have a contactomorphism $\psi_{i-1}^{-1}\circ \psi \circ \phi_i\colon (V,\xi_V)\to (V,\xi_V)$ bringing $Q_i$ to $Q_{i-1}$. Then by Corollary \ref{cor:iso}, $Q_i$ and $Q_{i-1}$ are Legendrian isotopic. 

If the boundary slope of $N_{i-1}$ is an integer $n$, then our construction gives nested solid tori $N_{i-1}=N_{i-1}'\subset N_i\subset N_{i}'$, and again 
by Lemma \ref{lem:tori1} $L_{i-1}$ is a stabilization of $L_i$. Now the same argument as above proves that depending on the sign of the stabilization $\zeta Q_{i-1}$ or $\sigma Q_{i-1}$ is Legendrian isotopic to $Q_{i}$.
\end{proof}

The above Theorem has a more trackable version when $\K$ is uniformly thick and Legendrian simple. In this case every representation of the Legendrian satellite is of the form $\Lc(\Q)$ for some peak $\Lc$ of the Legendrian mountain range, and for some Legendrian representation $\Q$ of the pattern $\Delta^{-\overline{t}}\Pat$. Thus in this case the map 
\[
\widetilde{\Sat}'\co \Leg_V(\Delta^{-\tt}\Pat)\times \Leg(\K;\tt) \to \Leg(\Pat(\K))
\]
is surjective. 
To phrase how the relation $\sim$ desecends to $\Leg_V(\Delta^{-\tt}\Pat)\times \Leg(\K;\tt)$ we introduce some notation.

A \emph{peak} of the mountain range for $\K$ is a Legendrian representative $\Lc$ that has Thurston--Bennequin number $\tt$.  
Two peaks $\Lc$ and $\Lc'$ are said to be neighbouring if there is no other peak with rotation number 
between $\rot(\Lc)$ and $\rot(\Lc')$. The \emph{valley} $\widetilde{\Lc}$ between neighbouring peaks is the common stabiliztaion of $\Lc$ and $\Lc'$ with maximal Thurston-Bennequin number. The \emph{depth} of the \emph{valley} $\widetilde{\Lc}$ is $d=\frac 12 |\rot(\Lc)-\rot(\Lc')|=\tt-\tb(\widetilde{\Lc})$. Then for $\rot(\Lc)<\rot(\Lc')$ we have $\widetilde{\Lc}=\mathit{St}_+^d\Lc'=\mathit{St}_-^d\Lc$.  If $\Lc$ is a peak in the mountain range of $\Leg(\K)$ we say $\Lc'\in \Leg(\K)$ is in the \emph{shadow} of $\Lc$ if $\Lc'$ is some (possibly iterated positive and negative) stabilization of $\Lc$.

Note that if $\Lc$ and $\Lc'$ are neighbouring peaks with $\rot(\Lc)<\rot(\Lc')$ and valley $\widetilde{\Lc}$ of depth $d=\frac{\rot(\Lc')-\rot(\Lc)}2$ then for $\widetilde{Q}\in  \Leg_V(\Delta^{-\tt+d}\Pat)$ we have:
\[\Q(\Lc)=(\zeta^d\widetilde{\Q})(\Lc)=\widetilde{\Q}(\St_+^d\Lc)=\widetilde{\Q}(\widetilde{\Lc})=\widetilde{\Q}(\St_-^d\Lc')=(\sigma^d\widetilde{Q})(\Lc')=\Q'(\Lc'),\]
where  $\Q=\zeta^d\widetilde{\Q}$ and $\Q'=\sigma^d\widetilde{\Q}$ are in $\Leg_V(\Delta^{-\tt}\Pat)$. 

Again, if there are no oriented topological symmetries then the equivalence relation $\sim'$ is what follows from the above observation, if we have oriented topological symmetries then we need to include them in the relation too. 

\begin{definition}\label{def:ker} Let $\K$ be a uniformly thick and Legendrian simple knot type, and $\Pat$ be a pattern in $V$. If the winding number of $\Pat$ is not zero then for the peaks 
$\Lc, \Lc'\in\Leg(\K;\tt)$ and the Legendrian patterns $\Q,\Q'\in\Leg_{V}(\Delta^{-\tt}\Pat )$ define $(\Q, \Lc)\sim' (\Q',\Lc')$ 
if and only if there is a sequence 
\[
(\Q,\Lc)=(\Q_0,\Lc_0),(\Q_1,\Lc_1),\dots,(\Q_k,\Lc_k)=(\Q',\Lc'),
\] 
where
\begin{itemize}
\item $\tb(\Lc_i)=\tt$ and $\Lc_{i-1}$ and $\Lc_{i}$ are neighboring peaks with valley $\widetilde{\Lc}_i$ of depth $d_i$ between them; and
\item if $\rot(\Lc_i)>\rot(\Lc_{i-1})$ (resp. $\rot(\Lc_i)<\rot(\Lc_{i-1})$) then there is a Legendrian pattern $\widetilde{\Q}_i\in \Leg_{V}(\Delta^{d_i-\tt}\Pat)$ with 
$\Q_{i-1}=\zeta^d\widetilde{\Q}_i$ and $\Q_{i}=\sigma^d\widetilde{\Q}_i$ (resp. $\Q_{i-1}=\sigma^d\widetilde{\Q}_i$ and $\Q_{i}=\zeta^d\widetilde{\Q}_i$).
\end{itemize}
If $\Pat$ has winding number zero, $-\K=\K$, and $\Pat(\K)$ has oriented topological symmetries as defined at the end of Subsection~\ref{subsec:satknot} then we define the same equivalence relation except that at the first step in the sequence we also allow, but do not require, $\Lc_{1}=-\Lc_0$ and $\Q_{1}=f(\Q_0)$. (The map $f$ is defined at the end of Section~\ref{subsec:satknot} and is a contactomorphism of $V$.) 
\end{definition}
Then by the above observation the map $\Sat'$ descends to an $\St_\pm$ equivariant map on the quotient 
\[
{{\Sat}'}\colon \left(\frac{\Leg_V(\Delta^{-\tt}\Pat)\times \Leg(\K;\tt)}{\sim'}\right) \to  \Leg(\Pat(\K))
\] 
with respect to the inherited map 
\[
\St_\pm \colon \left(\frac{\Leg_V(\Delta^{-\tt}\Pat)\times \Leg(\K;\tt)}{\sim'}\right) \to  \left(\frac{\Leg_V(\Delta^{-\tt}\Pat)\times \Leg(\K;\tt)}{\sim'}\right).
\]
Again, as stated in Theorem~\ref{lem:incommonstab} the relation $\sim'$ accounts for the non-injectivity of $\widetilde\Sat'$, thus ${\Sat'}$ is a bijection.
\begin{theorem}\label{lem:incommonstab}
Suppose that $\K$ is a uniformly thick and Legendrian simple knot type, and $\Pat$ is a pattern in $V$. Assume there are no topological symmetries of $\Pat(\K)$ that send the satellite incompressible torus to another incompressible torus in the complement of $\Pat(\K)$. In addition, if $\Pat$ has winding number zero, assume that $-\K=\K$ and $\Pat(\K)$ has oriented topological symmetries. 
Let $\Lc, \Lc'\in\Leg(\K;\tt)$ be peaks of the mountain range of $\K$ and let $\Q,\Q'\in\Leg_{V}(\Delta^{-\overline{tb}(\K)}\Pat )$ be Legendrian patterns. Then 
 $\Q(\Lc)$ is Legendrian isotopic to $\Q'(\Lc')$ if and only if $(\Q,\Lc) \sim' (\Q',\Lc')$ as in Definition \ref{def:ker}.
In particular, 
\[
{{\Sat}'}\co \left(\frac{\Leg_V(\Delta^{-\tt}\Pat)\times \Leg(\K;\tt)}{\sim'}\right) \to \Leg(\Pat(\K))
\]
is an $\St_\pm$ equivariant bijection.
\end{theorem}
\begin{proof}
Theorem \ref{thm:main} gives a sequence $(\Q,\Lc)=(\Q_0',\Lc'_0),(\Q'_1,\Lc'_1),\dots,(\Q'_k,\Lc'_k)=(\Q',\Lc')$, such that for each $i$ the Legendrian knot
$\Lc_i$ is a (de)stabilization of $\Lc_{i-1}$ and $Q_i$ and $Q_{i-1}$ are related appropriately by $\zeta$ and $\sigma$. Each $\Lc_i'$ is in the shadow of some set of consecutive peaks $\mathit{Peak}_i$. We will inductively choose $(\Lc_i,\Q_i)$ as in Definition \ref{def:ker}. Let $\Lc_0=\Lc$ and $\Q_0=\Q$ and suppose that $(i_1+1)$ is the first index such that $\mathit{Peak}_{i_1+1}$ does not contain $\Lc_0$. Then $\mathit{Peak}_{i_1+1}$ must contain exactly one of the neighbours of $\Lc_0$, say $\Lc_1$, and $\Lc_{i_1}$ must be a shadow of both $\Lc_0$ and $\Lc_1$. Then by Legendrian simplicity $\Lc_{i_1}$ is in the shadow of the valley $\widetilde{\Lc}_1$ between $\Lc_0$ and $\Lc_1$.  In fact, $\Lc_{i_1}$ is obtained from $\widetilde{\Lc}_1$ by some stabilisations of the same sign, say $\Lc_{i_1}=\St_+^a\widetilde{\Lc}_1$ for some $a$. 
Also $\widetilde{\Lc}_1=\St_+^{d_1}\Lc_0$ and $\widetilde{\Lc}_1=\St_-^{d_1}\Lc_1$. With this notation $\Lc_{i_1}=\St_+^{a+d_1}\Lc_0$.
Note that since $\Lc_0',\dots,\Lc_{i_1}'$ are all in the shadow of $\Lc_0$ then we can follow through the changes of the $\Q_i'$ and conclude that $\Q_0=\zeta^{a+d_1}\Q_{i_1}'$.  Let $\widetilde{\Q}_1=\zeta^a\Q_{i_1}$ then we have $\Q_0=\zeta^{d_1}\widetilde{\Q}_1$ and define $\Q_1$ as $\sigma^{d_1}\widetilde{\Q}_1$. Now $\Q_{i_1+1}'=\sigma^{d_1-1}\widetilde{\Q}_{1}$, and we can continue our induction to define the pairs $(\Lc_2,\Q_2),\dots (\Lc_l,\Q_l)$. Since $\Lc'$ is only in the shadow of itself $\Lc_k=\Lc'$, and once we are in the shadow of $\Lc_k$ we can follow through the changes of the $\Q_i'$ backwards, and see that $\Q_l=\Q'$, concluding the proof.
\end{proof}

For certain patterns Theorem \ref{lem:incommonstab} gives us a complete understanding of the number of Legendrian representatives of satellites. 

\begin{theorem}\label{thm:mainleg}
Suppose that $\K$ is a uniformly thick and Legendrian simple knot type.  Denote the maximal Thurston--Bennequin number of $\K$ by $\tt$. Suppose  
\[
\Leg(\K;\tt)=\{\Lc_0,\Lc_1 \dots,\Lc_k\}
\]
where  
\[
\rot(\Lc_0)=r_0<\rot(\Lc_1)=r_1<\cdots<\rot(\Lc_k)=r_k.
\]
And 
let $d_1=\frac{r_2-r_1}2,\dots,d_{k-1}=\frac{r_k-r_{k-1}}2$ denote the depths of the valleys $\widetilde{\Lc}_1,\dots,\widetilde{\Lc}_k$. 

Let $\Pat$ be a pattern with winding number $n$. 
Assume that for two patterns $\Q,\Q'\in \Leg_V(\Delta^{-m+d}\Pat)$ (here $m$ is any integer and $d$ is any natural number) we have 
$\sigma^d\Q=\sigma^d\Q'$ if and only if $\zeta^d\Q=\zeta^d\Q'$. 
Assume there are no topological symmetries of $\Pat(\K)$ that send the satellite incompressible torus to another incompressible torus in the complement of $\Pat(\K)$. 
\begin{enumerate}
\item\label{mainleg:1A} If $\Pat$ has nonzero winding number, then the number of Legendrian representatives of $\K(\Pat)$ with invariants $(\tb,\rot)=(t,r)$ is
\begin{align*}
\sum_{i=0}^{k}&\left\vert\Leg_{V}(\Delta^{-\overline{t}}\Pat;(t-n^2\overline{t}),r-nr_i)\right\vert\\
&\qquad\qquad-\sum_{i=1}^{k}\left\vert\sigma^{d_i}\left(\Leg_{V}(\Delta^{-\tb(\widetilde{\Lc}_i)}\Pat;(t-n^2\tb(\widetilde{\Lc}_i)),r-n\rot(\widetilde{\Lc}_i)\right)\right\vert.
\end{align*}
\item\label{mainleg:2A} If $\Pat$ has winding number zero, assume that $-\K=\K$ and $\Pat(\K)$ has oriented topological symmetries. Then
the number of Legendrian representatives of $\K(\Pat)$ with invariants $(\tb,\rot)=(t,r)$ is:
\begin{align*}
\sum_{i=0}^{\left\lfloor\frac{k+1}{2}\right\rfloor}\left\vert\Leg_{V}(\Delta^{-\overline{t}}\Pat;t,r)\right\vert
& +\mathcal{X}_{2\mid k}\left\vert \frac{\Leg_{V}(\Delta^{-\overline{t}}\Pat;t,r)}{\Q\sim f(\Q)} \right\vert \\
& - \sum_{i=1}^{\left\lfloor\frac{k}{2}\right\rfloor}\left\vert\sigma^{d_i}\left(\Leg_{V}(\Delta^{-\tb(\widetilde{\Lc}_i)}\Pat;t,r)\right)\right\vert ,
\end{align*}
where $\mathcal{X}_{2\mid k}$ is the indicator function with value 1 if $k$ is even, and 0 if $k$ is odd. 
\end{enumerate}
\end{theorem}

\begin{remark} The conditions of the Theorem hold when the maps $\sigma^d,\zeta^d\colon \Leg_V(\Delta^{-m+d}\Pat;t,r)\to  \Leg_V(\Delta^{-m}\Pat)$ are all injective, or when $\Leg_V(\Delta^{-m+d}\Pat;t,r)$ is empty or has only one element. In all of these cases 
$\left\vert\sigma^{d}\left(\Leg_{V}(\Delta^{-t+d}\Pat;(t,r)\right)\right\vert$ is $\left\vert\Leg_{V}(\Delta^{-t+d}\Pat;(t,r)\right\vert$, $1$, or $0$ and the formulas simplify accordingly. 
\end{remark} 
For an illustration on how to use the above statements, see the proofs of Theorems~\ref{2braidsclass} and~\ref{thm:twistsat}.
\begin{proof}
For the proof of the first statement fix the pair $(t,r)$. We first check that the maps
\[\Sat(\cdot,\Lc)\colon \Leg_{V}(\Delta^{-\tb(\Lc)}\Pat;(t-n^2\tb(\Lc)),r-n\rot(\Lc)) \to \Leg(\Pat(\K);t,r)\]
are injective for any $\Lc \in \Leg(\K)$. It is enough to prove the statement for a peak $\Lc$, so assume that  $\Q(\Lc)=\Q'(\Lc)$. Then Theorem~\ref{lem:incommonstab} gives a path  $(\Q,\Lc)=(\Q_0,\Lc_0),$ $(\Q_1,\Lc_1),\dots,$ $(\Q_k,\Lc_k)=(\Q',\Lc)$ as in Definitition~\ref{def:ker}. Since $\Lc_0=\Lc_k$ there must be an $i$ such that $\Lc_{i-1}=\Lc_{i+1}$. This means that the valleys $\widetilde{\Lc}_{i}$ and $\widetilde{\Lc}_{i+1}$ must agree too, and without loss of generality we can assume that $\St_+^d\Lc_{i-1}=\St_+^d\Lc_{i+1}=\widetilde{\Lc}_{i}=\widetilde{\Lc}_{i+1}=\St_-^d\Lc_{i}$. Then according to Definition~\ref{def:ker} there are patterns $\widetilde{\Q}_{i}$ and $\widetilde{\Q}_{i+1}$ such that $\sigma^d\widetilde{\Q}_{i}=\Q_{i}=\sigma^d\widetilde{\Q}_{i+1}$, thus by the hypothesis of the theorem $\Q_{i-1}=\zeta^d\widetilde{\Q}_{i}=\zeta^d\widetilde{\Q}_{i+1}=\Q_{i+1}$. So the sequence can be shortened. 
Repeating the same procedure, we will end up with a sequence of length one, which proves that $\Q=\Q'$.

Theorem~\ref{lem:incommonstab} and Lemma~\ref{lem:legbraidsat} say that ${{\Sat}'}$ maps the set
\[
S= \bigcup_{i=0}^k \left(\Leg_{V}(\Delta^{-\overline{t}}\Pat;(t-n^2\overline{t}),r-nr_i)\times \{\Lc_i\}\right)
\]
onto $\Leg(\Pat(\K); t,r)$ and thus the first term in the equation in Item~\ref{mainleg:1A} is an upper bound on $|\Leg(\Pat(\K); t,r)|$.

Arguing as above, two elements in $S$ corresponding to adjacent peaks $\Lc_{i-1}$ and $\Lc_{i}$ will map to the same element under $\Sat'$ if and only if they are also in the image of $\Sat$ restricted to 
\begin{align*}
\sigma^{d_i}&\left(\Leg_{V}(\Delta^{-\tb(\widetilde{\Lc}_i)}\Pat;(t-n^2\tb(\widetilde{\Lc}_i)),r-n\rot(\widetilde{\Lc}_i)\right) \times \{\Lc_{i-1}\}\\
&= \zeta^{d_i}\left(\Leg_{V}(\Delta^{-\tb(\widetilde{\Lc}_i)}\Pat;(t-n^2\tb(\widetilde{\Lc}_i)),r-n\rot(\widetilde{\Lc}_i)\right) \times \{\Lc_{i}\}.
\end{align*}
Thus, by the injectivity of $\Sat(\cdot, \Lc)$, the second term in the equation in Item~\ref{mainleg:1A} accounts for the over count in the first term. 

In the second part $n=0$ and the maps
\[\Sat(\cdot,\Lc)\colon \Leg_{V}(\Delta^{-\tb(\Lc)}\Pat;t,r) \to \Leg(\Pat(\K);t,r)\]
are only injective if $\Lc\neq -\Lc$, and by Legendrian simplicity and our hypothesis this is equivalent to $\rot(\Lc)\neq 0$. When $\rot(\Lc)=0$, then we only get an injective map after factoring out with the relation $\Q\sim f(\Q)$. 
Now the argument is identical to the one above after observing, that by the symmetry $\Lc_i(\Q)=\Lc_{k-i+1}(f(\Q))$ every Legendrian representative of the satellite can be written in the form $\Q(\Lc_i)$ for $i\leq \lceil \frac{k+1}2\rceil$.
\end{proof}

\subsubsection{Connected sums} 
Recall from Section~\ref{subsec:satknot} that the connected sum of two knots can be thought of as a satellite of one by the other. More specifically given $\K_1$ and $\K_2$ in $S^3$ let $\mu$ be a meridian to $K_1$. Then the complement of a neighborhood of $\mu$, $V=S^3-\nu(\mu)$, is a solid torus containing $\K_1$ and it has a canonical product structure (coming from $\mu$ and a meridian of $\mu$). So we can think of $\K_1$ as determining a pattern $\Pat_{\K_1}$ in $V=S^1\times D^2$. (It is useful to notice that the pattern $\Pat_{\K_1}$ is independent of the product structure chosen on the complement of $\nu(\mu)$. In particular, we could have chosen a Legendrian unknot with maximal Thurston-Bennequin invariant to represent $\mu$ and then taken the framing coming from the dividing curves on the boundary of the neighborhood. This will be convenient below.) One may easily see that $\K_1\#\K_2$ is the same topological knot as $\Pat_{\K_1}(\K_2)$. 

We will now see how to recover part of the structure theorem concerning connected sums from \cite{EtnyreHonda03} using Theorem~\ref{lem:incommonstab}. We begin by recalling Legendrian connected sums for Legendrian knots $\Lc_1$ and $\Lc_2$ in the standard contact structure on $S^3$ (there is a more general notion defined in any contact manifold, but we will not need that here). Choose Legendrian representatives $L_1$ and $L_2$ of $\Lc_1$ and $\Lc_2$ so that the front diagram for $L_1$ is to the left of the one for $L_2$ and a right cusp of $L_1$ is right beside a left cusp for $L_2$. Remove a small neighborhood of a right cusp from $L_1$ and a left cusp form $L_2$, and then connect the remained with two horizontal arcs. This will result in a Legendrian knot $L_1\#L_2$ and one can show that its Legendrian isotopy type is independent of the choices of $L_1$, $L_2$, and all other choices, \cite{EtnyreHonda03}. Thus we denote the resulting Legendrian isotopy type by $\Lc_1\#\Lc_2$.  We can now show the following. 
\begin{theorem}\label{csum}
Given two knots $\K_1$ and $\K_2$ in $S^3$. 
Let $\K=\K_1\#\K_2$ with either $\K_1$ or $\K_2$ Legendrian simple and uniformly thick. Also assume that $\K_1$ is not isotopic to $\K_2$ (and if the knots are not prime then neither contains a summand of the other). Then there is a bijection 
\[
C:\left(\frac{\Leg(\K_1)\times \Leg(\K_2)}{\sim}\right) \to \Leg(\K_1\#\K_2)
\]
given by Legendrian connected sum: $C(\Lc_1,\Lc_2)=\Lc_1\#\Lc_2$, where the equivalence relation $\sim$ is generated by $(\St_\pm(\Lc_1),\Lc_2)\sim(\Lc_1, \St_\pm(\Lc_2))$.
\end{theorem}
\begin{remark}
In \cite{EtnyreHonda03} a more general result was shown. Specifically the ambient manifolds did not have to be $(S^3,\xi_{std})$ and symmetries between $\K_1$ and $\K_2$ were allowed. More importantly, the requirement that one of $\K_1$ or $\K_2$ be Legendrian simple and uniformly thick was not needed. On the one hand this shows a deficiency in our understanding of satellites, but on the other hand it points to the fact that for some patterns we might be able to dispense with the Legendrian simple and uniformly thick hypotheses in Theorem~\ref{lem:incommonstab}.
\end{remark}
Before giving the proof we make a simple observation.
\begin{lemma}\label{csumpattern}
Given a knot type $\K$ in $S^3$ and the associated pattern $\Pat_\K$ in $V$ defined above, there is a bijection 
\[
\Leg_V(\Pat_\K)\to \Leg(\K)
\]
given by sending $\Q$ to $\Q(U)$ where $U$ is the maximal Thurston-Bennequin representative of the unknot. 
\end{lemma}
\begin{proof}
Given any $\Lc\in \Leg(\K)$ and specific Legendrian representative $L$ of $\Lc$ let $M$ be a Legendrian representative of the meridian to $L$ with $tb(M)=-1$. The complement of a standard neighborhood of $M$ is a solid torus $N$ that is a standard neighborhood of an unknot $U$ with $tb(U)=-1$. Now let $\phi\co V\to N$ be the map from the 1-jet space of $S^1$ to $N$ and let $Q_L$ be $\phi^{-1}(L)$ in $V$. 

Notice that $Q_L$ is well-defined up to isotopy since any isotopy, $L_t$, $t\in[0,1]$, of $L$ (or similarly isotopy of the chosen Legendrian meridian $M$) will induce a one parameter family of smooth maps $\phi_t\co V\to S^3$ and we can assume that $\phi_0$ and $\phi_1$ are the maps used to define the standard neighborhood of a Legendrian unknot. Pulling back $\xi_{std}$ by the $\phi_t$ will give a loop of contact structures on $V$. Lemma~\ref{lem:contactstr} says that the space of contact structures $\Xi(V)$ has trivial fundamental group, so arguing as in Theorem~\ref{isoIScontacto} we can conclude that the $\phi_t$ may be isotoped, relative to $t=0,1$, so that all the $\phi_t$ give parameterizations of standard neighborhoods of unknots. Thus $\phi_t(L_t)$ give a Legendrian isotopy from $Q_{L_0}$ to $Q_{L_1}$ and we see that $\Lc$ defines a Legendrian pattern $\Q_\Lc$.

From the construction it is clear that $\Q_\Lc(U)=\Lc$ and thus the map in the lemma is surjective. The argument in the previous paragraph also shows that if $\Q_\Lc(U)$ is Legendrian isotopic to $\Q_{\Lc'}(U)$ in $S^3$ then $\Q_\Lc$ is Legendrian isotopic to $\Q_{\Lc'}$ in $V$. Thus we see that the map in the lemma is also injective. 
\end{proof}
\begin{proof}[Proof of Theorem~\ref{csum}]
Suppose, without loss of generality, that $\K_2$ is Legendrian simple and uniformly thick. Let $\tt$ be the maximal Thurston-Bennequin invariant of $\K_2$. 
Notice that by the symmetries in the statement of Theorem~\ref{csum} the set 
\[
\frac{\Leg(\K_1)\times \Leg(\K_2)}{\sim}
\]
is the same as the set
\[
\frac{\Leg(\K_1)\times \Leg(\K_2;\tt)}{\sim}.
\]
In addition notice that when written in this form $\sim$ in the statement of the theorem reduces to $\sim$ from the statement of Theorem~\ref{lem:incommonstab}.

Notice that since there is a meridional disk in $V$ that intersects $\Pat_{\K_1}$ in a single point we know that $\Delta^k\Pat_{\K_1}=\Pat_{\K_1}$ for any integer $k$. Moreover, Lemma~\ref{csumpattern} implies that $\Leg_V(\Delta^{-\tt}\Pat_{\K_1})=\Leg_V(\Pat_{\K_1})$ can be identified with $\Leg(\K_1)$. 

We also know from above that $\Leg(\Pat_{\K_1}(\K_2))=\Leg(\K_1\#\K_2)$. Observing that the Legendrian satellite of $\Lc_2$ in $\Leg(\K_2)$ by a Legendrian pattern $\Q_{\Lc_1}$ in $\Leg_V(\Pat_{\K_1})$ is the same as the Legendrian connected sum of $\Lc_1\#\Lc_2$ we see that the map in the statement of the theorem before modding out by equivalence is simply 
\[
C={\Sat}\co \left(\frac{\Leg_V(\Delta^{-\tt}\Pat)\times \Leg(\K;\tt)}{\sim}\right) \to \Leg(\Pat(\K)).
\]
Thus the bijectivity of $\Sat$ in Theorem~\ref{lem:incommonstab} implies the bijectivity of $C$. 
\end{proof}

\subsubsection{Legendrian braid satellites} 

As a direct consequence of Theorems \ref{thm:mainleg} and \ref{2starndbraid} we can reprove a theorem originally proven in \cite{EtnyreHonda05} using quite different techniques. 

\begin{theorem}\label{2braidsclass}
Let $\K$ be a Legendrian simple and uniformly thick knot type. Then any 2--braided satellite of $\K$ is also Legendrian simple. 

More explicitly denote the 2--braided pattern with $m$ (odd) twists by $\Pat_m$ and the maximal Thurston-Bennequin invariant of $\K$ by $\tt$. Moreover, suppose 
\[
\Leg(\K;\tt)=\{\Lc_0,\Lc_1 \dots,\Lc_k\}
\]
where  
\[
\rot(\Lc_0)=r_0<\rot(\Lc_1)=r_1<\cdots<\rot(\Lc_k)=r_k.
\]

If $m>2\tt$, then  
the maximum Thurston-Bennequin number of $\Pat_{m}(\K)$ is $2\tt+m$, 
there are exactly $k+1$ elements in $\Leg(\Pat_m(\K),2\tt+m)$ 
that realize the rotation numbers $2r_i$ for $i=0,\ldots k$, and all other Legendrian knots realizing $\Pat_m(\K)$ destabilize to one of these. 

If $m<2\tt$, then 
the maximum Thurston-Bennequin number of $\Pat_{m}(\K)$ is $2m$,
the rotation numbers realized by elements in $\Leg(\Pat_m(\K),2m)$ are in the set 
\[
\mathcal{R}=\{2r_i + (m-2\tt) + 2l| i=0,\ldots, k, l=0,\ldots, 2\tt-m\}
\] 
(as this is a set multiplicities are ignored), the cardinality of the set $\Leg(\Pat_m(\K),2m)$ is the same as the set $\mathcal{R}$, and all other Legendrian knots realizing $\Pat_m(\K)$ destabilize to one of these. 
\end{theorem}
\begin{proof}
First note that since Legendrian representatives of the pattern $\Delta^{-\tt}\Pat_m=\Pat_{m-2\tt}$ always destabilize to a representative with maximal Thurston-Bennequin number, thus by Lemma~\ref{lem:utpsatelite} the same statement is true for Legendrian representatives of $\Pat_m(\K)$.

Since $\Pat_{m-2t}$ is Legendrian simple for any $t$ we see that $\Leg_V(\Pat_{m-2t}; a,b)$ contains one element or is empty. Moreover the maps $\sigma$ and $\zeta$ are always surjective. So the condition of Theorem~\ref{thm:mainleg} holds, thus we can use a simplified version of the first formula of Theorem~\ref{thm:mainleg}:
\[
\sum_{i=0}^{k}\left\vert\Leg_{V}(\Pat_{m-2\tt};(t-4\overline{t}),(r-2r_i))\right\vert-\sum_{i=1}^{k}\left\vert\left(\Leg_{V}(\Pat_{m-2\widetilde{t}_i};(t-4\widetilde{t}_i),(r-2\widetilde{r}_i)\right)\right\vert.
\]
where $\widetilde{t}_i$ and $\widetilde{r}_i$ are the Thurston-Bennequin and rotation numbers of the valleys $\widetilde{\Lc}_i$ between $\Lc_{i-1}$ and $\Lc_{i}$, i.e. $\widetilde{t}_i=\tb(\widetilde{\Lc}_i)=\tt-\frac{r_{i}-r_{i-1}}{2}$ and $\widetilde{r_i}=\rot(\widetilde{\Lc}_i)=\frac{r_{i}+r_{i-1}}2$.

Considering the case where $m-2\tt>0$ notice that $\Leg_V(\Pat_{m-2\tt}; t,r)$ is non-empty if and only if $|r|\leq m-2\tt-t$, and $t+r$ is odd. Thus $\Leg_{V}(\Pat_{m-2\tt};(t-4\overline{t}),(r-2r_i))$ is non-empty if and only if $|r-2r_i|\leq m+2\tt-t$ and $t+r$ is odd. 

In addition notice that if $|r-2r_{i-1}|\leq m+2\tt-t$, $|r-2r_i|\leq m+2\tt-t$, and $t+r$ is odd, then since $\widetilde{r_i}$ is between $r_{i-1}$ and $r_{i}$ we see that 
\begin{align*}
|r-2\widetilde{r_i}| &= |r - 2(r_{i-1}+\frac{r_i-r_{i-1}}{2})|\leq |r-2r_{i-1}| + ({r_i-r_{i-1}})\\
&\leq  m+2\tt-t +({r_i-r_{i-1}})= m+2\widetilde t_i-t
\end{align*}
and so $\left(\Leg_{V}(\Pat_{m-2\widetilde{t}_i};(t-4\widetilde{t}_i),(r-2\widetilde{r}_i)\right)$ is non-empty. Similarly if $|r-2\widetilde{r_i}|\leq m+2\widetilde t_i-t$ then one may easily check that $|r-2r_i|\leq m+2\tt-t$ and $|r-2r_{i+1}|\leq m+2\tt-t$. So in the equation above, the terms in the first sum are 1 exactly one more time than the terms in the second sum. This establishes Legendrian simplicity of $\Pat_m\K$.

To complete the classification of Legendrian knots notice that $\Leg_{V}(\Pat_{m-2\tt};(t-4\overline{t}),(r-2r_i))$ will be empty if $(t-4\overline{t})> m+2\tt$ and if $(t-4\overline{t})= m+2\tt$ then we must have $r=2r_i$ for some $i$. 

We can argue similarly for the case when $m-2\tt<0$ by noticing that $\Leg_V(\Pat_{m-2\tt}; t,r)$ is non-empty if and only if $|r|\leq m-2\tt-t$, $t\leq -2m-4\tt$, and $t+r$ is odd. The extra constraint does not significantly affect the argument that the sum determining $|\Leg(\Pat_m\K;t,r)|$ is either 0 or 1 and thus $\Pat_m\K$ is Legendrian simple. Identifying the representatives with maximal Thurston-Bennequin invariant is also similar. 
\end{proof}

Theorem~\ref{thm:mainleg} together with Theorem~\ref{cablepats} gives an alternate proof of a theorem from \cite{EtnyreHonda05} by an an argument identical to the one give above. 
\begin{theorem}\label{pqcableclass}
Let $\K$ be a Legendrian simple and uniformly thick knot type. Then any $(p,q)$-cable of $\K$ is also Legendrian simple. 

More explicitly denote the $(p,q)$-cable pattern $\CC_{p,q}$ and suppose the maximum Thurston-Bennequin number of $\K$ is denoted $\tt$ and 
\[
\Leg(\K;\tt)=\{\Lc_0,\Lc_1 \dots,\Lc_k\}
\]
where  
\[
\rot(\Lc_0)=r_0<\rot(\Lc_1)=r_1<\cdots<\rot(\Lc_k)=r_k.
\]

If $p/q>\tt$, then  
the maximum Thurston-Bennequin number of $\CC_{p,q}(\K)$ is $pq-p+\tt q$, 
there are exactly $k+1$ elements in $\Leg(\CC_{p,q}(\K),pq-p+\tt q)$ 
that realize the rotation numbers $qr_i$ for $i=0,\ldots k$, and all other Legendrian knots realizing $\CC_{p,q}(\K)$ destabilize to one of these. 

If $p/q<\tt$ then taking $n$ so that $-n-1< p/q<-n$ we see that 
the maximum Thurston-Bennequin number of $\CC_{p,q}(\K)$ is $pq$,
the rotation numbers realized by elements of $\Leg(\CC_{p,q}(\K),pq)$ are in the set 
\[
\mathcal{R}=\{\pm(q\cdot\rot(L) + (p+nq)| L\in \Leg(\K, \tt-n)\}
\]
(as this is a set multiplicities are ignored), the cardinality of the set $\Leg(\CC_{p,q}(\K),pq)$ is the same as the set $\mathcal{R}$, and all other Legendrian knots realizing $\CC_{p,q}(\K)$ destabilize to one of these. \qed
\end{theorem}

While we cannot classify satellites for general braided patterns we can come close for positive braided patterns.
\begin{theorem}
Let $\K$ be a Legendrian simple and uniformly thick knot type with with maximal Thurston-Bennequin invariant $\tt$. Let $\Pat$ be any pattern such that $\Delta^{-\tt}\Pat$ is a positive braided pattern. Then maximal Thurston-Bennequin invariant representatives of $\Pat(\K)$ are distinguished by their rotation numbers and when any two such Legendrian knots are stabilized to have the same Thurston-Bennequin invariant and rotation number then they become Legendrian isotopic.
\end{theorem}
\begin{remark}
Notice that if one could show that any Legendrian knot in the knot type $\Pat(\K)$ destabilized to a maximal Thurston-Bennequin invariant representative then we would know that $\Pat(\K)$ is Legendrian simple. Unfortunately the tools we have developed so far do not establish this.  
\end{remark}
\begin{proof}
The proof of this theorem is almost identical to the proof of the first part of the Theorem \ref{2braidsclass} except we use Theorem~\ref{thm:closedposlegendrianbraid} in palace of Theorem~\ref{2starndbraid}.
\end{proof}

\subsubsection{Legendrian Whitehead doubles}\label{LWD}                       
The previous theorems about cabled satellites do not actually need the full strength of Theorem~\ref{lem:incommonstab} but only the observation made before Definition~\ref{def:ker} about when two Legendrian satellites are isotopic. We will now consider a situation where we do need the full strength of Theorem~\ref{lem:incommonstab} and also see another situation where the satellite of a Legendrian simple knot by a Legendrian simple pattern need not be Legendrian simple.

\begin{theorem} \label{thm:twistsat}
Let $\W_m$ be the Whitehead pattern with $m$ half-twists. And
let $\K$ be a uniformly thick and Legendrian simple knot type for which $\K=-\K$, and such that $\W_m(\K)$ has no topological symmetries (as in Theorem~\ref{lem:incommonstab}). Suppose the maximum Thurston-Bennequin number of $\K$ is denoted $\tt$ and 
\[
\Leg(\K;\tt)=\{\Lc_0,\Lc_1 \dots,\Lc_k\}
\]
where  
\[
\rot(\Lc_0)=r_0<\cdots<\rot(\Lc_k)=r_k.
\]

Let $j$ be the maximal depth of a valley  in the Legendrian mountain range of $\K$ and let $n_d$ be the number of valleys of depth $d$ corresponding to Legendrian knots with negative rotation numbers, so the total number of valleys is $k=2(n_1+\ldots + n_j)+\mathcal{X}_{2\nmid k}$. Here, again $\mathcal{X}_{2\nmid k}$ is the indicator function, with value 1 if $k$ is odd and value 0 if $k$ is even.

All Legendrian knots in $\Leg(\W_m(\K))$ destabilize to one of the maximal Thurston-Bennequin invariant representatives. 

\begin{enumerate}
\item\label{it:twistsat1} 

If $m\geq 2\tt$ even, then the maximal Thurston-Bennequin invariant of $\W_m(\K)$ is $2\tt-m+1$. There are $k+1$ elements in $ \Leg(\W_m(\K); 2\tt-m+1)$ and they all have rotation number $0$. Moreover for any $a=a_1+a_2>0$ with $a_i \ge 0$ and $h=\min\{a_1,a_2\}$ we have  
\[
\left|\Leg(\W_m(\K); 2\tt-m+1-a,a_1-a_2)\right|= \left\lceil \frac{k+1}2 \right\rceil -\sum_{d=1}^hn_d.
\] 
\item\label{it:twistsat2}  If $m>2\tt$ odd, then  the maximal Thurston-Bennequin invariant of $\W_m(\K)$ is $2\tt-m-3$. There are exactly $\left\lceil \frac{k+1}2\right\rceil$ elements in   $\Leg(\W_m(\K); 2\tt-m-3,\pm 1)$ and for other $r$ $\Leg(\W_m(\K); -(m-\tt)-3,r)$ is empty. Moreover, for any $a_1,a_2$ non-negative with $a=a_1+a_2$  and $h=\min\{a_1,a_2\}$ we have
\[
\left|\Leg(\W_m(\K); 2\tt-m-3-a,\pm(1+a_1-a_2))\right|=\left\lceil\frac{k+1}2 \right\rceil-\sum_{d=1}^hn_d.
\]
\item\label{it:twistsat3} If  $m<2\tt $ odd,  then the maximal Thurston-Bennequin invariant of $\W_m(\K)$ is $-3$. Let $l=\left\lfloor \frac{2\tt-m}2\right\rfloor$. There are 
\[
(k+1)\left( \frac{2\tt-m+1}2\right) -\sum_{d=1}^{l} n_d\left({2\tt-2d-m+1}\right)
\]
elements in $\Leg(\W_m(\K); -3)$ and they all have rotation number $0$. Moreover, for any $a=a_1+a_2>0$ with $a_i$ non-negative and $h=\min\{a_1,a_2\}$ we have  
\[
\left\lceil\frac{k+1}2 \right\rceil -\sum_{d=1}^ln_d-
\sum_{d=l+1}^hn_d
\]
elements in
\[
\Leg(\W_m(\K); -3-(a+1),\pm (1+a_1-a_2)),
\]
where the last sum is 0 unless $h>l$
\item\label{it:twistsat4} If $m<2\tt$ even, then the maximal Thurston-Bennequin invariant of $\W_m(\K)$ is $1$. Let $l=\frac{2\tt-m}2$, then there are 
\begin{equation*}
\left\lceil\left(\frac{k+1}2\right) \left(\frac{2\tt-m}{2}+1\right)^2\right\rceil - \sum_{d=1}^{l-1} n_d \left(\frac{2\tt-m-2d}2 +1\right)^2 -2n_l \end{equation*}
elements in $ \Leg(\W_m(\K); 1)$ and they all have rotation number $0$.  
Moreover, for any $a>0$ we have 
\[
\left\lceil\frac{k+1}2\right\rceil \left(\frac{2\tt-m}{2}+1\right) - \sum_{d=1}^{l} n_d \left(\frac{2\tt-m-2d}2 +1\right) \]
elements in $\Leg(\W_m(\K); 1-a,\pm a)$ and for $a=a_1+a_2$ with $a_1\geq 0$ and $a_2>0$, and $h=\min\{a_1,a_2\}$ we have  
\[
\left\lceil\frac{k+1}2 \right\rceil -\sum_{d=1}^ln_d-
\sum_{d=l+1}^hn_d
\]
elements in
\[
\Leg(\W_m(\K); 1-(a+1),\pm (1+a_1-a_2)),
\]
where the last sum is 0 unless $h>l$.
\end{enumerate}
\end{theorem}
\begin{proof}
We will only spell out the proof of Item~(\ref{it:twistsat4}), the proof for the rest of the statements is simpler. So let $m<2\tt$ even.
First notice, that the patterns $\Delta^{-t}\W_m=\W_{m-2t}$ destabilize, whenever they don't have maximal Thurston-Bennequin number, so the same will hold for $\W_m(\K)$. 

Setting $l=\frac{2\tt-m}2$,
the maps $\sigma^{d},\zeta^d\colon \Leg_V(\W_{m-2\tt+2d})\to \Leg_V(\W_{m-2\tt})$ are injective for all $d$ 
and have a single image when $d>l$. Thus Item~(\ref{mainleg:2A}) of Theorem~\ref{thm:mainleg} can be applied, and the number of Legendrian represenatives of $\W_m(\K)$ with Thurton-Bennequin number $t$ and rotation number $r$ is:
\begin{align*}
\sum_{i=0}^{\left\lfloor\frac{k+1}{2}\right\rfloor}\left\vert\Leg_{V}(\W_{m-2\tt};t,r)\right\vert
& +\mathcal{X}_{2\mid k}\left\vert \frac{\Leg_{V}(\W_{m-2\tt};t,r)}{\Q\sim f(\Q)} \right\vert \\
& - \sum_{i=1, d_i\le l}^{\left\lfloor\frac{k}{2}\right\rfloor}\left\vert\Leg_{V}(\W_{m-2\widetilde{t}_i};t,r)\right\vert 
- \sum_{i=1, d_i >l}^{\left\lfloor\frac{k}{2}\right\rfloor}\mathcal{X}_{(\Leg_V(\W_{m-2\widetilde{t}_i};t,r)\neq\emptyset)},
\end{align*}
Here $\widetilde{\Lc}_i$ is the valley between $\Lc_{i-1}$ and $\Lc_{i}$ of depth $d_i=\frac{r_{i}-r_{i-1}}{2}$, and $\widetilde{t}_i=\tb(\widetilde{\Lc}_i)=\tt-d_i$, the indicator function $\mathcal{X}_{(\Leg_V(\W_{m-2\widetilde{t}_i};t,r)\neq\emptyset)}$ is 1 when $\Leg_V(\W_{m-2\widetilde{t}_i};t,r)\neq\emptyset$ and 0 otherwise. Notice that we have separated the last sum in Theorem~\ref{thm:mainleg} according to the depths of the corresponding valleys (this is because when the valley has depth bigger than $l$ the corresponding set of Legendrian pattens will either be empty or contain one element).  

Next we will understand the action of $f$ on $\Leg_V(\W_{m-2\tt};t,r)$. First notice, that $f$ maps $\Leg_V(\W_{m-2\tt};t,r)$ to $\Leg_V(\W_{m-2\tt};t,-r)$, thus 
if $r\neq 0$, then \[\left|\frac{\Leg_V(\W_{m-2\tt};t,r)}{Q\sim f(Q)}\right|=|\Leg_V(\W_{m-2\tt};t,r)|.\] 
Also, for $t<1$ we have $\Leg_V(\W_{m-2\tt};t,0)=1$, thus $f$ brings the unique representative to itself,  and we have 
 \[\left|\frac{\Leg_V(\W_{m-2\tt};t,0)}{Q\sim f(Q)}\right|=1\]
Using the notation of the proof of Theorem \ref{whiteheadpatternclass} the symmetry of the Legendrian representations of $\W_{m-2\tt}$ with maximal Thurston-Bennequin number is given by 
\[
f\colon \Q_{(z^-,z^+)}\mapsto \Q_{\left(\frac{2\tt-m}2-z^+,\frac{2\tt-m}2-z^-\right)}.
\]
Thus $\Q\sim f(\Q)$ have equivalence classes of size two when neither $z_+$ or $z_-$ equals $\frac 12 \left(\frac{2\tt-m}2+1\right)$, and has a class of size one otherwise (notice this can happen only if $\frac 12 \left(\frac{2\tt-m}2+1\right)$ is an integer and then it can only happen once). This means that
\[\left|\frac{\Leg_V(\W_{m-2\tt};1,0)}{Q\sim f(Q)}\right|=\left\lceil\frac{(\frac{2\tt-m}2+1)^2}2\right\rceil.\]
The maximal Thurston-Bennequin number for $\W_m(\Q)$ is 1, and all such representatives have rotation number 0. Using the above formula for $(r,t)=(0,1)$ we get
\[
\left\lfloor\frac{k+1}{2}\right\rfloor \left(\frac{2\tt-m}2+1 \right)^2
 +\mathcal{X}_{2\mid k}  \left\lceil\frac{(\frac{2\tt-m}2+1)^2}2\right\rceil
 - \sum_{i=1, d_i< l}^{\left\lfloor\frac{k}{2}\right\rfloor}\left(\frac{2\widetilde{t}_i-m}2+1 \right)^2
- \sum_{i=1, d_i=l}^{\left\lfloor\frac{k}{2}\right\rfloor}2
\]
since for $d_i>l$ the maximal Thurston-Bennequin number of $\W_{m-2\tt+2d_i}$ is $1-m+2\tt-2d_i<1$, thus $\Leg_V(\W_{m-2\widetilde{t}_i};1,0)=\emptyset$.  In addition, notice that when $d_i=l$ the space $\Leg(\W_{n-2\widetilde{t}_i}; 1,0)=\Leg(\W_0;1,0)$ has two elements.  The above gives the first formula of Item~(\ref{it:twistsat4}), by simply noting that for integers $b$ and $c$ we have $\left\lfloor\frac{b}2\right\rfloor c+\mathcal{X}_{2\mid b-1}\left\lceil \frac c 2 \right\rceil=\left\lceil \frac{bc}2\right\rceil$.

In the case when the pair $(r,t)$ is on an edge of the mountain range, i.e. $(r,t)=(\pm a,1-a)$, then the formula can be computed similarly:
\[
\left\lfloor\frac{k+1}{2}\right\rfloor \left(\frac{2\tt-m}2+1 \right)
 +\mathcal{X}_{2\mid k}  \left(\frac{2\tt-m}2+1\right)
 - \sum_{i=1, d_i\le l}^{\left\lfloor\frac{k}{2}\right\rfloor}\left(\frac{2\widetilde{t}_i-m}2+1 \right)\]
This again, by $\left\lfloor\frac{b}2\right\rfloor +\mathcal{X}_{2\mid b-1}=\left\lceil \frac{b}2\right\rceil$, agrees with the second formula of Item~(\ref{it:twistsat4}). 
When $(r,t)$ is in the interior of the triangle, i.e.\ $(r,t)=(a_1-a_2,1-a)$ with $a=a_1+a_2$, then all relevant maps $\sigma^{d_i},\zeta^{d_i}$ have a single image, so we just need to understand for which $\widetilde{\Lc}_i$ we get $\Leg_V(\W_{m-2\widetilde{t}_i};t,r)\neq\emptyset$. This happens, when $m-2\widetilde{t}_i\le0$ or if $1-m+2\widetilde{t}_i\le \min\{1-a_1,1-a_2\}$, or equivalently if $d_i\le l$ or $l<d_i\le h$.  Again, the resulting sum agrees with the third formula of Item~(\ref{it:twistsat4}). 
\end{proof}

\begin{example}\label{whdtorus}
To give an example of Theorem~\ref{thm:twistsat} we will consider the Whitehead doubles of the $(-13, 3)$ torus knot $\K$. According to \cite{EtnyreHonda01b} this is a Legendrian simple knot type, the maximal Thurston-Bennequin invariant is $-39$, and there are 8 representatives with maximal Thurston-Bennequin invariant which realize the rotation numbers 
\[
-10, -8,-4,-2, 2, 4,8,10.
\]
In \cite{EtnyreHonda05} it was shown that negative torus knots are uniformly thick. So we can apply Theorem~\ref{thm:twistsat}. Clearly we have $n_1=2$ and $n_2=1$. From this one immediately computes the Legendrian representatives of the Whitehead doubles of $\K$. See Figures~\ref{fig:evenpos} and~\ref{fig:oddpos} for $m\geq -78$. 
For $m<  -78$ the general closed from for the numbers in the mountain range are more complicated, but see Figures~\ref{fig:oddneg} and~\ref{fig:oddnegplus}.
\end{example}

\begin{figure}
\centering
{\tiny\[
\xymatrixrowsep{2em}
\xymatrixcolsep{2em}
\xymatrix{
&&&&&&&8\ar[rd]^{+} \ar[ld]_{-}  &&&&&&\\
&&&&&&4\ar[rd]^{+} \ar[ld]_{-}&  &4\ar[rd]^{+} \ar[ld]_{-}&&&&&\\
&&&&&4\ar[rd]^{+} \ar[ld]_{-}&&2\ar[rd]^{+} \ar[ld]_{-}  &&4\ar[rd]^{+} \ar[ld]_{-}&&&&\\
&&&&4\ar[rd]^{+} \ar[ld]_{-}&&2\ar[rd]^{+} \ar[ld]_{-}&  &2\ar[rd]^{+} \ar[ld]_{-}&&4\ar[rd]^{+} \ar[ld]_{-}&&&\\
&&&4\ar[rd]^{+} \ar[ld]_{-}&&2\ar[rd]^{+} \ar[ld]_{-}&& 1\ar[rd]^{+} \ar[ld]_{-} &&2\ar[rd]^{+} \ar[ld]_{-}&&4\ar[rd]^{+} \ar[ld]_{-}&&\\
&&4\ar@{.}[rd] \ar@{.}[ld]&&2\ar@{.}[rd] \ar@{.}[ld]&&1\ar@{.}[rd] \ar@{.}[ld]&  &1\ar@{.}[rd] \ar@{.}[ld]&&2\ar@{.}[rd] \ar@{.}[ld]&&4\ar@{.}[rd] \ar@{.}[ld]&\\
&&&&&&&&&&&&&
}
\]}
\caption{Legendrian mountain range for the Whitehead double $\W_m(\K)$ of the $(-13,3)$ torus knot $\K$ where $m\geq -78$ is even. The maximal Thurston-Bennequin invariant is $-77-m$. 
}
\label{fig:evenpos}
\end{figure}
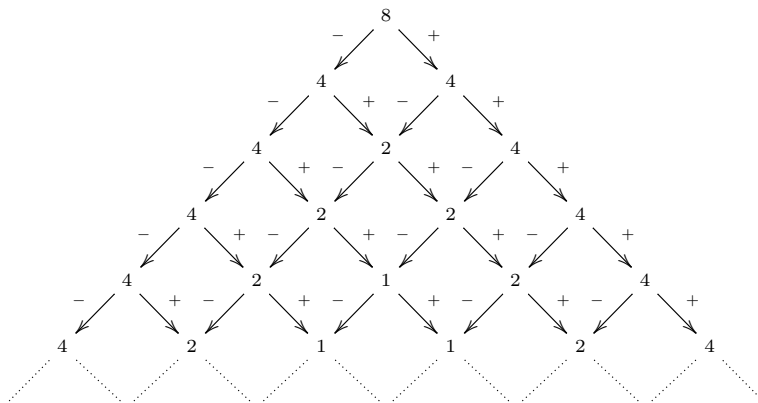

\begin{figure}
\centering
{\tiny\[
\xymatrixrowsep{2em}
\xymatrixcolsep{2em}
\xymatrix{
&&&&&&4\ar[rd]^{+} \ar[ld]_{-}&  &4\ar[rd]^{+} \ar[ld]_{-}&&&&&&\\
&&&&&4\ar[rd]^{+} \ar[ld]_{-}&& 4\ar[rd]^{+} \ar[ld]_{-} &&4\ar[rd]^{+} \ar[ld]_{-}&&&&&\\
&&&&4\ar[rd]^{+} \ar[ld]_{-}&&2\ar[rd]^{+} \ar[ld]_{-}&  &2\ar[rd]^{+} \ar[ld]_{-}&&4\ar[rd]^{+} \ar[ld]_{-}&&&&\\
&&&4\ar[rd]^{+} \ar[ld]_{-}&&2\ar[rd]^{+} \ar[ld]_{-}&& 2\ar[rd]^{+} \ar[ld]_{-} &&2\ar[rd]^{+} \ar[ld]_{-}&&4\ar[rd]^{+} \ar[ld]_{-}&&&\\
&&4\ar[rd]^{+} \ar[ld]_{-}&&2\ar[rd]^{+} \ar[ld]_{-}&&1\ar[rd]^{+} \ar[ld]_{-}&  &1\ar[rd]^{+} \ar[ld]_{-}&&2\ar[rd]^{+} \ar[ld]_{-}&&4\ar[rd]^{+} \ar[ld]_{-}&&\\
&4\ar@{.}[rd] \ar@{.}[ld]&&2\ar@{.}[rd] \ar@{.}[ld]&&1\ar@{.}[rd] \ar@{.}[ld]&& 1\ar@{.}[rd] \ar@{.}[ld] &&1\ar@{.}[rd] \ar@{.}[ld]&&2\ar@{.}[rd] \ar@{.}[ld]&&4\ar@{.}[rd] \ar@{.}[ld]&\\
&&&&&&&  &&&&&&&
}
\]}
\caption{Legendrian mountain range for the Whitehead double $\W_m(\K)$ of the $(-13,3)$ torus knot $\K$ where $m\geq -78$ is odd. The maximal Thurston-Bennequin invariant is $-81-m$.}
\label{fig:oddpos}
\end{figure}
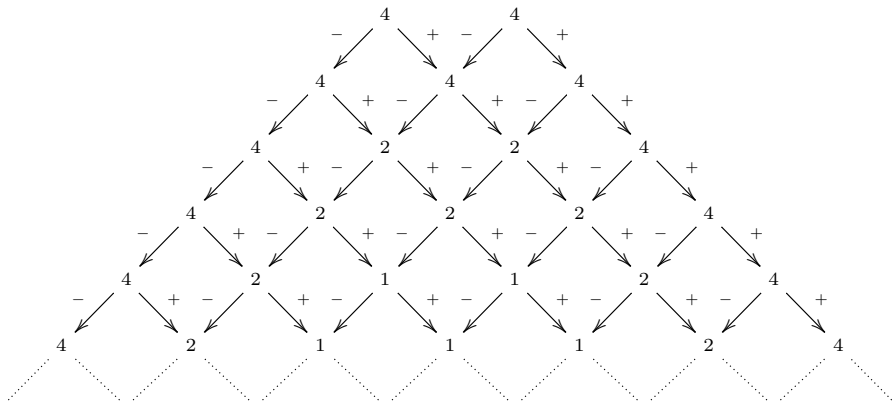

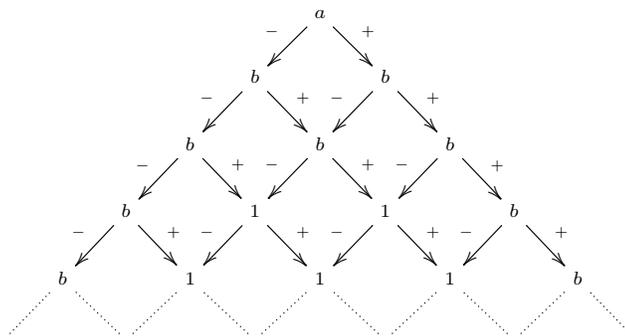
\begin{figure}
\centering
{\tiny\[
\xymatrixrowsep{2em}
\xymatrixcolsep{2em}
\xymatrix{
&&&&&&&a\ar[rd]^{+} \ar[ld]_{-}  &&&&&&\\
&&&&&&b\ar[rd]^{+} \ar[ld]_{-}&  &b\ar[rd]^{+} \ar[ld]_{-}&&&&&\\
&&&&&b\ar[rd]^{+} \ar[ld]_{-}&&b\ar[rd]^{+} \ar[ld]_{-}  &&b\ar[rd]^{+} \ar[ld]_{-}&&&&\\
&&&&b\ar[rd]^{+} \ar[ld]_{-}&&1\ar[rd]^{+} \ar[ld]_{-}&  &1\ar[rd]^{+} \ar[ld]_{-}&&b\ar[rd]^{+} \ar[ld]_{-}&&&\\
&&&b\ar@{.}[rd] \ar@{.}[ld]&&1\ar@{.}[rd] \ar@{.}[ld]&& 1\ar@{.}[rd] \ar@{.}[ld] &&1\ar@{.}[rd] \ar@{.}[ld]&&b\ar@{.}[rd] \ar@{.}[ld]&&\\
&&&&&&&&&&&&&
}
\]}
\caption{Legendrian mountain range for the Whitehead double $\W_m(\K)$ of the $(-13,3)$ torus knot $\K$ where $m<-78$ is odd. The maximal Thurston-Bennequin invariant is $-3$. The numbers $a$ and $b$ are determined by the formulas in Theorem~\ref{thm:twistsat}, for example when $m=-81$, $a=12$ and $b=2$.}
\label{fig:oddnegplus}
\end{figure}

\begin{figure}
\centering
{\tiny\[
\xymatrixrowsep{2em}
\xymatrixcolsep{2em}
\xymatrix{
&&&&&&&12\ar[rd]^{+} \ar[ld]_{-}  &&&&&&\\
&&&&&&6\ar[rd]^{+} \ar[ld]_{-}&  &6\ar[rd]^{+} \ar[ld]_{-}&&&&&\\
&&&&&6\ar[rd]^{+} \ar[ld]_{-}&&2\ar[rd]^{+} \ar[ld]_{-}  &&6\ar[rd]^{+} \ar[ld]_{-}&&&&\\
&&&&6\ar[rd]^{+} \ar[ld]_{-}&&2\ar[rd]^{+} \ar[ld]_{-}&  &2\ar[rd]^{+} \ar[ld]_{-}&&6\ar[rd]^{+} \ar[ld]_{-}&&&\\
&&&6\ar[rd]^{+} \ar[ld]_{-}&&2\ar[rd]^{+} \ar[ld]_{-}&& 1\ar[rd]^{+} \ar[ld]_{-} &&2\ar[rd]^{+} \ar[ld]_{-}&&6\ar[rd]^{+} \ar[ld]_{-}&&\\
&&6\ar@{.}[rd] \ar@{.}[ld]&&2\ar@{.}[rd] \ar@{.}[ld]&&1\ar@{.}[rd] \ar@{.}[ld]&  &1\ar@{.}[rd] \ar@{.}[ld]&&2\ar@{.}[rd] \ar@{.}[ld]&&6\ar@{.}[rd] \ar@{.}[ld]&\\
&&&&&&&&&&&&&
}
\]}
\caption{Legendrian mountain range for the Whitehead double $\W_m(\K)$ of the $(-13,3)$ torus knot $\K$ where $m=-80$. The maximal Thurston-Bennequin invariant is $1$.}
\label{fig:oddneg}
\end{figure}
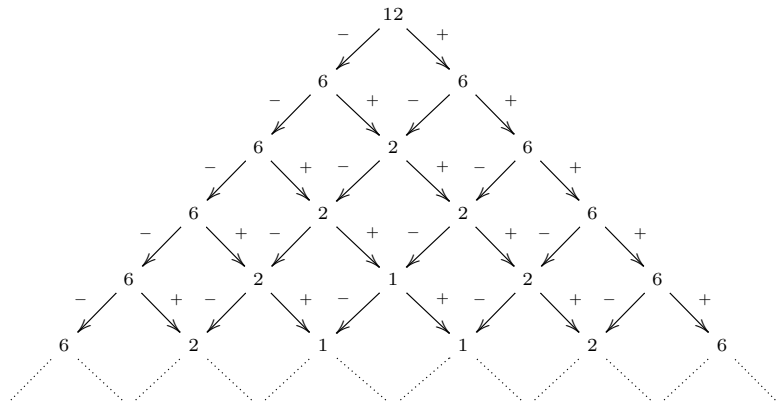

\subsection{Transverse satellites}
\label{sec:transversebraidsatellite}

We define transverse satellites with respect to Legendrian knots. Let $L\in\Lc\in\Leg(\K)$ be a Legendrian knot let $R\in\mathcal{R}\in\Trans_{V}(\Delta^{-\tb(\Lc)}\Pat)$ be a transverse pattern and let $\psi\colon(V,\xi_V)\to(\nu(L),\xi_{\textrm{st}}\vert_{\nu(L)})$ be a contactomorphism. Then the transverse knot $R(L)=\psi(R)$ is the \emph{transverse satellite with companion $L$ and pattern $R$}. The transverse satellite smoothly represents $\Pat(\K)$, and it is independent of the isotopy classes.
\begin{lemma} Let $L_0, L_1\in\Lc$, and choose standard contact neighborhoods $\nu(L_0)$ and $\nu(L_1)$. Let $R_0, R_1\in \mathcal{R}$  and suppose that $\psi_0\colon (V,\xi_V)\to (\nu(L_0),\xi\vert_{\nu(L_0)})$ and $\psi_1\colon (V,\xi_V)\to (\nu(L_1),\xi\vert_{\nu(L_1)})$ are contactomorphisms (that of course must bring the product framing $\mathbf{l}$ to the Thurston--Bennequin framings). Then $R_0(L_0)$ and $R_1(L_1)$ are transverse isotopic. \qed
\end{lemma}
This means that the transverse isotopy class $\mathcal{R}(\Lc)$ is well defined.
Using Legendrian approximations and the formulas in Lemma~\ref{lem:legbraidsat}  the self linking number of transverse satellites can be computed as follows.
\begin{lemma} Let $\Lc\in\Leg(\K)$ and $\Q\in\Trans_{V}(\Delta^{\tb(\Lc)}\Pat)$. Then
\[\self(\mathcal{R}(\Lc))=(n^2\tb(\Lc)-n\cdot\rot(\Lc))-\relsl_{V}(\mathcal{R})\]
\end{lemma}
Similarly to the Legendrian case we have the following. 
\begin{lemma}\label{lem:utptrsatelite}
Suppose $\K$ is uniformly thick, and let $\Pat$ be a pattern in $V$. Then any element of  $\Trans(\Pat(\K))$
can be written as $\mathcal{R}(\Lc)$, where $\Lc\in\Leg(\K)$ and $\mathcal{R}\in\Trans_{V}(\Delta^{-\overline{\tb}(\K)}\Pat)$. \qed
\end{lemma}
\subsubsection{Transverse braid satellites}
Braid patterns behave nicely under the satellite operation.
\begin{theorem}\label{thm:univthicktransversebraid}
Let $\K$ be a uniformly thick transversally simple knot type and suppose that $\Pat$ is a braid pattern,  then $\Pat(\K)$ is transversally simple. 
\end{theorem}
\begin{proof}
Let $\Lc_0$ be the Legendrian representation of $\K$ with the smallest rotation number amongst the ones with maximal Thurston-Bennequin number. By Lemma \ref{lem:utptrsatelite} any transverse representation $T$ of $\Pat(\K)$ can be written in the form  $R(L)$, where $L$ is a representative of a peak $\Lc$  for the mountain range for $\K$ and $R$ is a transverse represntative of $\Delta^{-\overline{t}}\Pat$.  
Notice that by Lemma \ref{lem:deltastab} for $a>0$ we have $\mathcal{R}(\Lc)=(\Delta^a\mathcal{R})(\St_-^a\Lc)$.
Since $\K$ is transversally simple for $a$ big enough $\St_-^a(\Lc)$ will be the unique Legendrian representative of $\K$ with $(\tb,\rot)=(\tb(\Lc)-a,\rot(\Lc)-a)$. By choosing sufficiently large $a$ we can also assume, that $\St_-^a\Lc$ is a stabilization of $\Lc_0$. This means that for some representative $L_0$ of $\Lc_0$ we have $T\subset \nu(\St_-^aL)\subset \nu(L_0)$, and thus $T=R_0(L_0)$ for some transverse pattern $R_0$ smoothly representing
 $\Delta^{-\overline{t}}\Pat$. We have just proved that any transverse representative of $\Pat(\K)$ can be written in the form $\mathcal{R}_0(\mathcal{L}_0)$. Since $\Pat$ is transversely simple by Theorem~\ref{thm:transversed2s1} all these representatives are distinguished by their self-linking number 
\[\self(\mathcal{R}_0(\mathcal{L}_0))=(n^2\tb(\Lc_0)-n\cdot\rot(\Lc_0))-\relsl_{V}(\mathcal{R}_0).\] Thus $\Pat(\K)$ is indeed transversally simple.
\end{proof}
\subsubsection{Transverse Whitehead doubbles} 
Recall that there is a map $\Leg(\K)$ to $\Trans(\K)$ obtained by transverse push-off and if $\Leg(\K)$ is modded out by negative stabilization then the map becomes a bijection, \cite{EtnyreHonda01b, FuchsTabachnikov97}. Thus a direct consequence of Theorem~\ref{thm:twistsat} is the classification of transverse Whitehead doubles of uniformly thick, Legendrian simple knot types. 

\begin{theorem} \label{thm:twistsattr}
Let $\W_m$ be the Whitehead pattern with $m$ half-twists. And
let $\K$ be a uniformly thick and Legendrian simple knot type for which $-\K=\K$ and such that $\W_m(\K)$ has no topological symmetries (as in Theorem~\ref{lem:incommonstab}). Suppose the maximum Thurston-Bennequin number of $\K$ is denoted $\tt$ and 
\[
\Leg(\K;\tt)=\{\Lc_0,\Lc_1 \dots,\Lc_k\}
\]
where  
\[
\rot(\Lc_0)=r_0<\cdots<\rot(\Lc_k)=r_k.
\]

Let $j$ be the maximal depth of a valley  in the Legendrian mountain range of $\K$ and let $n_d$ be the number of valleys of depth $d$ corresponding to Legendrian knots with negative rotation numbers, so the total number of valleys is $k=2(n_1+\ldots + n_j)+\mathcal{X}_{2\nmid k}$. Here, again $\mathcal{X}_{\text{event}}$ is the indicator function, with value 1 if ``event'' is true and value 0 if ``event'' is false.

All transverse knots in$\Trans(\W_m(\K))$ destabilize to one of the maximal self-linking number representatives. 

\begin{enumerate}
\item If $m\geq 2\tt$ even, then the maximal self-linking number of $\W_m(\K)$ is $2\tt-m+1$. Moreover for any $a\geq 0$ we have 
\[
\left|\Trans(\W_m(\K); 1-(m-\tt)-2a)\right|= \left\lceil \frac{k+1}2 \right \rceil -\sum_{d=1}^an_d.
\] 
\item  If $m>2\tt$ odd, then the maximal self-linking number of $\W_m(\K)$ is $2\tt-m-2$. Moreover for any $a\geq 0$ we have 
\[
\left|\Trans(\W_m(\K); -(m-\tt)-2-2a)\right|=\left\lceil \frac{k+1}2 \right \rceil -\sum_{d=1}^an_d.
\]
\item If  $m<2\tt $ odd,  then the maximal self-linking number of $\W_m(\K)$ is $-3$. Let $l=\left\lfloor\frac{2\tt-m}2\right\rfloor$. Then for any $a\geq 0$ we have
\[
\left|\Trans(\W_m(\K); -3-2a)\right| = \left\lceil \frac{k+1}2 \right \rceil -\sum_{d=1}^ln_d- \sum_{d=l+1}^an_d,
\]
where  the last sum is 0 unless $a>l$.
\item If  $m<2\tt $ even, then the maximal self-linking number of $\W_m(\K)$ is $1$. Let $l=\frac{2\tt-m}2$. There are
\[
\left\lceil\frac{k+1}2 \right\rceil \left(\frac{2\tt-m}{2}+1\right) - \sum_{d=1}^{l} n_d \left(\frac{2\tt-m-2d}2 +1\right) \]
elements in $\Trans(\W_m(\K); 1)$ and for any $a\geq 1$ we have
\[
\left|\Trans(\W_m(\K); 1-2a)\right|=\left\lceil\frac{k+1}2 \right\rceil -\sum_{d=1}^ln_d-
\sum_{d=l+1}^an_d
\]
where the last sum is 0 unless $a>l$.
\end{enumerate}
\end{theorem}

\def\cprime{$'$} \def\cprime{$'$}
\providecommand{\bysame}{\leavevmode\hbox to3em{\hrulefill}\thinspace}
\providecommand{\MR}{\relax\ifhmode\unskip\space\fi MR }
\providecommand{\MRhref}[2]{%
  \href{http://www.ams.org/mathscinet-getitem?mr=#1}{#2}
}
\providecommand{\href}[2]{#2}

\end{document}